\documentclass[11pt]{amsart}
\usepackage{amsfonts}
\usepackage{amssymb}
\usepackage{amsmath}
\usepackage{amsthm}
\usepackage{euscript}
\usepackage{graphics}
\usepackage{verbatim}
\usepackage{color}
\usepackage{hyperref}

\usepackage[margin=1.05in]{geometry}

\newcommand{\E}{\mathbb E}

\newcommand{\W}{\mathbb W^\theta_N}

\newcommand{\Wb}{\mathbb W_N}

\newcommand{\D}{\mathcal D^2}

\newcommand{\Dr}{\mathcal D}

\newcommand{\B}{\mathcal B}

\newcommand{\z}{\mathbf z}
\newcommand{\w}{\mathbf w}

\renewcommand{\P}{{\mathbb P_N}}

\DeclareMathOperator{\arccot}{arccot}

\newcommand{\Pp}{{\mathbb P}}

\newcommand{\Pq}{{\mathbb P_N^q}}

\newcommand{\m}{\mathfrak m}

\newcommand{\M}{\mathbb M_N}

\renewcommand{\t}{\mathbf t}
\renewcommand{\v}{\mathbf v}

\newcommand{\eps}{\varepsilon}

\newtheorem{theorem}{Theorem}[section]
\newtheorem{proposition}[theorem]{Proposition}
\newtheorem{lemma}[theorem]{Lemma}
\newtheorem{corollary}[theorem]{Corollary}

\theoremstyle{definition}
\newtheorem{definition}[theorem]{Definition}
\newtheorem{assumption}{Assumption}

\theoremstyle{remark}
\newtheorem{remark}[theorem]{Remark}

\renewcommand{\Re}{\mathbf{Re}}

\renewcommand{\i}{\mathbf i}

   \title{Gaussian asymptotics of discrete $\beta$--ensembles}

   \author{Alexei Borodin}

\address[Alexei Borodin]{Department of Mathematics, Massachusetts Institute of Technology, Cambridge, MA, USA, and
 Institute for Information Transmission Problems of Russian Academy of Sciences, Moscow, Russia. E-mail: borodin@math.mit.edu}

\author{Vadim Gorin}

\address[Vadim Gorin]{Department of Mathematics, Massachusetts Institute of Technology, Cambridge, MA, USA, and
 Institute for Information Transmission Problems of Russian Academy of Sciences, Moscow, Russia. E-mail: vadicgor@gmail.com}

\author{Alice Guionnet}

\address[Alice Guionnet]{Department of Mathematics, Massachusetts Institute of Technology, Cambridge, MA, USA. E-mail: guionnet@math.mit.edu}

\begin{document}
\maketitle

\begin{abstract}
 We introduce and study stochastic $N$--particle ensembles which are
 discretizations for general--$\beta$ log-gases of random matrix theory.
 The examples include random tilings, families of
non--intersecting paths,  $(z,w)$--measures, etc.
  We prove that under technical assumptions on general analytic potential,
  the global fluctuations for such ensembles are asymptotically
 Gaussian as $N\to\infty$. The covariance is universal and coincides with its counterpart in
 random matrix theory.

 Our main tool is an appropriate discrete version of the Schwinger--Dyson (or loop) equations, which
 originates in the work of Nekrasov and his collaborators.
\end{abstract}



\tableofcontents

\section{Introduction}

\subsection{Continuous log--gases}

A general--$\beta$ log-gas is a probability distribution on $N$--tuples of reals
$x_1< x_2<\dots< x_N$ with density proportional to
\begin{equation}
\label{eq_log_gas} \prod_{1\le i \le j\le N} (x_j-x_i)^{\beta} \prod_{i=1}^N \exp(-N V(x_i)),
\end{equation}
where $V(x)$ is a continuous function called \emph{potential}. For $V(x)=x^2$ and
$\beta=1,2,4$, the density \eqref{eq_log_gas} describes the joint distribution of
the eigenvalues of random matrices from Gaussian Orthogonal/Unitary/Symplectic
Ensemble; much more general potentials $V(x)$ are widespread and extensively studied
in random matrix theory and beyond, see the books \cite{Meh}, \cite{For},
\cite{AGZ}, \cite{ABF}, \cite{PS}.

Under weak assumptions on the potential, the ensembles \eqref{eq_log_gas} exhibit a
Law of Large Numbers as $N\to\infty$, which means that the (random) empirical
measure $\mu_N$ defined via
$$
 \mu_N=\frac{1}{N} \sum_{i=1}^{N} \delta_{x_i}
$$
converges (weakly, in probability) to a non-random \emph{equilibrium measure} $\mu$.
For $\beta=1,2,4$ and $V(x)=x^2$, this statement dates back to the work of Wigner
\cite{Wigner}, and $\mu$ is known in this case as the Wigner \emph{semicircle law}.
The results for generic $V(x)$ were established much later, see
 \cite{BPS}, \cite{BeGu}, \cite{Johansson_CLT}.

The next order asymptotics asks about \emph{global fluctuations}, i.e.\ how the functional
$\mu_N-\mu$ behaves as $N\to\infty$. A natural approach here is to take (sufficiently smooth)
functions $f(x)$ and consider the asymptotic behavior of random variables
\begin{equation}
\label{eq_global_fluctuations}
  N\left( \int f(x) \mu_N(dx) -\int f(x) \mu(dx) \right), \quad  N\to\infty.
\end{equation}
For quite general  potentials $V(x)$ the limits of \eqref{eq_global_fluctuations} are Gaussian with
\emph{universal} covariance depending only on the support of the equilibrium measure $\mu$. In the
breakthrough paper \cite{Johansson_CLT} Johansson proved such a statement for general $\beta>0$ and
wide class of potentials under the assumption that $\mu$ has a single interval of support. Further
developments have led to establishing such results for \emph{generic} analytic potentials, see
\cite{KS}, \cite{BoG1}, \cite{Shch}, \cite{BoG2}. Note that when the support of $\mu$ has several
intervals one needs to be careful as an additional \emph{discrete} component might appear. This
does not happen if one deterministically fixes the \emph{filling fractions}, which are the numbers
of particles in each interval of the support. In the one-interval case the limiting covariance can
be identified with that of a $1d$ section of the two--dimensional Gaussian Free Field, see
\cite{B_GFF}, \cite{BG_GFF} for the details.

While for certain specific choices of potentials $V(x)$ as well as for special
values $\beta=1,2,4$ there are several different methods for establishing central
limit theorems for global fluctuations, all the developments for generic $\beta$ and
$V(x)$ rely on the analysis of \emph{loop equations} (also known as Schwinger--Dyson
equations). These are equations for certain observables of the log--gases
\eqref{eq_log_gas}, which originated and have been widely used in physics
literature, cf.\ \cite{Mi}, \cite{AM}, \cite{Ey}, \cite{CE} and references therein.
More precisely, the observables of interest are the Stieltjes transforms of the
empirical measure $\mu_N$ and the Schwinger--Dyson equations involve their
cumulants. \emph{Formally assuming} that these cumulants have a converging  $N^{-1}$
expansion as $N\to\infty$, one can use the equations to derive recursively the
asymptotics of the cumulants starting from the equilibrium measure. These recursive
relations are sometimes called the \emph{topological recursion} because, at least
when $\beta=2$, they mimic the relations between maps of different genus, see e.g.
\cite{EO} and allow to recover the result
 from 't Hooft and Br\'ezin-Itzykson-Parisi-Zuber \cite{BIPZ} showing that matrix integrals can be seen as  generating functions of
 maps; for the general $\beta$ recursion see \cite{CE}.

Loop equations  were introduced to the mathematical community by Johansson in
\cite{Johansson_CLT} to derive the Gaussian behavior; his results were significantly
extended in the later work \cite{KS}, \cite{BoG1}, \cite{Shch}, \cite{BoG2}. The
loop equations, or rather ``their spirit'', also have further applications far
beyond the central limit theorems for global fluctuations, e.g.\ they were used in
the recent work on \emph{local universality} for random matrices, see \cite{BEY},
\cite{BFG}. We would like to emphasize an important distinction between the
approaches of physics and mathematics literature: the latter operates not with
formal expansions, but with converging asymptotic expansions. This requires \emph{a
priori} estimates, whose derivations rely on a set of tools different from the loop
equations themselves.

The Schwinger-Dyson equations for general--$\beta$ log--gases are obtained by integration by parts
and derivation with respect to the potential. In fact, they can be derived in a much more general
continuous  setting, for instance when the underlying measure is the Haar measure on a compact Lie
group, see e.g.\ \cite[(5.4.29)]{AGZ}, and further used to derive the topological asymptotic
expansions of related matrix models, see \cite{CGM}, \cite{GN}. For a recent application of the
Schwinger-Dyson equations to the lattice gauge theory see \cite{Chat}.

\subsection{Integrable discretization of log--gases}

The \emph{discrete} versions of the distribution \eqref{eq_log_gas} (with $x_i$'s living on a
lattice) at $\beta=2$ arise in numerous problems of $2d$ statistical mechanics. Examples include
random tilings (cf.\ Figure \ref{Fig_tilings} and \cite{CLP}, \cite{Johansson_paths}, \cite{Gor},
\cite{BKMM}), stochastic systems of non--intersecting paths (cf.\ \cite{KOR}, \cite{BBDT}), last
passage percolation (cf.\ \cite{Johansson_Annals}), interacting particle systems (cf.\
\cite{Johansson_shape}, \cite{BF}).

\begin{figure}[t]
\scalebox{1.0}{\includegraphics{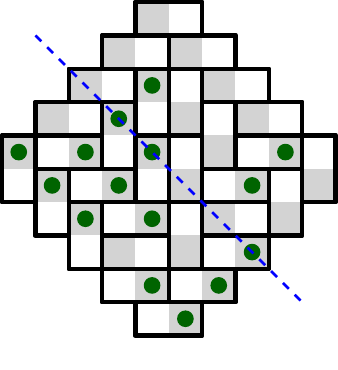}}\quad \quad \quad \quad
\scalebox{0.6}{\includegraphics{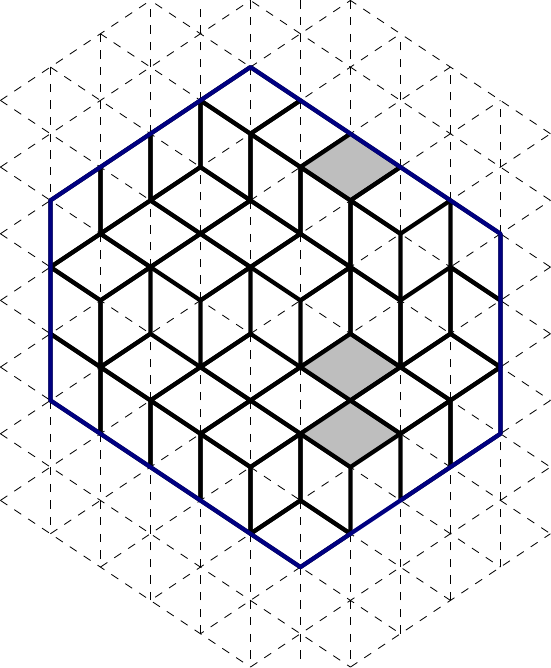}}
 \caption{The discrete particle systems arising in uniformly random domino tilings of the Aztec
 diamond (left panel) and uniformly random lozenge tilings of a hexagon (right panel). The
 distributions of 3 particles on both pictures have the form \eqref{eq_log_gas} with $N=3$, $\beta=2$ and
 suitable (different) potentials $V(x)$.
\label{Fig_tilings}}
\end{figure}

The law of large numbers for such systems (even for general values of $\beta$) can
be established by essentially the same methods as for the continuous ones, cf.\
\cite{Johansson_shape}, \cite{Johansson_paths}, \cite{Feral}. However, the situation
is drastically different for the study of the asymptotics for global fluctuations.
While for some very specific \emph{integrable} choices of the potential at $\beta=2$
the central limit theorem was proven (cf.\ \cite{BF}, \cite{Petrov_GFF}, \cite{BD}),
 until now no general approach that would work for generic $\beta$ and $V(x)$
existed. We note however the works \cite[Chapter 12]{Hora}, \cite{DFer}, \cite{Moll}
where the global Gaussian asymptotics was proven for certain (very concrete)
discrete general $\beta$ probabilistic models originating in the representation
theory. Moreover, it was not clear whether the CLT in the case of general potential
should be the same as in the continuous case, or there should be some significant
differences (indeed, for instance the \emph{local} limits in the bulk must be
different). The main technical difficulty lied in the absence of a nice
generalization of the loop equations to the discrete setting.

In \cite{Ey2}, \cite{Ey3}, Eynard proposed to interpret laws of random partitions,
which include discrete $\beta=2$ analogues of \eqref{eq_intro_distribution}, as
matrix models.  One of the key points in his interpretation is to view discrete sums
as highly oscillatory continuous integrals. Based on this identification, he
conjectured that the asymptotic expansions in this case are described by the same
Schwinger-Dyson equations with initial step given by the equilibrium measure of the
model, see e.g \cite[Section 2.4, Section 2.7.1]{Ey2}. In particular, the
fluctuations in the context of the central limit theorem should be universal.
However, so far this approach has not yet progressed much beyond the predictions.

\medskip

The central goal of this article is to study global fluctuations of the empirical measure for
discrete analogues of general-$\beta$ log--gases. One outcome is that indeed these fluctuations are
universal and described by the same covariance as for their continuous counterparts. Our analysis
is based on appropriate discrete versions of the Schwinger-Dyson equations, which (unlike in the
continuous setting) do not appear as a direct consequence of integration by parts or perturbative
arguments.

The search for the discrete loop equations starts with identifying a good discrete analogue of the
general $\beta$ distribution \eqref{eq_log_gas}. For that we fix a parameter $\theta>0$ and a
positive real--valued function $w(x;N)$\footnote{$w(x;N)$ should decay at least as
$|x|^{-(2N-2)\theta-1- const}$ with $const>0$ as $|x|\to\infty$.}. Consider the probability
distribution
\begin{equation}
\label{eq_intro_distribution}
 \P(\ell_1,\dots,\ell_N)=\frac{1}{Z_N} \prod_{1\le i<j \le N}
 \frac{\Gamma(\ell_j-\ell_i+1)\Gamma(\ell_j-\ell_i+\theta)}{\Gamma(\ell_j-\ell_i)\Gamma(\ell_j-\ell_i+1-\theta)}
 \prod_{i=1}^N w(\ell_i;N)
\end{equation}
on ordered $N$--tuples $\ell_1<\ell_2<\dots<\ell_N$, such that
$\ell_i=\lambda_i+\theta i$ and $\lambda_1\le \lambda_2\dots\le \lambda_N$ are
integers. We refer to $\ell_i$'s as the positions of $N$ particles. Let us note that
if $\theta\ne 1$, then $\ell_i$ do not sit on the fixed lattice.

Note that if $\theta=1$, then \eqref{eq_intro_distribution} has the same form as
\eqref{eq_log_gas} with $\beta=2$. Similarly, $\theta=1/2$ leads to
\eqref{eq_log_gas} with $\beta=1$\footnote{However, observe that the lattice where
particles sit at $\theta=1/2$ is not $\mathbb Z$.}.
 More generally, if we set $\ell_i = L x_i$, then
as $L\to\infty$, the ratio of Gamma-functions in \eqref{eq_intro_distribution}
behaves as $(\ell_j-\ell_i)^{2\theta}$ and mimics \eqref{eq_log_gas} with
$\beta=2\theta$.

The most important reason to view \eqref{eq_intro_distribution} as a correct \emph{integrable}
discretization of the continuous log--gas \eqref{eq_log_gas} is the following observation which is
the starting point of the results in the present paper.

\begin{theorem}[Nekrasov's equation] \label{Theorem_discrete_loop_intro}
 Consider the probability distribution \eqref{eq_intro_distribution}, and assume that
$$
 \frac{w(x;N)}{w(x-1;N)} = \frac{\phi^+_N(x)}{\phi^-_N(x)}
$$
 and define
\begin{equation} \label{eq_observable_intro}
R_N(\xi)=\phi^-_N(\xi)\cdot \E_{\P}\left[\prod_{i=1}^N\left(1-\frac{\theta}{\xi-\ell_i}\right)
\right]
+\phi^+_N(\xi)\cdot
\E_{\P}\left[\prod_{i=1}^N\left(1+\frac{\theta}{\xi-\ell_i-1}\right)\right].
\end{equation}
If $\phi^\pm_N(\xi)$ are holomorphic in a domain $\mathcal M_N \subset\mathbb C$, then so is
$R_N(\xi)$. Moreover, if $\phi^\pm_N(\xi)$ are polynomials of degree at most $d$, then so is
$R_N(\xi)$.
\end{theorem}
Theorem \ref{Theorem_discrete_loop_intro} is essentially due to Nekrasov and his collaborators, as
it is a variant of similar statements in \cite{Nekrasov}, \cite{Nek_Pes}, \cite{Nek_PS}. For the
proof see Theorem \ref{Theorem_discrete_loop} below.
\medskip

It is reasonable to ask how one could guess the form of the measure \eqref{eq_intro_distribution},
which would lead to Theorem \ref{Theorem_discrete_loop_intro}. The integrability properties of such
measures, including the product of the ratios of Gamma functions as in
\eqref{eq_intro_distribution}, can be traced to their connections to representation theory and
symmetric functions. The same products of Gamma functions appear in the evaluation formulas for
Jack symmetric polynomials (cf.\ \cite[Chapter VI, Section 10]{Mac}) and in problems of asymptotic
representation theory (cf.\ \cite{Olsh_hyper}). Another trace of integrability is the existence of
\emph{discrete Selberg integrals}, which are evaluation formulas for the partition function $Z_N$
in \eqref{eq_intro_distribution} for special choices of the weight $w(x;N)$. Examples of such
evaluations can be found in \cite[Section 2]{Olsh_hyper}, \cite[Section 2.2]{GS}. The latter
reference also explains the degeneration to the conventional Selberg integral, which is the
computation of the normalization constants for continuous log--gases \eqref{eq_log_gas} with
specific choices of the potential $V(x)$, see \cite[Chapter 17]{Meh}, \cite[Chapter 4]{For}.
 On the
other hand, for a ``naive'' discrete version, when one takes the same formula \eqref{eq_log_gas} as
a definition of a discrete log--gas, we are not aware of the existence of similar evaluation
formulas (outside $\beta=2$, when \eqref{eq_log_gas} and \eqref{eq_intro_distribution} coincide).

\subsection{Global fluctuations of discrete log--gases: one cut case}
We proceed to our results on global fluctuations for the distributions
\eqref{eq_intro_distribution} as $N\to\infty$. For that we need to postulate how the weight
$w(x;N)$ changes as $N\to\infty$. Our methods work for a variety of possibilities here, but to
simplify the exposition we stick to the following assumption in this section:
\begin{equation} \label{eq_intro_weight}
 w(x;N)=\exp\left (-N V \left(\frac{x}{N} \right) \right),
\end{equation}
where $V(z)$ is an analytic function of real argument $z$ such that for $|z|$ large
enough, $V(z)$ is monotone and satisfies
\begin{equation} \label{eq_intro_decay_of_weight}
V(z)> c\ln(|z|),\quad \text{ where } c>2\theta.
\end{equation}

Our first result is the law of large numbers for the empirical measures $\mu_N$ defined via
$$
 \mu_N=\frac{1}{N} \sum_{i=1}^N \delta_{\ell_i/N},\quad \quad  (\ell_1,\dots,\ell_N)\text{ is }
 \P-\text{distributed.}
$$
\begin{theorem} \label{theorem_LLN_intro}
 There exists a deterministic absolutely continuous compactly supported probability measure $\mu(x)dx$ with $0\le \mu(x)\le \theta^{-1}$,
 such that $\mu_N$ converges to $\mu(x)dx$ as $N\to\infty$, in the sense that for any compactly
 supported Lipshitz function $f(x)$ the following convergence in probability holds:
 $$
  \lim_{N\to\infty} \int_{\mathbb R} f(x) \mu_N(dx) =\int_{\mathbb R} f(x) \mu(x) dx.
 $$
\end{theorem}
 In fact, we prove a more general statement where, in particular, $V(x)$ does not have to be
 analytic, see Theorems \ref{Theorem_LLN} and \ref{Theorem_LDP_support} below. The measure $\mu(x)dx$ is
  the equilibrium measure, and it can be found as a solution to a variational problem. At
 $\theta=1$, Theorem \ref{Theorem_LLN} reduces to results of \cite{Johansson_shape}, \cite{Johansson_paths},
 \cite{Feral}. For general values of $\theta$, additional arguments are required, and we present them.
\smallskip

Note the condition $\mu(x)\le\theta^{-1}$, which is not present in the continuous log--gases, and
arises from the fact that the minimal distance between adjacent particles in
\eqref{eq_intro_distribution} is at least $\theta$. This is a specific feature of the discrete
models.

\smallskip

At this moment we need to make certain assumptions on the equilibrium measure. A
\emph{band}\footnote{We follow the terminology from \cite{BKMM}.} of $\mu(x)$ is a maximal interval
$(\alpha,\beta)$ such that $0<\mu(x)<\theta^{-1}$ on $(\alpha,\beta)$. From the random matrix
literature (cf.\ \cite{BDE}, \cite{Shch}, \cite{BoG2}) one expects that the global fluctuations are
qualitatively different depending on whether $\mu(x)$ has one or more bands. Here we stick to the
\emph{one band} case.

\begin{theorem} \label{Theorem_CLT_intro}
 Assume \eqref{eq_intro_weight}, \eqref{eq_intro_decay_of_weight}, that $\mu(x)$ has a unique band
 $(a_-,a_+)$, and that (technical) Assumption \ref{Assumption_simple} from Section \ref{Section_list_of_assumptions} holds. Take any $m\ge 1$ bounded
 analytic functions $f_1,\dots,f_m$ on $\mathbb R$. Then the $m$ random variables
$$
{\mathcal L}_{f_j}= \sum_{i=1}^{N} \bigl( f_j(\ell_i)-\E_{\P}f_j(\ell_i) \bigr),\quad \quad
(\ell_1,\dots,\ell_N)
 \text{ is } \P\text{--distributed,}
$$
converge (in distribution and in the sense of moments) to centered Gaussian random variables with
explicit covariance depending on $a_\pm$ and $\theta$, and given in
\eqref{eq_linear_covariance_convex} below. In particular, if $f_i(x)=(z_i-x)^{-1}$, $z_i\in\mathbb
C\setminus \mathbb R$, $i=1,\dots,m$, then
\begin{equation} \label{eq_intro_cov}
 \lim_{N\to\infty} \E_{\P} {\mathcal L}_{f_i} {\mathcal L}_{f_j}
=-\frac{1}{2\theta(z_i-z_j)^2} \left(1-\frac{z_i z_j
-\frac{1}{2}(a_-+a_+)(z_+z_j)+a_+ a_-}
 {\sqrt{(z_i-a_-)(z_i-a_+)}\sqrt{(z_j-a_-)(z_j-a_+)}}\right),
\end{equation}
\end{theorem}
\begin{remark} We prove below in Theorem \ref{Theorem_LDP_support} that with exponentially high probability all particles
$\ell_i$ are inside an interval $[-DN,DN]$. Further, it is enough to assume in Theorem
\ref{Theorem_CLT_intro} that $f_i$ are analytic only in $[-D,D]$. Moreover, we believe (but do not
prove) that the analyticity assumption can be replaced by sufficient smoothness.
\end{remark}
\begin{remark} The second order asymptotic expansion of the \emph{mean} $\E_{\P}  \left(\sum_{i=1}^{N} f_j(\ell_i) \right)$ can be also analyzed, cf.\ Theorem \ref{Theorem_main_general} with $k=1$, $m=0$.
\end{remark}
The proof of Theorem \ref{Theorem_CLT_intro} is a combination of Theorem \ref{Theorem_CLT} and
Theorem \ref{Theorem_LDP_support} in the main text.

The technical Assumption \ref{Assumption_simple}  from Section \ref{Section_list_of_assumptions} is
a statement that a certain function produced from the equilibrium measure $\mu(x)dx$ has no zeros.
In Section \ref{Section_convex} we show that this assumption always holds when $V(x)$ is a convex
function. A somewhat similar assumption appears in the work of Johansson \cite{Johansson_CLT} on
the central limit theorem for global fluctuations of continuous log--gases.

Let us emphasize that the limiting covariance in Theorem \ref{Theorem_CLT_intro} depends only on
the support of the equilibrium measure, but it is not sensitive to other features. The same
phenomenon is known in the random matrix setting. Moreover, comparing \eqref{eq_intro_cov} with
expressions in \cite[Theorem
 4.2]{Johansson_CLT}, \cite[Chapter 3]{PS}, we conclude that the covariance is \emph{precisely the
 same} as for the continuous log--gases. Thus, the discreteness of the model is invisible on the
 level of the central limit theorem.

The covariance in Theorem \ref{Theorem_CLT_intro} can be related to that of a section of the 2d
Gaussian Free Field in the upper half--plane with Dirichlet boundary conditions. One way to predict
that is by noticing that sections in lozenge tilings models with several specific boundary
conditions yield distributions of the form \eqref{eq_intro_distribution} with $\theta=1$, see
Section \ref{Section_hex} for one example. On the other hand, there exist several results on the
appearance of the Gaussian Free Field in the asymptotics for global fluctuations in lozenge
tilings, cf.\ \cite{Kenyon}, \cite{BF}, \cite{Petrov_GFF}, \cite{BuGo2}.

Subsequent work \cite{KS}, \cite{BoG1} for continuous log--gases showed that Johansson's technical
assumption on the equilibrium measure can be relaxed and replaced by certain weaker assumptions,
which hold for \emph{generic} analytic potentials. We hope that a similar thing can also be done in
the present discrete setting, but this would require further investigations.

\subsection{Weight supported on several finite intervals} In several applications the weight
$w(x;N)$ is not defined on the whole real line, but instead it is supported by a union of several
disjoint intervals $(a_i(N),b_i(N))$, $i=1,\dots,k$. In other words, the particles $\ell_i$ are now
confined to the union of these intervals $\bigcup_{i=1}^N (a_i(N),b_i(N))$. For instance, this
happens in tilings models, see Section \ref{Section_hex}.

At this point we have an additional choice: One could either \emph{fix} the filling fractions
$n_i(N)$, which are the numbers of particles $\ell_i$ in each of the intervals $(a_i(N),b_i(N))$,
 or not. In the present paper we stick to the former and fix $n_i(N)$.

We refer to Sections \ref{Section_setup} and \ref{Section_LLN} below for details of the exact
assumptions that we impose on the weight $w(x;N)$, intervals $(a_i(N),b_i(N))$, and filling
fractions $n_i(N)$. Here we will only briefly summarize the results obtained in such a framework.

The first result is the law of large numbers, which is the exact analogue of Theorem
\ref{theorem_LLN_intro}, see Theorem \ref{Theorem_LLN} for the details. As before, for the
asymptotics of the global fluctuations we need to have some information about the bands of the
equilibrium measure $\mu(x)dx$. We assume that there is \emph{one band per interval}.

\begin{theorem}[Theorem \ref{Theorem_CLT}] \label{Theorem_CLT_intro_multicut} Assume that all the data specifying the model satisfies Assumptions
\ref{Assumptions_basic}--\ref{Assumption_vanishing} in Section \ref{Section_list_of_assumptions}
below. In particular, the equilibrium measure $\mu(x) dx$ has $k$ bands $(\alpha_i,\beta_i)$,
$i=1,\dots,k$,
 one per interval of the support of the model.
 Take $m\ge 1$ functions
$f_1(z),\dots,f_m(z)$, which are analytic in a neighborhood of $\cup_{i=1}^k[a_i(N)/N, b_i(N)/N]$
for large $N$. Then as $N\to\infty$ the joint moments of the $m$ random variables
$$
{\mathcal L}_{f_j}= \sum_{i=1}^{N} \bigl( f_j(\ell_i)-\E_{\Pp_{N}}f_j(\ell_i) \bigr),\quad \quad
(\ell_1,\dots,\ell_N)
 \text{ is } \Pp_{N}\text{--distributed,}
$$
approximate those of centered Gaussian random variables with asymptotic covariance depending only
on $\theta, \alpha_1,\dots,\alpha_k, \beta_1,\dots,\beta_k$ and given by
\eqref{eq_linear_covariance}.
\end{theorem}
For $k=1$, the limiting covariance coincides with that of Theorem \ref{Theorem_CLT_intro}. For
$k>1$ it is no longer expressible through elementary functions, e.g.\ for $k=2$ elliptic functions
appear in the formulas. We again emphasize that the covariance depend only on the end--points of
the bands  and is not sensitive to other features of the equilibrium measure. Furthermore, the
covariance is the same as for the continuous log--gases, cf.\ \cite{Shch}, \cite{BoG2}.

Similarly to the $k=1$ case, one could try to identify the covariance with that of a canonically
defined random field. We believe that there should be a link to sections of the Gaussian Free Field
in a domain with $k-1$ holes, but we postpone this discussion till a future publication.

\subsection{Organization of the paper}

In Section \ref{Section_toy} we present our approach and results: Nekrasov's
equation, Law of Large Numbers and Central Limit Theorem --- in the simplest yet
non-trivial case when $\theta=1$ and $\phi^\pm_N$ in Theorem
\ref{Theorem_discrete_loop_intro} are linear functions. The measure $\P$ in this
case is known as the Krawtchouk orthogonal polynomial ensemble; it appeared and was
studied in numerous previous articles. In particular, the Central Limit Theorem for
global fluctuations in this specific case can be also obtained by other methods: see
\cite{CJY}, \cite{BD}, and \cite{BuGo2}\footnote{We remark that while the approaches
of these articles admit various generalizations, to the best of our knowledge, they
do not extend further in the direction of the present paper.}.

In Sections \ref{Section_setup}--\ref{Section_non_van} we explain a much more general framework, in
which the same ideas work and lead to the Law of Large Numbers and Central Limit Theorem. In these
sections $\theta$ is an arbitrary positive number, and the weight $w(x;N)$ and functions
$\phi^{\pm}_N$ are general (subject to certain technical assumptions).

In Section \ref{Section_examples} we specialize the general framework to certain specific examples,
which include lozenge tilings and $(z,w)$--measures from asymptotic representation theory. We
explain how all the technical assumptions are checked in all these examples.

Finally, Section \ref{Section_support} explains how the case of the infinite support of $w(x;N)$
can be reduced to the case of the bounded support (which is studied in Sections
\ref{Section_setup}--\ref{Section_non_van}) through large deviations estimates.

\bigskip

\bigskip

\subsection{Acknowledgements} We are very grateful to Nikita Nekrasov for numerous fruitful
discussions which led us to  Theorem \ref{Theorem_discrete_loop_intro} in its present form. We also
would like to thank Andrei Okounkov for drawing our attention to the existence of discrete loop
equations. We are grateful to two anonymous referees for their useful comments.

A.~B. was partially supported by the NSF grant DMS-1056390. V.~G.\ was partially supported by the
NSF grant DMS-1407562. A.~G.~ was partially supported  by the Simons Foundation and by the NSF
Grant DMS-1307704.

\section{Toy example: $\theta=1$ binomial weight}

\label{Section_toy}

The aim of this section is to describe our method in the simplest, yet non-trivial
case of the binomial weight and $\theta=1$. The resulting $N$--particle ensemble is
known as the \emph{Krawtchouk orthogonal polynomial ensemble}. It appeared in the
literature before in the connection with uniformly random domino tilings of the
Aztec diamond (cf.\ \cite{Johansson_paths}), last passage percolation (cf.\
\cite{Johansson_Annals}), stochastic systems of non--intersecting paths (cf.\
\cite{KOR}, \cite{BBDT}), and with representation theory of the
infinite--dimensional unitary group $U(\infty)$ (cf.\ \cite[Section 5]{BO1},
\cite[Section 4]{Bor}).

In the subsequent sections we will extend our approach to more general cases, but
the methodology and many key ideas remain the same.

\bigskip

Fix two integers $0<N\le M$ and consider the space $\Wb$ of $N$--tuples of integers
$$
 0\le \ell_1<\ell_2<\dots<\ell_N\le M.
$$
We define a probability distribution $\Pp_{N}$ on $\Wb$ through
\begin{equation}
\label{eq_Binomial_ensemble}
 \Pp_{N}(\ell_1,\dots,\ell_N)=\frac{1}{Z(N,M)} \prod_{1\le i<j \le N}
 (\ell_i-\ell_j)^2 \prod_{i=1}^{N} {M \choose \ell_i}.
\end{equation}
We remark that the partition function $Z(N,M)$ is explicitly known in this case:
$$
 Z(N,M)= 2^{N(M-N+1)} (M!)^N  \prod_{j=0}^{N-1} \frac{ j!}{(M-j)!}.
$$
 However, in a generic situation
there are typically no simple formulas for the partition function and we are not going to use its
explicit form.

Our analysis of the distribution $\Pp_{N}$ is based on the following $\theta=1$ version of Theorem
\ref{Theorem_discrete_loop}. This is essentially due to \cite{Nekrasov}, \cite{Nek_Pes},
\cite{Nek_PS} and we will call the main statement the \emph{Nekrasov's equation}. It should be
viewed as a discrete space analogue of the Schwinger--Dyson (also known as ``loop'') equations.

\begin{proposition} \label{Proposition_discrete_loop_1}
 Let $\P$ be a probability distribution on $N$--tuples $0\le \ell_1<\dots<\ell_N\le M$ such that
 $$
  \P(\ell_1,\dots,\ell_N)=\frac{1}{Z} \prod_{1\le i<j\le N} (\ell_i-\ell_j)^2
  \prod_{i=1}^N w(\ell_i),
 $$
 where
 $$
 \frac{w(x)}{w(x-1)} = \frac{\phi^+_N(x)}{\phi^-_N(x)},
$$
 and $\phi^{\pm}_N(x)$ are analytic functions in a complex neighborhood of
 $[-1,M+1]$ which are
 positive on $[0,M]$ and satisfy $\phi^+_N(M+1)=\phi^-_N(0)=0$.
 Define
\begin{equation} \label{eq_expectation_Nekrasov_1}
R_N(\xi)=\phi^-_N(\xi)\cdot
\E_{\P}\left[\prod_{i=1}^N\left(1-\frac{1}{\xi-\ell_i}\right)
\right]+\phi^+_N(\xi)\cdot
\E_{\P}\left[\prod_{i=1}^N\left(1+\frac{1}{\xi-\ell_i-1}\right)\right].
\end{equation}
Then $R_N(\xi)$ is analytic in the same complex neighborhood of $[-1,M+1]$. If $\phi^\pm_N(\xi)$
are polynomials of degree at most $d$, then $R_N(\xi)$ is also a polynomial of degree at most $d$.
\end{proposition}
\begin{proof} The possible singularities of $R_N(\xi)$ are simple poles at points
$m=0,1,2,\dots,M$. Let us compute a residue at such a point.

The expectation $\E_\P$ in \eqref{eq_expectation_Nekrasov_1} is a sum over all
possible configurations $(\ell_1,\dots,\ell_N)\in\Wb$. Such a configuration
contributes to the residue at $m$ if either $\ell_i=m$ or $\ell_i+1=m$ for some
$i=1,\dots,N$.

We separately analyze the contributions appearing from each $i=1,\dots,N$, which we
now fix. Given a particle configuration $\ell=(\ell_1,\dots,\ell_N)$, let $\ell^+$
denote the configuration with $i$th coordinate increased by $1$ and let $\ell^-$
denote the configuration with $i$th coordinate decreased by $1$. Note that, in
principle $\ell^+$ (similarly $\ell^-$) might fail to be in $\Wb$, as the
coordinates might coincide. However, in this case the formula for $\P(\ell^+)$ still
applies and gives zero\footnote{This is where we need the condition of
$\phi^+_N(M+1)=\phi^-_N(0)=0$, which translates into $w(M+1)=w(-1)=0$.}.

The contribution to the residue of $R_N$ at $m\in\{0,\dots,M\}$, arising from the
$i$th coordinate of $(\ell_1,\dots,\ell_N)$ is
\begin{multline}
\label{eq_sum_res}
 -\sum_{\ell\in\Wb\mid \ell_i=m} \phi^-_N(m) \P(\ell_1,\dots,\ell_N) \prod_{j\ne i}
 \left(1-\frac{1}{m-\ell_j}\right)\\ + \sum_{\ell\in\Wb\mid \ell_i=m-1} \phi^+_N(m)
 \P(\ell_1,\dots,\ell_N) \prod_{j\ne i} \left(1+\frac{1}{m-\ell_j-1}\right)
\end{multline}
But using the definition of $\P$ we see that
\begin{multline*}
\phi^-_N(m) \P(\ell_1,\dots\ell_{i-1},m,\ell_{i+1},\dots,\ell_N) \prod_{j\ne i}
 \left(1-\frac{1}{m-\ell_j}\right)
\\ =\phi^+_N(m)
 \P(\ell_1,\dots\ell_{i-1},m-1,\ell_{i+1},\dots,\ell_N) \prod_{j\ne i}
 \left(1+\frac{1}{m-\ell_j-1}\right).
\end{multline*}
We conclude that for each $\mu$, the terms with $\ell=\mu$ and $\ell=\mu^+$ (or
$\ell=\mu^-$ and $\ell=\mu$) in the first and second sum in \eqref{eq_sum_res}
cancel out and the total residue is zero.

For the polynomiality statement it suffices to notice that if $\phi^\pm_N(\xi)$ are polynomials of
degree at most $d$, then $R_N(\xi)$ is an entire function which grows as $O(\xi^d)$ as
$\xi\to\infty$. Hence, by Liouville's theorem $R_N(\xi)$ is a polynomial.
\end{proof}

Note that for the distribution \eqref{eq_Binomial_ensemble} the functions $\phi^\pm_N(x)$ can be
chosen to be linear and we set
\begin{equation}
 \phi^-_N(x)=\frac{x}{N},\quad \phi^+_N(x)=\frac{M+1}{N}-\frac{x}{N}.
\end{equation}
 Thus, in a sense, \eqref{eq_Binomial_ensemble} is one of the simplest possible
distribution in the framework of Proposition \ref{Proposition_discrete_loop_1}.

In this section we aim to study the asymptotics of the distributions $\Pp_{N}$ as
$N\to\infty$. The parameter $M$ will also depend on $N$. We fix $\m>1$ and set
$M=\lfloor \m N \rfloor$.

\subsection{Law of Large Numbers}

\label{Section_LLN_1}

As $N\to\infty$, a certain law of large numbers holds for the measures $\Pp_{N}$.
Let us introduce a random probability measure $\mu_N$ on $\mathbb R$ via
\begin{equation}
 \mu_N=\frac{1}{N}\sum_{i=1}^N \delta\left(\frac{\ell_i}{N}\right),\quad
 (\ell_1,\dots,\ell_N) \text{ is } \Pp_{N}
 \text{-distributed.} \label{eq_empirical_mes_1}
\end{equation}
The measure $\mu_N$ is often referred to as the \emph{empirical measure} of the
point configuration $\ell_1,\dots,\ell_N$. Note that our definitions imply the
condition $\ell_{i+1}-\ell_i\ge 1$, which shows that for any interval $[p,q]$, its
$\mu_N$--measure is bounded from above by $(q-p+1/N)$.

The idea for the proof of the law of large numbers for $\mu_N$ as $N\to\infty$ is to establish the
large deviations principle for the measures $\Pp_{N}$, which would show that the measure is
concentrated on $\ell_i$ which maximize the probability density \eqref{eq_Binomial_ensemble}. This
was done rigorously in \cite{Johansson_shape}, \cite{Johansson_paths}, \cite{Feral}, see also
Section \ref{Section_weak_estimate_1} below for some details. The explicit formula for
$\lim_{N\to\infty} \mu_N$ is then obtained as a solution of a variational problem, this solution
for our weight was found earlier in \cite[Example 4.2]{DS}. The developments of these articles are
summarized in the following proposition.

\begin{proposition} \label{Proposition_LLN_bin}
The measures $\mu_N$ converge (weakly, in
 probability) to a deterministic absolutely continuous measure $\mu_{\m}(x)dx$,
 which is called \emph{equilibrium measure}. For $\m\ge 2$, the density of the
 measure $\mu_\m(x)dx$ is
 $$
  \mu_\m(x)=\begin{cases}
  \dfrac{1}{\pi}\arccot\left(\dfrac{\m-2}{2\sqrt{\m-1-\left(x-\m/2\right)^2}}\right),&
  \left|x-\frac{\m}{2}\right|<\sqrt{\m-1},\\
  0,& \left|x-\frac{\m}{2}\right|\ge \sqrt{\m-1},
  \end{cases}
 $$
 and for $1<\m\le 2$,
 $$
  \mu_\m(x)=\begin{cases}
  \dfrac{1}{\pi}\arccot\left(\dfrac{\m-2}{2\sqrt{\m-1-\left(x-\m/2\right)^2}}\right),&
  \left|x-\frac{\m}{2}\right|<\sqrt{\m-1},\\
  1,& \frac{\m}{2}\ge \left|x -\frac{\m}{2}\right|\ge \sqrt{\m-1},\\
  0, & \left|x-\frac{\m}{2}\right|\ge \frac{\m}{2},
  \end{cases}
 $$
where $\arccot(y)$ is the inverse cotangent function.
\end{proposition}

A convenient way of working with the equilibrium measure $\mu_{\m}(x)$ is through its Stieltjes
transform $G_{\m}(z)$ defined through
\begin{equation} \label{eq_St_transform_1}
 G_\m(z)=\int_{-\infty}^{\infty} \frac{\mu_\m(x) \,dx}{z-x}.
\end{equation}
Observe that \eqref{eq_St_transform_1} makes sense for all $z$ outside the support of $\mu_\m(x)$,
and $G_\m(z)$ is holomorphic there. Further, as $z\to\infty$, we have $G_\mu(z)\sim z^{-1}$. An
explicit formula for $G_\m(z)$ can be readily extracted from Proposition \ref{Prop_Q_and_R} below.

We introduce two functions
$$
 R_\m(z)=z\exp(-G_\m(z))+(\m-z)\exp(G_\m(z)),
$$
$$
 Q_\m(z)=z\exp(-G_\m(z))-(\m-z)\exp(G_\m(z)).
$$
for $z\in\mathbb C\setminus[0,\m]$ and consider their analytic continuations.

\begin{proposition} \label{Prop_Q_and_R} For any $\m>1$,
$$
 R_\m(z)=\m-2, \quad \quad
 Q_\m(z)= 2 \sqrt{(z-\m/2)^2-(\m-1)},
$$
with the branch of the square root chosen so that $Q_\m(z)\sim 2z$ as $z\to\infty$.
\end{proposition}
\begin{proof}
 We first observe that $R_\m(z)$ is, in fact, a linear polynomial.
 Indeed, this is the
 $N\to\infty$ limit of $R_N(Nz)$ in Proposition \ref{Proposition_discrete_loop_1} for the measure
 $\Pp_{N}$. The fact that $G_\m(z)\sim z^{-1}$ as $z\to\infty$, implies
 $\lim_{z\to\infty} R_\m(z)=\m-2$, and thus $R_\m(z)=\m-2$.
 Further, observe that the definition of $Q_\m$ and $R_\m$ implies
 $$
  Q_\m(z)^2-R_\m(z)^2=-4 z(\m-z).
 $$
 Therefore,
 $$
  Q_\m(z)^2=4 z^2-4z\m+(\m-2)^2=4(z-\m/2)^2-4(\m-1).\qedhere
 $$
\end{proof}

\subsection{Second order expansion}
\label{Section_Second_order_1} Define the Stieltjes transform $G_N(z)$ of the
prelimit empirical measure \eqref{eq_empirical_mes_1} through
\begin{equation} \label{eq_St_transform_pre}
 G_N(z)=\int_{-\infty}^{\infty} \frac{1}{z-x}\, \mu_N(dx)=\frac{1}{N}
 \sum_{i=1}^N
 \frac{1}{z-\ell_i/N},\quad (\ell_1,\dots,\ell_N) \text{ is }
 \Pp_{N}\text{-distributed}.
\end{equation}

We aim to study how $G_N(z)$ approximate $G_\m(z)$ as $N\to\infty$. For that we introduce a
deformed version of the same function. Take $2k$ parameters $\t=(t_1,\dots,t_k)$,
$\v=(v_1,\dots,v_k)$ such that $v_a+t_a-y\ne 0$ for all $a=1,\dots,k$ and all $y\in[0,\m]$ and let
the deformed distribution $\Pp_{N}^{\t,\v}$ be defined through
\begin{multline}\label{eq_distr_deformed_1}
 \Pp_{N}^{\t,\v}(\ell_1,\dots,\ell_N)=\frac{1}{Z(N,M;\t,\v)}
\prod_{1\le i<j\le N} (\ell_i-\ell_j)^2 \prod_{i=1}^N \left[ {M\choose \ell_i}
\prod_{a=1}^{k}\left(1+\frac{t_a}{v_a-\ell_i/N}\right)\right].
\end{multline}
If $k=0$, then $\Pp_{N}^{\t,\v}=\P$ is the undeformed measure. In general,
$\Pp_{N}^{\t,\v}$ may be a complex--valued measure. The normalizing constant
$Z(N,M;\t,\v)$ in \eqref{eq_distr_deformed_1} is chosen so that the total mass of
$\Pp_{N}^{\t,\v}$ is $1$, i.e.\ $\sum_{\ell\in \Wb} \Pp_{N}^{\t,\v}(\ell)=1$. Let us
note that the numbers $t_a$ are always chosen small enough, which guarantees that
$Z(N,M;t,v)\neq 0$.

Observe that $\Pp_{N}^{\t,\v}$ satisfies the assumptions of Proposition
\ref{Proposition_discrete_loop_1} with
\begin{eqnarray}
\label{eq_phin_pl_1}
 \phi^+_N(x)&=&\left(\frac{M+1}{N}-\frac{x}{N}\right)\cdot\prod_{a=1}^k
 \biggl(v_a-\frac{x}{N}+\frac{1}{N}\biggr)\biggl(t_a+v_a-\frac{x}{N}\biggr),\\
\label{eq_phin_min_1}
  \phi^-_N(x)&=&\frac{x}{N}\cdot\prod_{a=1}^k
 \biggl(v_a-\frac{x}{N}\biggr)\biggl(t_a+v_a-\frac{x}{N}+\frac{1}{N}\biggr).
\end{eqnarray}
As above, we  set $M=\lfloor \m N \rfloor$, and omit it from the notations. We also define
$$
 \psi_N^+(z)=\bigl(\lfloor \m N\rfloor+1-\m N\bigr)+ \sum_{a=1}^k \frac{\m-z}{v_a-z}, \quad \quad \psi_N^-(z)= \sum_{a=1}^k
 \frac{z}{t_a+v_a-z}.
$$
Clearly,
$$
 \phi^+_N(Nz)=
 \left(\m-z+\frac{\psi_N^+(z)}{N}+O\left(\frac{1}{N^2}\right)\right) \cdot \prod_{a=1}^k\bigl(v_a-z\bigr)\bigl(t_a+v_a-z\bigr),
$$
$$
 \phi^-_N(Nz)=
 \left(z+\frac{\psi_N^-(z)}{N}+O\left(\frac{1}{N^2}\right)\right) \cdot \prod_{a=1}^k\bigl(v_a-z\bigr)\bigl(t_a+v_a-z\bigr).
$$

We further define $\mu_N^{\t,\v}$ as the empirical distribution of $\Pp_{N}^{\t,\v}$
and set
\begin{multline} \label{eq_St_transform_deformed}
 G_N(z)=
 \int_{-\infty}^{\infty} \frac{1}{z-x}\, \mu_N^{\t,\v}(dx)
 =\frac{1}{N}
 \sum_{i=1}^N
 \frac{1}{z-\ell_i/N},\quad (\ell_1,\dots,\ell_N) \text{ is }
 \Pp_{N}^{\t,\v}\text{-distributed}.
\end{multline}
Define
\begin{equation} \label{eq_Delta_GN}
 \Delta G_N(z)=N\bigl(G_N(z)-G_\m(z)\bigr).
\end{equation}
Note that we often omit the dependence on $\t,\v$ from the notations, to keep them
concise.

 The definition of the deformed measure $ \Pp_{N}^{\t,\v}$ is motivated by the
 following observation. It was used before in the related context in random matrix
 theory, cf.\ \cite{Mi}, \cite{Ey}.
\begin{lemma} \label{Lemma_cumulant} For any $k\ge 1$, the $k$th mixed derivative
\begin{equation}
\label{eq_cumulant}
 \frac{\partial^k}{\partial t_1 \partial t_2\dots \partial t_k}  \E_{\Pp_{N}^{\t,\v}} \bigl(  \Delta G_N(u)\bigr)
 \Biggr|_{t_a=0,\, 1\le a\le k}
\end{equation}
is the joint cumulant of $k+1$ random variables $N G_N (u)$, $N G_N(v_1)$, \dots, $NG_N(v_k)$ with
respect to the measure $\P$.
\end{lemma}
\begin{proof}
 We first note that the joint cumulants are invariant under addition of constants. Therefore, we can
 replace $N G_N (u)$ by $\Delta G_N(u)$ for the purpose of computing the cumulants.

 Further recall that one way to define the joint cumulant of $k+1$ bounded random variables $\xi_0,\dots,\xi_{k}$
 is through
 $$
  \frac{\partial^{k+1}}{\partial t_0 \partial t_1 \dots \partial t_k} \ln\left( \E
  \exp\left(\sum_{i=0}^{k} t_i \xi_i \right)\right)\Biggr|_{t_0=t_1=\dots=t_k=0}.
 $$
 Taking the derivative with respect to $t_0$ explicitly, we can rewrite this also as
 $$
    \frac{\partial^{k}}{\partial t_1 \dots \partial t_k} \frac{ \E
  \left(\xi_0 \exp\left(\sum_{i=1}^{k} t_i \xi_i \right) \right)}{\E
  \left(\exp\left(\sum_{i=1}^{k} t_i \xi_i \right)\right)}
  \Biggr|_{t_1=\dots=t_k=0}.
 $$
 Setting $\xi_0=N \Delta G_N (u)$, $\xi_i=N G_N(v_i)$, $i=1,\dots,k$  and observing that
$$
 \exp\bigl(t N G_N(v)\bigr)=\prod_{i=1}^N\left(1+
 \frac{t}{v-\ell_i/N}\right)+O(t^2), \quad t\to 0,
$$
 we get the desired statement.
\end{proof}

\begin{theorem} \label{Theorem_main_1} Fix $k=0,1,\dots$ and choose $k$ complex numbers
$\v=(v_1,\dots,v_k)\in (\mathbb C\setminus [0,\m])^k$. Then, as $N\to\infty$,
\begin{multline} \label{eq_first_order_1}
 \frac{\partial^k}{\partial t_1\cdots \partial t_k} \E_{\Pp_{N}^{\t,\v}} \bigl(  \Delta
 G_N(u)\bigr) \Bigr|_{t_a=0,\, 1\le a\le k}
 = o(1)+ \frac{\partial^k}{\partial t_1\cdots \partial t_k}\Biggl[ \frac{1}{2\pi\i \cdot 2
{\sqrt{(u-\m/2)^2-(\m-1)}} }\\ \times    \oint_{ \gamma_{[0,\m]}} \frac{dz}{(u-z)}
  \cdot\Biggl(\psi_N^-(z) e^{- G_\m(z)} +\psi_N^+(z) e^{
G_\m(z)} + \sqrt{(z-\m/2)^2-(\m-1)}\cdot \frac{\partial}{\partial z}G_\m(z)\Biggr)
\Biggr]_{t_a=0,\, 1\le a\le k},
\end{multline}
where $\gamma_{[0,\m]}$ is a simple positively--oriented contour enclosing the
segment $[0,\m]$ (the points $u$ and $v_1,\dots,v_k$ are outside the contour). The
remainder $o(1)$ is uniform over $u, v_1,\dots,v_k$ in compact subsets of the
unbounded component of $\mathbb C\setminus \gamma_{[0,\m]}$. The $k=0$ case is that
we take no derivatives in \eqref{eq_first_order_1}.
\end{theorem}
\begin{remark} \label{Remark_non_trivial_cases}The only cases, where the right--hand side of \eqref{eq_first_order_1} is
meaningful are $k=0,1$, since its $N\to\infty$ limit is zero for $k \ge 2$. Indeed, only $\psi^-_N$
depends on $t_a$, and moreover it is a sum of functions of single variables $t_a$, $a=1,\dots,k$;
therefore, all mixed partial derivatives vanish. However, we present Theorem \ref{Theorem_main_1}
in this way, since that's the form which appears in our proofs. For $k=1$ we will compute the limit
of \eqref{eq_first_order_1} below in Theorem \ref{Theorem_Stil_CLT}. For $k=0$ and generic $\m$ the
right--hand side of \eqref{eq_first_order_1} has no limit, as $\bigl(\lfloor \m N\rfloor+1-\m
N\bigr)$ in the definition of $\psi_N^+$ oscillates.
\end{remark}


\begin{remark}
 It might seem a bit unexpected that for $k=0$ the right--hand side of
 \eqref{eq_first_order_1} does not vanish when $\theta=1$. Indeed, in the context of
 random matrices this corresponds to $\beta=2$ case, where the mean is known to
 vanish (cf.\ \cite{Johansson_CLT}). For our model the non-zero mean can be traced
 back to two features. First, due to discreteness we can not adjust $M=\lfloor \m
 N\rfloor$ so that $\bigl(\lfloor \m N\rfloor+1-\m N\bigr)$ becomes zero for all
 $N$. Second, the logarithm of the weight, $\ln \left({M\choose x} \right)$, has a
 non-trivial $N\to\infty$ asymptotic expansion, which affects the result.
\end{remark}

Before providing a proof of Theorem \ref{Theorem_main_1}, let us give its important corollaries.
Note that if we set $k=0$ in Theorem \ref{Theorem_main_1}, then it gives the limit behavior for
$\E_{\Pp_{N}}\Delta G_N(z)$. In fact, it also gives  the Central Limit Theorem for $G_N(z)$.

\begin{theorem}\label{Theorem_Stil_CLT} The random field $N\Bigl(G_N(z)-\E_{\Pp_{N}} G_N(z)\Bigr)$, $z\in \mathbb
C \setminus [0,\m]$, converges as $N\to\infty$ (in the sense of joint moments, uniformly in $z$ in
compact subsets of $\mathbb C\setminus [0,\m]$) to a centered complex Gaussian random field with
second moment
\begin{multline} \label{eq_limit_covariance}
 \lim_{N\to\infty} N^2\Bigl( \E_{\Pp_{N}}\bigl(  G_N(u) G_N(v)\bigr) -\E_{\Pp_{N}} G_N(u)
 \,\E_{\Pp_{N}} G_N(v) \Bigr)=:{\mathcal C}(u,v)
\\ =-\frac{1}{2(u-v)^2} \left(1-\frac{uv -\frac{1}{2}(a_-+a_+)(u+v)+a_+ a_-}
 {\sqrt{(u-a_-)(u-a_+)}\sqrt{(v-a_-)(v-a_+)}}\right),
\end{multline}
where $a_\pm=\frac{\m}{2}\pm\sqrt{\m-1}$.
\end{theorem}
\begin{remark}
 Note that the covariance ${\mathcal C}(u,v)$ has no singularity at $u=v$, since the
 right--hand side of \eqref{eq_limit_covariance} has a finite $u\to v$ limit.
\end{remark}

\begin{remark}
 Since $\overline{G_N(u)}=G_N(\overline u)$, the formula \eqref{eq_limit_covariance} is sufficient
 for determining the asymptotic covariance of the random field $G_N(u)$.
\end{remark}

\begin{proof}[Proof of Theorem \ref{Theorem_Stil_CLT}]
 We start from the result of Lemma \ref{Lemma_cumulant} for the joint cumulant of $k+1$ random variables $N G_N (u)$, $N G_N(v_1)$, \dots, $NG_N(v_k)$.
In particular, if $k=1$, then we get the covariance of $N G_N (u)$ and $N G_N(v_1)$.

For $k>1$, differentiating \eqref{eq_first_order_1} we see that the result vanishes
as $N\to\infty$, see Remark \ref{Remark_non_trivial_cases}. This implies the
asymptotic Gaussianity of the random field $N\Bigl(G_N(z)-\E_{\Pp_{N}}
G_N(z)\Bigr)$.

In the case $k=1$, differentiating \eqref{eq_first_order_1} we see that the
covariance given by \eqref{eq_cumulant} is (recall that $u$ and $v_1$ lie outside
the integration contour)
\begin{multline}
 \frac{1}{2\pi\i \cdot 2
{\sqrt{(u-\m/2)^2-(\m-1)}} } \oint_{ \gamma_{[0,\m]}} \frac{dz}{(u-z)}
  \cdot\Biggl( -\frac{z}{(v_1-z)^2} e^{- G_\m(z)} \Biggr)
\\= \frac{1}{2\pi\i \cdot 4
{\sqrt{(u-\m/2)^2-(\m-1)}} } \oint_{ \gamma_{[0,\m]}} \frac{dz}{(z-u)(v_1-z)^2}
  \cdot\Biggl(R_\m(z)+Q_\m(z) \Biggr).
\end{multline}
The term with $R_\m(z)$ integrates to $0$, as it has no singularities inside $\gamma_{[0,\m]}$. The
term with $Q_\m$ is computed as the sum of the residues at $z=u$ and at $z=v_1$, which gives the
desired covariance formula.
\end{proof}
Theorem \ref{Theorem_Stil_CLT} implies the central limit theorem for general analytic linear
statistics.

\begin{corollary} \label{Corollary_CLT_linear}  Let the distribution $\P$ be given by
\eqref{eq_Binomial_ensemble} with $M=\lfloor \m N\rfloor$. Take $k$ real valued functions
$f_1(z),\dots,f_k(z)$ on $[0,\m]$, which can be extended to holomorphic functions in a complex
neighborhood $\B$ of $[0,\m]$. Then as $N\to\infty$ the $k$ random variables
$$
{\mathcal L}_{f_j}= \sum_{i=1}^{N} \bigl( f_j(\ell_i)-\E_{\Pp_{N}}f_j(\ell_i) \bigr),\quad
(\ell_1,\dots,\ell_N)
 \text{ is } \Pp_{N}\text{--distributed,}
$$
converge in the sense of moments to centered Gaussian random variables with covariance
\begin{equation}
\label{eq_linear_covariance_1}
 \lim_{N\to\infty} \E_{\Pp_{N}} {\mathcal L}_{f_i} {\mathcal L}_{f_j}= \frac{1}{(2\pi\i)^2}
 \oint_{\gamma_{[0,\m]}} \oint_{\gamma_{[0,\m]}} f_i(u) f_j(v) {\mathcal C}(u,v) du dv,
\end{equation}
where ${\gamma_{[0,\m]}}$ is a positively oriented contour in $\B$ which encloses $[0,\m]$, and
${\mathcal C}(u,v)$ is given by \eqref{eq_limit_covariance}.
\end{corollary}
\begin{remark}
 The covariance \eqref{eq_linear_covariance_1} has the same form as for random matrices and log--gases in the one cut
 regime. It depends only on the restrictions of functions $f_j$ onto the interval $[a_-,a_+]$ and can be rewritten in several other equivalent forms, cf.\ \cite[Theorem
 4.2]{Johansson_CLT}, \cite[Chapter 3]{PS}, \cite[Section 4.3.3]{AGZ}.
\end{remark}
\begin{proof}[Proof of Corollary \ref{Corollary_CLT_linear}] Observe that
$$
 {\mathcal L}_{f}=\frac{N}{2\pi\i} \oint_{\gamma_{[0,\m]}} f(z) \left( G_N(z)-\E_{\P} G_N(z)\right)
 dz.
$$
Therefore, all the moments of ${\mathcal L}_{f}$ are obtained from the centered
moments of $G_N(z)$ by integration. Since the latter converge uniformly in $z$ on
the integration contour, so do the former. It remains to use the fact that the
integrals of jointly Gaussian random variables are also Gaussian.
\end{proof}
\begin{remark}
 Similarly to Corollary \ref{Corollary_CLT_linear}, Theorem \ref{Theorem_main_1} can be used to obtain the first two terms in the $N\to\infty$ asymptotic
 expansion of $\E_{\Pp_{N}} \sum_{i=1}^{N} f(\ell_i)$ for functions $f$
 holomorphic in a neighborhood of $[0,\m]$.
\end{remark}

\medskip

The rest of this section is devoted to the proof of Theorem \ref{Theorem_main_1}.

\subsection{Heuristic argument for Theorem \ref{Theorem_main_1}}

\label{Section_heuristics_1}

In this section we present a sketch of the proof for Theorem \ref{Theorem_main_1} in which we omit
 crucial bounds on remainders in the asymptotic formulas. These bounds will be
established further on.

\smallskip

We start from the statement of Proposition \ref{Proposition_discrete_loop_1} for the measures
$\Pp_{N}^{\t,\v}$. Making the change of variables $\xi=Nz$, we have for $z$ outside $[0,\m]$
\begin{multline}\label{eq_first_term_1}
 \prod_{i=1}^{N}\left(1-\frac{1}{\xi-\ell_i} \right)=
 \exp\left(\sum_{i=1}^{N} \ln\left(1-\frac{1}{N}\cdot
 \frac{1}{z-\ell_i/N}\right)\right)
\\
 =
  \exp\left(-G_N(z)+\frac{1}{2N}\cdot \frac{\partial}{\partial z}G_N(z)+O\left(\frac1{N^2}\right) \right)
\\
 =
  \exp\left(-G_\m(z) -\frac{1}{N} \Delta G_N(z)+\frac{1}{2N}\cdot \frac{\partial}{\partial z}G_N(z)+O\left(\frac1{N^2}\right)
  \right),
\end{multline}
where the remainder is uniform over $z$ in compact subsets of $\mathbb C\setminus [0,\m]$.
Similarly,
\begin{multline} \label{eq_second_term_1}
 \prod_{i=1}^{N}\left(1+\frac{1}{\xi-\ell_i-1} \right)=
 \exp\left(\ln\left(1+\frac{1}{N}\cdot
   \frac{1}{z-\ell_i/N-1/N}\right)\right)
 \\
 =
  \exp\left( G_\m(z)+\frac{1}{N} \Delta G_N(z)-\frac{1}{2N}\cdot \frac{\partial}{\partial z}G_N(z)+O\left(\frac1{N^2}\right)
  \right).
\end{multline}

Recalling the definition of $R_\m(z)$, we conclude that the function $R_N(Nz)$ from Proposition
\ref{Proposition_discrete_loop_1} can be written in the following form
\begin{multline} \label{eq_x19}
 R_N(Nz)=\prod_{a=1}^k (v_a-z)(t_a+v_a-z) \cdot \Biggl[ R_\m(z)\\ +
 z e^{-G_\m (z)}\, \E_{\Pp_{N}^{\t,\v}} \left( \exp\left( \frac{1}{N} \Delta G_N(z) \right) -1 \right)
 +(\m-z)e^{G_\m (z)}\, \E_{\Pp_{N}^{\t,\v}} \left( \exp\left(-\frac{1}{N} \Delta G_N(z) \right) -1 \right)
\\ + \frac{\psi_N^-(z)}{N} e^{-G_\m(z)} +\frac{\psi_N^+(z)}{N} e^{
G_\m(z)}
\\+
\dfrac{z e^{-G_\m(z)}}{N} \left( \frac{1}{2} \frac{\partial}{\partial
z}G_\m(z)\right) +
 \dfrac{(\m-z) e^{G_\m(z)}}{N} \left(
-\frac{1}{2} \frac{\partial}{\partial z}G_\m(z)\right) \Biggr]
+o\left(\frac{1}{N}\right),
\end{multline}
where $\psi_N^\pm$ appeared from the second terms in $1/N$ expansion of $\phi^{\pm}_N$ from
\eqref{eq_phin_pl_1}, \eqref{eq_phin_min_1}. We further want to simplify the expression in the
second line of \eqref{eq_x19}, by replacing $e^{h}-1$ by $h$ under expectations. For that we note a
simple inequality, which we will use  with $n=2$:
\begin{equation}
\label{eq_exponent_bound}
 \left|e^h-\sum_{j=0}^{n-1} \frac{h^j}{j!}\right|\le |h|^n e^{|h|},\quad h\in\mathbb C, \quad n=1,2,\dots.
\end{equation}
We will later establish in Section \ref{Section_estimates} that uniformly in $z$ in compact subsets
of $\mathbb C\setminus[0,\m]$ we have
\begin{equation}
\label{eq_crucial_estimate_1}
 \E_{\Pp_{N}^{\t,\v}}\left( \left|\frac{1}{N} \Delta G_N(z) \right|^2 \exp\left( \left|\frac{1}{N} \Delta
 G_N(z)\right|\right) \right)=o\left(\frac1N\right), \quad N\to\infty.
\end{equation}

We therefore can rewrite \eqref{eq_x19} as
\begin{multline} \label{eq_x20}
 R_N(Nz)= \prod_{a=1}^k (v_a-z)(t_a+v_a-z) \cdot \Biggl[ R_\m(z)
 +\frac{Q_\m(z)}{N} \E_{\Pp_{N}^{\t,\v}} \left(  \Delta G_N(z) \right)
\\ + \frac{\psi_N^-(z)}{N} e^{-G_\m(z)} +\frac{\psi_N^+(z)}{N} e^{ G_\m(z)}
+ \dfrac{z e^{-G_\m(z)}}{N} \left( \frac{1}{2} \frac{\partial}{\partial z}G_\m(z)\right) +
 \dfrac{(\m-z) e^{ G_\m(z)}}{N} \left(
-\frac{1}{2} \frac{\partial}{\partial z}G_\m(z)\right)\Biggr]
+o\left(\frac{1}{N}\right),
\end{multline}
where the estimate of the remainder is uniform over $z$ in compact subsets of $\mathbb
C\setminus[0,\m]$.

Let us now fix $u$ outside the contour $\gamma_{[0,\m]}$ enclosing the interval $[0,\m]$, divide
\eqref{eq_x20} by
$$2\cdot 2\pi \i \cdot (u-z)\cdot \prod_{a=1}^k
(v_a-z)(t_a+v_a-z)$$
  and integrate over $\gamma_{[0,\m]}$. Since both $R_N(Nz)$ and
$R_\m(z)$ are holomorphic inside the contour, the integrals of the corresponding terms vanish. From
the rest we get, with the help of Proposition \ref{Prop_Q_and_R},
\begin{multline} \label{eq_x21}
 \frac{1}{2\pi\i} \oint_{ \gamma_{[0,\m]}}
 \frac{\sqrt{(z-a_-)(z-a_+)}}{u-z}
 \cdot \E_{\Pp_{N}^{\t,\v}} \left(  \Delta G_N(z)
 \right) dz \\ =- \frac{1}{2\pi\i} \oint_{\gamma_{[0,\m]}} \frac{dz}{2(u-z)}
  \cdot\Biggl(\psi_N^-(z) e^{-G_\m(z)} +\psi_N^+(z) e^{
G_\m(z)}
\\+
ze^{-G_\m(z)} \left( \frac{1}{2} \frac{\partial}{\partial z}G_\m(z)\right) +
 (\m-z) e^{G_\m(z)} \left(
 -\frac{1}{2} \frac{\partial}{\partial z}G_\m(z)\right) \biggr) + o(1).
\end{multline}
For the first line of \eqref{eq_x21}, note that $\E_{\Pp_{N}^{\t,\v}} \left(  \Delta
G_N(z) \right)$ is analytic outside the contour of integration and decays as $1/z^2$
when $z\to\infty$. Therefore, we can compute the integral as (minus) the residue at
$z=u$, which is
$$
\sqrt{(u-a_-)(u-a_+)}
 \cdot \E_{\Pp_{N}^{\t,\v}} \left(  \Delta G_N(u) \right).
$$
This leads, via differentiation in $t$'s, to the desired formula \eqref{eq_first_order_1} of
Theorem \ref{Theorem_main_1}.

Let us point out the parts of the above argument  that are not yet rigorous and
 whose justification is necessary to complete the proof of Theorem
\ref{Theorem_main_1}:
\begin{itemize}
\item We need to show that the bound \eqref{eq_crucial_estimate_1} is valid.
\item We need to prove that all the remainders remain small when we differentiate
with respect to variables $t_a$, $a=1,\dots,k$.
\end{itemize}

\subsection{A weak a priori estimate.}

\label{Section_weak_estimate_1}

The following lemma is a first step for establishing the desired estimates.

\begin{lemma} \label{Lemma_a_priory_1}
 Fix a positive integer $n$ and take a compact set $A\subset \mathbb C\setminus [0,\m]$.
 Then for every $\eps>0$ there exists a constant $C>0$ such that for every $z_1,\dots,z_n\in A$ and every $N=1,2,\dots$ we
 have
 \begin{equation}
 \label{eq_weak_estimate}
  \E_{\Pp_{N}} \left| \prod_{i=1}^n  \Delta G_N(z_i)\right| \le C \cdot N^{n(\frac{1}{2}+\eps)},
 \end{equation}
 with $\Delta G_N(z)$ defined by \eqref{eq_Delta_GN} with $k=0$.
\end{lemma}
\begin{remark}
 We will show in the next section that $C\cdot N^{n(\frac{1}{2}+\eps)}$ in the righthand side of \eqref{eq_weak_estimate} can be
 replaced by a constant. However, to produce such a sharp estimate we need to start from this
 weaker one.
\end{remark}

The proof of Lemma \ref{Lemma_a_priory_1} that we now present is similar to the argument of
\cite[Lemma 3.4.5]{BoG2}, which follows the ideas of \cite{MMS}.

Before we proceed we need to recall the characterization of the equilibrium measure
$\mu_\m$.
 Introduce the functional $I[\mu]$ of a measure $\mu$ on $[0,\m]$ via
 \begin{equation} \label{eq_functional}
  I[\mu]=\iint\limits_{\begin{smallmatrix} 0\le x,y \le \m\\ x\ne y\end{smallmatrix}} \ln|x-y| d\mu(x)d\mu(y)-\int_0^\m V(x) d\mu(x),
 \end{equation}
 where
 $$
  V(x)=x\ln(x)+(\m-x)\ln(\m-x).
 $$
 Note that if the measure $\mu$ has atoms, then it is important to exclude the diagonal in the integral
 \eqref{eq_functional}, i.e.\ integrate only over $x\ne y$.

 The variational characterization of the measure $\mu_\m(x)dx$\footnote{Throughout the paper the density of a measure $\mu$ is denoted $\mu(x)$.} (see \cite{Johansson_shape}, \cite{Johansson_paths}, \cite{Feral})
 yields that $\mu_\m(x) dx$ is the unique minimizer of $I[\mu]$
 among the probability measures of density at most $1$.

 Define the functions $F_\m(x)$ through
 \begin{equation}
 \label{eq_effective_pot}
  F_\m(x)=2\int_0^\m \ln|x-t|\mu_\m(t)dt-V(x).
 \end{equation}
 Then varying the functional $I[\mu]$ at $\mu_\m$, one proves (cf.\ \cite{DS}, \cite{Feral}, \cite{ST}) that
 there exists a real number $f$ such that $F_\m(x)-f=0$ on $S=\{0\le x\le \m \mid
 0<\mu_\m(x)<1\}$, $F_\m(x)-f<0$ on the complement of the support of $\mu_\m$, and
 $F_\m(x)-f>0$ when the density $\mu_\m(x)$ is equal to $1$.

\bigskip

Now take any two compactly supported absolutely continuous probability measures with uniformly
bounded densities $\nu(dx)=\nu(x)dx$ and $\rho(dx)=\rho(x)dx$ and define $\Dr(\nu,\rho)$ through
\begin{equation}
\label{eq_quadratic_main}
 \D(\nu,\rho)=-\int_{\mathbb R}\int_{\mathbb R}\ln|x-y| (\nu(x)-\rho(x))(\nu(y)-\rho(y))dxdy.
\end{equation}
There is an alternative formula for $\Dr(\nu(x),\rho(x))$ in terms of Fourier transforms, cf.\
\cite{BeGu}:
\begin{equation}
\label{eq_quadratic_alternative}
 \Dr(\nu,\rho)=\sqrt{\int_0^\infty \frac{1}{t} \left|\int_{\mathbb R} e^{\i tx}
 (\nu(x)-\rho(x))dx\right|^2 dt}.
\end{equation}

Fix a parameter $p>2$ and let $\tilde \mu_N$ denote the convolution of the empirical measure
$\mu_N$ given by \eqref{eq_empirical_mes_1} with uniform measure on the interval $[0,N^{-p}]$.

\begin{proposition} \label{Prop_pseudodistance_bound}
 There exists $C\in\mathbb R$ such that for all $\gamma>0$ and all $N\ge 1$ we have
 $$
  \Pp_{N}
  \Bigl( \Dr(\tilde\mu_N,\mu_\m)\ge \gamma\Bigr)\le\exp\bigl( CN\ln(N)
  -\gamma^2 N^2\bigr).
 $$
\end{proposition}
\begin{proof}
 Observe that for every $N$--tuple $0\le \ell_1<\ell_2<\dots<\ell_N\le M$ we have,
 using Stirling's formula for factorials,
\begin{equation} \label{eq_x29}
\Pp_{N}(\ell_1,\dots,\ell_N) =\dfrac{\exp\Bigl(2 N(N-1)\ln(N)+ N^2 I\bigl[ {\rm mes}[\ell_1,\dots,\ell_N] \bigr]
+O(N\ln(N))\Bigr)} {Z(N,\lfloor \m
 N\rfloor)},
\end{equation}
 where
$$
{\rm mes}[\ell_1,\dots,\ell_N]=\frac{1}{N}\sum_{i=1}^N \delta\left(\frac{\ell_i}{N}\right).
$$
Let us obtain a lower bound for the partition function $Z(N,\lfloor \m
 N\rfloor)$ in \eqref{eq_x29}. For that let $x_i$,
 $i=1,\dots,N$ be quantiles of $\mu_\m$ defined through
 $$
  \int_0^{x_i} \mu_\m(x) dx = \frac{i-1/2}{N},\quad 1\le i\le N.
 $$
 Since $\mu_\m(x)\le 1$, $x_{i+1}-x_i\ge 1/N$ and therefore, the numbers $\lfloor N
 x_i\rfloor$, $1\le i \le N$ are all distinct. We can then write
 \begin{equation} \label{eq_x54}
  Z(N,\lfloor \m
 N\rfloor)\ge
 \exp\Bigl(2N(N-1)\ln(N)+ N^2 I\bigl[ {\rm mes}[\lfloor N
 x_1\rfloor,\dots,\lfloor N
 x_N\rfloor] \bigr] +O(N\ln(N))\Bigr).
 \end{equation}
 We claim that \eqref{eq_x54} can be transformed into
 $$
  \exp\Bigl(2N(N-1)\ln(N)+ N^2 I\bigl[ \mu_\m  \bigr] +O(N \ln(N))\Bigr).
 $$
 Indeed, for the double--integral part of $I[\cdot]$ we write using the monotonicity of logarithm
 \begin{multline*}
    \sum_{i<j} \ln\left(\frac{\lfloor N
 x_j\rfloor}{N} - \frac{\lfloor N
 x_i\rfloor}{N}\right)\le
   \sum_{i<j} \ln\left(x_j - x_i +\frac{1}{N}\right)
 \\ \le
  N^2 \sum_{i<j} \int_{x_{j}}^{x_{j+1}} \int_{x_{i-1}}^{x_i} \ln\left(t - s +\frac{1}{N}\right) \mu_\m(t) \mu_\m(s)\, dt
  \,ds +O(N\ln(N))
\\= N^2 \iint_{s<t} \ln(t-s) \mu_\m(t)\mu_\m(s) \,dt\, ds +O(N\ln(N)),
\end{multline*}
and similarly for the opposite inequality. For the single--integral part of $I[\cdot]$ we have
$$
 N \sum_{i=1}^N V\left(\frac{\lfloor N
 x_i\rfloor}{N}\right)=N^2 \sum_{i=2}^{N-2} \int_{x_i}^{x_{i+1}} \biggl( V(t)+\left(t-\frac{\lfloor N
 x_i\rfloor}{N}\right)
 V'(\kappa(t)) \biggr)\mu_\m(t) dt + O(N),
$$
where $\kappa(t)$ is a point inside $[x_1,x_N]$. Observe that $|V'(\kappa(t))|=O(\ln(N))$, and thus
\begin{multline*}
  N \sum_{i=1}^N V\left(\frac{\lfloor N
 x_i\rfloor}{N}\right)
= N^2 \sum_{i=2}^{N-2} \int_{x_i}^{x_{i+1}} \biggl( V(t) +\left(x_{i+1}-x_i+\frac{1}{N}\right)
O(\ln(N)) \biggr) \mu_\m(t)dt+O(N)\\=
 N^2 \int_0^\m V(t) \mu_m(t) dt + N^2 \frac{O(\ln(N))}{N}\sum_{i=2}^{N-2}
 \left(x_{i+1}-x_i+\frac{1}{N}\right)=  N^2 \int_0^\m V(t) \mu_m(t) dt +O(N\ln(N)).
\end{multline*}

 The next step is to replace ${\rm mes}[\ell_1,\dots,\ell_N]$ in \eqref{eq_x29} by
 its convolution with the uniform measure on $[0,N^{-p}]$, that we denote $\widetilde{\rm mes}[\ell_1,\dots,\ell_N]$. For that take two
 independent random variables $u$, $\tilde u$ uniformly distributed on $[0,N^{-p}]$, where $p>2$ as above.
 Then
\begin{multline}
 I\bigl[\widetilde{\rm mes}[\ell_1,\dots,\ell_N] \bigr]=\E_{u,\tilde u} \int_0^\m\int_0^\m
 \ln|x-y+u-\tilde
 u| {\rm mes}[\ell_1,\dots,\ell_N](dx) {\rm mes}[\ell_1,\dots,\ell_N](dy)\\-\E_{u} \int_0^\m V(x+u){\rm mes}[\ell_1,\dots,\ell_N](dx)
 \\=I\bigl[ {\rm mes}[\ell_1,\dots,\ell_N] \bigr]+ \frac{1}{N}\E_{u,\tilde u} \int_0^\m \ln|u-\tilde
 u| {\rm mes}[\ell_1,\dots,\ell_N](dx)\\+\E_{u,\tilde u}\int\int_{x\ne y} \ln\left|1+\frac{u-\tilde u}{x-y}\right|{\rm mes}[\ell_1,\dots,\ell_N](dx) {\rm mes}[\ell_1,\dots,\ell_N](dy)
 \\+\E_{u} \int_0^\m (V(x+u)-V(x)){\rm mes}[\ell_1,\dots,\ell_N](dx)= I\bigl[ {\rm mes}[\ell_1,\dots,\ell_N]
 \bigr] +O\left(\frac{\ln(N)}{N}\right).
\end{multline}
We conclude that there exists a constant $C$ such that
$$
\Pp_{N}(\ell_1,\dots,\ell_N)\le \exp(C N\ln(N)) \exp\biggl(N^2 \bigl(I \bigl[\widetilde{\rm
mes}[\ell_1,\dots,\ell_N]\bigr]-I[\mu_\m]\bigr)\biggr).
$$
Further, completing the square we get
\begin{multline} \label{eq_x30}
  I\bigl[\widetilde{\rm mes}[\ell_1,\dots,\ell_N]\bigr]-I[\mu_\m]\\=-\D(\mu_\m,\widetilde{\rm
 mes}[\ell_1,\dots,\ell_N])+\int_0^\m  F_\m(x) (\widetilde{\rm
 mes}[\ell_1,\dots,\ell_N]-\mu_\m)(dx)\\
 =-\D(\mu_\m,\widetilde{\rm
 mes}[\ell_1,\dots,\ell_N])+\int_0^\m ( F_\m(x)-f) (\widetilde{\rm
 mes}[\ell_1,\dots,\ell_N]-\mu_\m)(dx),
\end{multline}
with $F_\m(x)$ and $f$ defined in \eqref{eq_effective_pot} and directly below that
formula, respectively. Let us analyze the last term in \eqref{eq_x30}. On the set
$S=\{x\mid 0<\mu_\m(x)<1\}$ the function $F_\m(x)-f$ vanishes. On the complement of
the support of $\mu_\m(x)$ we have $F_\m(x)-f <0$ and the corresponding part of the
last integral in \eqref{eq_x30} is negative. Finally, on the set $S'=\{x\mid
\mu_\m(x)=1\}$, we have $F_\m(x)-f>0$. Since all the points $\ell_i/N$ are at least
$1/N$ apart, for any $a$ such that
 $[a,a+1/N]\subset S'$ we have
$$
 (\widetilde{\rm
 mes}[\ell_1,\dots,\ell_N]-\mu_\m)\bigl([a,a+1/N]\bigr)\le 0.
$$
Thus,
\begin{multline} \label{eq_x46}
\int_a^{a+1/N} ( F_\m(x)-f) (\widetilde{\rm
 mes}[\ell_1,\dots,\ell_N]-\mu_\m)(dx)\\
 \le
 \int_a^{a+1/N} ( F_\m(x)-F_\m(a)) (\widetilde{\rm
 mes}[\ell_1,\dots,\ell_N]-\mu_\m)(dx)
 \le \frac{2}{N} \sup_{\begin{smallmatrix} x,y\in S'\\ |x-y|\le 1/N \end{smallmatrix}}
 |F_\m(x)-F_\m(y)|.
\end{multline}
Observe that the last $\sup$ is at most $O(\ln(N)/N)$. Now partition $S'$ into segments of the form
$[a,a+1/N]$ and note that for boundary segments (which have to be shorter than $1/N$) the bound for
the integral of the form \eqref{eq_x46} is still valid as $F_\m(x)$ is equal to $f$ in one of the
end-points of such segment. Summing the bounds over all segments, we get a bound on the integral
over $S'$. It follows that as $N\to\infty$
$$
 \int_0^\m ( F_\m(x)-f) (\widetilde{\rm
 mes}[\ell_1,\dots,\ell_N]-\mu_\m)(dx)\le O\left(\frac{\ln(N)}{N}\right).
$$
Therefore, we finally obtain
$$
\Pp_{N}(\ell_1,\dots,\ell_N)\le \exp(C' N\ln(N)) \exp\biggl(-N^2 \D(\widetilde{\rm
mes}[\ell_1,\dots,\ell_N],\mu_\m)\biggr).
$$
Since the total number of $N$--tuples $0\le \ell_1<\dots\ell_N\le \lfloor \m N\rfloor$ is
${{\lfloor \m N\rfloor+1}\choose N}=\exp(O(N\ln(N)))$, the proof is complete.
\end{proof}

\begin{corollary} \label{Corrollary_tail_bound}
For a compactly supported Lipschitz function $g$   define
 $$
  \| g\|_{1/2} = \left(\int_{-\infty}^{\infty} |s| \left|\int_{-\infty}^{\infty} e^{\i s x} g(x)dx
  \right|^2 ds\right)^{1/2}, \quad \| g\|_{\rm Lip}=\sup_{x\ne y}
  \left|\frac{g(x)-g(y)}{x-y}\right|.
 $$
  Fix any $p>2$. Then there exists $C\in\mathbb R$ such that for all $\gamma>0$, all $N\ge 1$ and all $g$ we have
\begin{multline} \label{eq_x38}
 \Pp_{N}\left( \left|\int_0^\m g(x) \mu_N(dx)-\int_0^\m g(x) \mu_\m(dx)\right|
 \ge \gamma \| g \|_{1/2}+ \frac{\|g\|_{\rm Lip}}{N^p} \right)\\ \le \exp\left( C N\ln(N)-\frac{\gamma^2
 N^2}{2}\right).
\end{multline}
\end{corollary}
\begin{proof} We have
\begin{multline*}
 \left|\int_0^\m g(x) \mu_N(dx)-\int_0^\m g(x) \mu_\m(dx)\right|\\ \le
 \left|\int_0^\m g(x) \mu_N(dx)-\int_0^\m g(x) \tilde \mu_N(dx)\right|
 +\left|\int_0^\m g(x) \tilde \mu_N(dx)-\int_0^\m g(x) \mu_\m(dx)\right|.
\end{multline*}
The first term is bounded by $ \frac{\|g\|_{\rm Lip}}{N^p}$ and corresponds to such term in
\eqref{eq_x38}. Therefore, it remains to work out the second term. Since scalar products are
preserved under Fourier transform, with the notation
$$
 \hat \phi (s)= \int_{-\infty}^{\infty} e^{\i s x} \phi(x)dx
$$
we write the Plancherel formula (note that $g(x)$ and $(\tilde \mu_N(x)-\mu_\m(x))$ are bounded and
belong to $L^1[0,\m]\cap L^2 [0,\m]$) and use the Cauchy--Schwartz inequality
\begin{multline*}
\left|\int_0^\m g(x)(\tilde \mu_N(x)-\mu_\m(x)) dx\right|=\left|\int_{-\infty}^{\infty}
\left(\sqrt{|t|} \hat g(t)\right) \frac{\hat {\tilde {\mu}}_N(t)-\hat \mu_\m(t)}{\sqrt{|t|}}
dt\right|
\\ \le \| g\|_{1/2}
\sqrt{\int_{-\infty}^{\infty} \frac{|\hat{\tilde \mu}_N(t)-\hat \mu_\m(t)|^2}{|t|}dt}= \sqrt{2} \|
g\|_{1/2} \Dr(\tilde \mu_N,\mu_\m).
\end{multline*}
It remains to use Proposition \ref{Prop_pseudodistance_bound}.
\end{proof}

\begin{proof}[Proof of Lemma \ref{Lemma_a_priory_1}]
Choose a small $\eta>0$ and take an infinitely differentiable function $h(x)$, whose
support is inside $[-\eta,\m+\eta]$ and such that $h(x)=1$ for $0\le x\le \m$.

 Since both $\mu_N$ and $\mu_\m$ are supported
on $[0,\m]$, we can replace $1/(z-x)$ in the definition of $G_N(z)$ and $G_\m(z)$ by
a nice smooth compactly supported function $h(x)/(z-x)$ without changing $G_N(z)$
and $G_\m(z)$. Now choose $\gamma= q\cdot N^{-1/2+\eps}$, $q>0$, in Corollary
\ref{Corrollary_tail_bound} and note that for this choice the right--hand side of
\eqref{eq_x38} still exponentially decays as $N\to\infty$. This readily implies the
bound of Lemma \ref{Lemma_a_priory_1}.
\end{proof}

\subsection{Self--improving estimates and the proof of Theorem \ref{Theorem_main_1}}

\label{Section_estimates}

The aim of this section is to finish the proof of Theorem \ref{Theorem_main_1}. The key part is
establishing the following statement.

\begin{proposition} \label{Prop_uniform_bound}
 For any $k\ge 1$ and any $u_1,\dots,u_k\in \mathbb C\setminus [0,\m]$, the joint moments
$$
 \E_{\P} \prod_{a=1}^{k} \left| \Delta G_N(u_a) \right|
$$
are uniformly (in $N$ and in $u_1,\dots,u_k$ in compact subsets of $\mathbb C\setminus [0,\m]$)
bounded.
\end{proposition}

The idea for getting such estimates is to start from Lemma \ref{Lemma_a_priory_1} and then
recursively feed the existing estimates into the argument of Section \ref{Section_heuristics_1} to
obtain the stronger ones. This is very similar to the argument of \cite[Section 4.3]{BoG1}.

\begin{proof}[Proof of Proposition \ref{Prop_uniform_bound}]

Fix $n=0,1,\dots$ and $\v=(v_1,\dots,v_n)$. Arguing as in Lemma \ref{Lemma_cumulant}, we prove
 that for any bounded random variable $\xi$ we have
\begin{equation}
\label{eq_moment_via_derivative}
 \frac{\partial^n}{\partial t_1\cdots \partial t_n}\left( \E_{\Pp_{N}^{\t;\v}} \xi\right) \biggr|_{t_1=\dots=t_n=0}=
 M_c\bigl(\xi, NG_N(v_1),\dots, N G_N(v_n) \bigr),
\end{equation}
where $M_c$ is the joint cumulant. Since the cumulants are unchanged under shifts, we can also
replace $N G_N(v_a)$ by $\Delta G_N(v_a)$ in \eqref{eq_moment_via_derivative}.

We aim to differentiate the formulas of Section \ref{Section_heuristics_1} with respect to $t_a$ at
$t_a=0$. For that we need to examine the remainders. The remainder $o(N^{-1})$ in \eqref{eq_x19}
comes from three sources: from the expansions \eqref{eq_first_term_1}, \eqref{eq_second_term_1};
 from the $N^{-1}$ expansion of $\phi^\pm_N(z)$; from the replacement of $\frac{\partial}{\partial z} G_N(z)$ by $\frac{\partial}{\partial z} G_\m(z)$.
 The remainder can then be written as a sum corresponding to these three sources:
\begin{equation}
\label{eq_remainder_1} \prod_{a=1}^n (v_a-z)(t_a+v_a-z)\left[ \frac{1}{N^2} \E_{\Pp_{N}^{\t;\v}}
\xi_N(z) + \frac{1}{N^2} {\mathfrak c}(z;\t,\v) \E_{\Pp_{N}^{\t;\v}}
 \xi_N'(z)+
 \frac{1}{N^2}  \E_{\Pp_{N}^{\t;\v}}  \left( \xi_N''(z)  \frac{\partial}{\partial z} \Delta G_N(z) \right)\right]
\end{equation}
 where $\xi_N(z)$, $\xi_N'(z)$, $\xi_N''(z)$ are random variables, which are bounded uniformly in $N$, in $z$
 belonging to compact subsets of $\mathbb C\setminus [0,\m]$, and which do not depend on $\t,\v$. The function ${\mathfrak c}(z;\t,\v)$
 arises from the large $N$ expansions of \eqref{eq_phin_pl_1}, \eqref{eq_phin_min_1}, and it is
 uniformly bounded in $z$ belonging to compact subsets of $\mathbb C\setminus\{v_1,\dots, v_n; t_1+v_1;\dots
 t_n+v_n\}$. The dependence on $z$ is holomorphic in all the terms.

 Further, when we pass from \eqref{eq_x19} to \eqref{eq_x20}, we need to expand $\E_{\Pp_{N}} \left( \exp\left( \frac{1}{N} \Delta G_N(z) \right) -1
 \right)$. Since $\frac{1}{N} \Delta G_N(z)$ is a bounded random variable, we can use the Taylor expansion to get
\begin{multline}
\label{eq_x23}
 \E_{\Pp_{N}^{\t;\v}} \left( \exp\left( \frac{1}{N} \Delta G_N(z) \right) -1 \right)=\E_{\Pp_{N}^{\t,\v}}\left(\frac{1}{N} \Delta
 G_N(z)\right)+\frac{1}{2}\,\E_{\Pp_{N}^{\t,\v}}\left(\frac{1}{N} \Delta G_N(z)\right)^2\\+\frac{1}{6}\,\E_{\Pp_{N}^{\t,\v}}\left(\frac{1}{N} \Delta
 G_N(z)\right)^3+\dots
\end{multline}
Thus, after reconstructing the remainders, \eqref{eq_x21} is replaced by

\begin{multline} \label{eq_x26}
 \frac{1}{4\pi\i} \oint_{ \gamma_{[0,\m]}}
 \frac{1}{u-z}
 \cdot\Biggl(  Q_\m(z) \E_{\Pp_{N}^{\t,\v}} \left(  \Delta G_N(z)\right)
 +\frac{R_\m(z)}{2N} \E_{\Pp_{N}^{\t,\v}} \left(  \Delta G_N(z)^2\right)+\frac{Q_\m(z)}{6 N^2} \E_{\Pp_{N}^{\t,\v}} \left(  \Delta
 G_N(z)^3
 \right) +\dots
\\ + \frac{1}{N} \left[\E_{\Pp_{N}^{\t;\v}} \xi_N(z) +  {\mathfrak c}(z;\t,\v) \E_{\Pp_{N}^{\t;\v}}
 \xi_N'(z)+
  \E_{\Pp_{N}^{\t;\v}}  \left( \xi_N''(z)  \frac{\partial}{\partial z} \Delta G_N(z) \right)) \right]
 \Biggr) dz \\ =- \frac{1}{2\pi\i} \oint_{\gamma_{[0,\m]}} \frac{dz}{2(u-z)}
  \cdot\Biggl(\psi_N^-(z) e^{-G_\m(z)} +\psi_N^+(z) e^{
G_\m(z)} + \frac{Q_\m(z)}{2}\frac{\partial}{\partial z}G_\m(z)  \biggr).
\end{multline}
We now differentiate \eqref{eq_x26} with respect to all $t_a$ at $t_a=0$. For all the terms
involving random variables we use \eqref{eq_moment_via_derivative} to rewrite the result as a
cumulant. Note that when we differentiate ${\mathfrak c}(z;\t,\v) \E_{\Pp_{N}^{\t;\v}}  \xi_N''(z)
$, we need to apply the Leibnitz rule and therefore to differentiate each of the two factors and
get a sum. We also compute the integral of $\frac{1}{u-z}
  Q_\m(z) \E_{\Pp_{N}^{\t,\v}} \left(  \Delta G_N(z)\right)$  as minus the residue at $z=u$ to get
\begin{multline} \label{eq_x28}
\frac{Q_\m(u)}{2} M_c\bigl( \Delta G_N(u), \Delta G_N(v_1),\dots, \Delta G_N(v_n)\bigr)+
 \\ \frac{1}{4\pi\i} \oint_{ \gamma_{[0,\m]}}
 \frac{1}{u-z} \Biggl(\frac{R_\m(z)}{4N} M_c\bigl( (\Delta
G_N(z))^2, \Delta G_N(v_1),\dots, \Delta G_N(v_n)\bigr)
 \\+\frac{Q_\m(z)}{12 N^2} M_c\bigl( (\Delta
G_N(z))^3, \Delta G_N(v_1),\dots, \Delta G_N(v_n)\bigr) +\dots \Biggr) dz
\\
+\frac{1}{N} \cdot \frac{1}{4\pi\i} \oint_{ \gamma_{[0,\m]}}
 \frac{1}{u-z}  M_c\Bigl(\xi_N(z)+\xi_N''(z)  \frac{\partial}{\partial z} \Delta G_N(z) , \Delta G_N(v_1),\dots, \Delta
 G_N(v_n)\Bigr)dz
\\
+ \frac{1}{N} \cdot \frac{1}{4\pi\i} \sum_{A\subset\{1,\dots,n\}} \oint_{ \gamma_{[0,\m]}}
 \frac{1}{u-z} {\mathfrak c}_A(z;\t,\v) M_c
 (\xi_N'(z), \Delta G_N(v_a), a\in A) dz
\\=
 \frac{\partial^n}{\partial t_1\cdots \partial t_n}\Biggl[ \frac{1}{4\pi\i} \oint_{ \gamma_{[0,\m]}} \frac{dz}{(u-z)}
  \cdot\Biggl(\psi_N^-(z) e^{- G_\m(z)} +\psi_N^+(z) e^{
G_\m(z)} +\frac{Q_\m(z)}{2}\cdot \frac{\partial}{\partial z}G_\m(z)\Biggr) \Biggr]_{t_a=0,\, 1\le
a\le n},
\end{multline}
where the fifth line is the result of the differentiation of ${\mathfrak c}(z;\t,\v)
\E_{\Pp_{N}^{\t;\v}}  \xi_N''(z) $, i.e.\ ${\mathfrak c}_A(z;\t,\v)$ is the mixed derivative of
${\mathfrak c}(z;\t,\v)$ with respect to $t_a$, $a\in \{1,\dots,n\}\setminus A$. Let us analyze the
resulting expression \eqref{eq_x28} as $N\to\infty$. At this moment all we need from the
right--hand side of \eqref{eq_x28} is that it is $O(1)$ as $N\to\infty$.

Consider the infinite sum over growing powers of $\Delta G_N(z)$ in \eqref{eq_x28}. All the terms
starting from the $H$th one can be combined into
\begin{equation}
\label{eq_x34}
 M_c\left( \sum_{h=H}^\infty (Q/R)_h \frac{(\Delta G_N(z))^h}{2 h! N^{h-1}}, \Delta G_N(v_1),\dots, \Delta G_N(v_n)
 \right),
\end{equation}
where $(Q/R)_h$ is either $Q_\m(z)$ or $R_\m(z)$ depending on the parity of $h$. Expanding the
cumulants in terms of centered moments and using Holder's inequality, we observe that if $\zeta$ is
a bounded random variable, then for any fixed $k\ge 1$ there is a constant $C_k$ such that
\begin{equation}\label{eq_cum_mom}
\left| M_c(\zeta, \Delta G_N(v_1),\ldots, \Delta G_N(v_k))\right|\le C_k \, \sqrt{\E|\zeta|^2
}\prod_{i=1}^k \bigl[ \E |\Delta G_N(v_i)- \E[\Delta G_N(v_i)]|^{2k}\bigr]^{1/2k}
\end{equation}
Observe that the random variables $G_N(v_a)$ are uniformly bounded,  same is true about $(Q/R)_h$
(for $z$ and $v_a$ in compact subsets of $\mathbb C\setminus [0,\m]$), and
$$\zeta= \sum_{h=H}^\infty (Q/R)_h \frac{(\Delta G_N(z))^h}{2 h! N^{h-1}}$$
satisfies (uniformly over $z$ in compact subsets of $\mathbb C\setminus [0,\m]$)  the point--wise
bound
$$|\zeta|\le {\rm const} \cdot N \cdot \left|\frac{\Delta G_N(z)}{N}\right|^H \exp\left |\frac{\Delta G_N(z)}{N}\right|.$$
Thus, we can bound \eqref{eq_x34} by
\begin{multline} \label{eq_x39}
 {\rm const} \cdot  N^{n+1} \sup_z \sqrt{\E_{\P} \Biggl( \left|\frac{\Delta G_N(z)}{N}\right|^{2H} \exp\left| \frac{2 \Delta
 G_N(z)}{N}\right|
 \Biggr)}
 \le
   {\rm const}  \cdot N^{n+1}   \sup_z \sqrt{\E_{\P} \Biggl( \left|\frac{\Delta G_N(z)}{N}\right|^{2H}
 \Biggr)},
\end{multline}
where in the last inequality we used the uniform boundness of $\frac{\Delta G_N(u)}{N}$, and
$\sup_z$ is to be taken over any set which includes both $\gamma_{[0,\m]}$ and all points $v_a$,
$a=1,\dots,n$.

 We can use
Lemma \ref{Lemma_a_priory_1} to bound the expectation in \eqref{eq_x39} and conclude
that if $H>3n$, then \eqref{eq_x34} is $o(1)$. Therefore, renaming $u$ into $v_0$,
\eqref{eq_x28} is rewritten as
\begin{multline} \label{eq_x35}
\frac{Q_\m(v_0)}{2} M_c\bigl( \Delta G_N(v_0), \Delta G_N(v_1),\dots, \Delta G_N(v_n)\bigr) +
 \\ + \frac{1}{4\pi\i} \oint_{ \gamma_{[0,\m]}}
 \frac{1}{v_0-z} \left( \sum_{h=2}^{3n}  \frac{1}{N^{h-1}} \cdot \frac{(Q/R)_h}{2 h! } M_c\bigl( (
\Delta G_N(z))^h, \Delta G_N(v_1),\dots, \Delta G_N(v_n)\bigr) \right)dz \\ +
\frac{1}{N} \cdot
\frac{1}{4\pi\i} \oint_{ \gamma_{[0,\m]}}
 \frac{1}{v_0-z}  M_c\Bigl(\xi_N(z)+\xi_N''(z)  \frac{\partial}{\partial z} \Delta G_N(z) , \Delta G_N(v_1),\dots, \Delta
 G_N(v_n)\Bigr)dz
\\
+ \frac{1}{N} \cdot \frac{1}{4\pi\i} \sum_{A\subset\{1,\dots,n\}} \oint_{ \gamma_{[0,\m]}}
 \frac{1}{v_0-z} {\mathfrak c}_A(z;\t,\v) M_c
 (\xi_N'(z), \Delta G_N(v_a), a\in A) dz
 = O(1),
\end{multline}
where $(Q/R)_h$ is either $Q_\m(z)$ or $R_\m(z)$ depending on the parity of $h$, and we replaced
$u$ by $v_0$.

At this moment we claim that \eqref{eq_x35} for all $n=0,1,2\dots$ together with the
bound of Lemma \ref{Lemma_a_priory_1} implies Proposition \ref{Prop_uniform_bound}.
Indeed, take any two \emph{disjoint} compact sets $\mathcal U,\mathcal V \subset
\mathbb C\setminus[0,\m]$, which are invariant under conjugation, and suppose that
$\gamma_{[0,\m]}\subset\mathcal U$ . Expanding cumulants in terms of centered
moments and using Lemma \ref{Lemma_a_priory_1}, we obtain
\begin{multline} \label{eq_bound_x1}
\sup_{v_0,\dots,v_n\in \mathcal V}
\frac{1}{N} \cdot \frac{1}{4\pi\i} \oint_{
\gamma_{[0,\m]}}
 \frac{1}{v_0-z}  M_c\Bigl(\xi_N(z)+\xi_N''(z)  \frac{\partial}{\partial z} \Delta G_N(z) , \Delta G_N(v_1),\dots, \Delta
 G_N(v_n)\Bigr)dz\\ =O\left(N^{(n+1)(1/2+\eps)-1}\right),
\end{multline}
\begin{multline} \label{eq_bound_x2}
\sup_{v_0,\dots,v_n\in \mathcal V}
 \frac{1}{N} \cdot \frac{1}{4\pi\i} \sum_{A\subset\{1,\dots,n\}} \oint_{
\gamma_{[0,\m]}}
 \frac{1}{v_0-z} {\mathfrak c}_A(z;\t,\v) M_c
 (\xi_N'(z), \Delta G_N(v_a), a\in A) dz
 \\=O\left(N^{n(1/2+\eps)-1}\right),
\end{multline}
\begin{multline} \label{eq_bound_x3}
 \sup_{v_0,\dots,v_n\in \mathcal V}
 \frac{1}{4\pi\i} \oint_{ \gamma_{[0,\m]}}
 \frac{1}{v_0-z} \left( \frac{1}{N^{h-1}} \cdot \frac{(Q/R)_h}{2 h! } M_c\bigl( (
\Delta G_N(z))^h, \Delta G_N(v_1),\dots, \Delta G_N(v_n)\bigr) \right)dz
\\=O\left(N^{(h+n)(1/2+\eps)+1-h}\right).
\end{multline}
One might be cautious about \eqref{eq_bound_x1}, as it involves the derivative of
$\Delta G_N(z)$ (instead of $\Delta G_N(z)$ itself as in Lemma
\ref{Lemma_a_priory_1}), yet for analytic functions a uniform bound for a function
implies a uniform bound for its derivative, therefore, the bound is valid.

 We plug
 \eqref{eq_bound_x1}, \eqref{eq_bound_x2}, \eqref{eq_bound_x3} into \eqref{eq_x35} to get
\begin{multline}
\label{eq_iterative_bound}
 \sup_{v_0,\dots,v_n\in \mathcal V} \frac{Q_\m(v_0)}{2} M_c\bigl( \Delta G_N(v_0), \Delta G_N(v_1),\dots, \Delta
G_N(v_n)\bigr) \\ + \left[\sum_{h=2}^{3n}  O\left(N^{(h+n)(1/2+\eps)+1-h}\right)
\right]+ O\left(N^{(n+1)(1/2+\eps)-1}\right)+O\left(N^{n(1/2+\eps)-1}\right)
 = O(1),
\end{multline}
which implies that for each $n=0,1,2\dots$
\begin{equation}
\label{eq_x36}
 \sup_{v_0,v_1,\dots,v_{n}\in \mathcal V}\left|M_c\left(\Delta G_N(v_0),\dots, \Delta G_N(v_{n-1})\right)\right| = O(
 N^{(n+1)(1/2+\eps)-(1/2-\eps)})+O(1),
\end{equation}
where the remainder depends on the choice of the compact set $\mathcal V\subset
\mathbb C \setminus [0,\m]$. As centered moments are linear combinations of products
of joint cumulants, we deduce from \eqref{eq_x36} that for all $k=1,2,\dots$,
\begin{equation}
\label{eq_x40} \sup_{v_1,\dots,v_k\in \mathcal V} \E\left[\prod_{a=1}^k (\Delta
G_N(v_a)-\E[\Delta G_N(v_a)])\right]=O( N^{k(1/2+\eps)-(1/2-\eps)})+O(1).
\end{equation}
Combining $n=0$ version of \eqref{eq_x36} with \eqref{eq_x40} we finally conclude
that
\begin{equation}
\label{eq_x41} \sup_{v_1,\dots,v_k\in \mathcal V} \E  \left[ \prod_{a=1}^k \Delta
G_N(v_a) \right]=O\left( N^{k(1/2+\eps)-(1/2-\eps)}\right)+O(1), \quad k=1,2,\dots.
\end{equation}
If $k$ is even, then we can choose $v_1=\dots=v_{k/2}=v$ and
$v_{k/2+1}=\dots=v_k=\bar v$ in \eqref{eq_x41} and get
\begin{equation}
\label{eq_x41_5} \sup_{v\in \mathcal V} \E  \left[ |\Delta G_N(v)|^{k}
\right]=O\left( N^{k(1/2+\eps)-(1/2-\eps)}\right)+O(1), \quad k=2,4,6,\dots.
\end{equation}
The estimate for odd $k$ is reduced to the even $k$ case with the use of Jensen's
inequality in the form
$$
\E( \xi^k)\le \left(\E\left(\xi^{k+1}\right)\right)^{\frac{k}{k+1}}, \quad k>0,
$$
which leads to the following bound for all $k$:
\begin{equation}
\label{eq_x41_6} \sup_{v\in \mathcal V} \E  \left[ |\Delta G_N(v)|^{k}
\right]=O\left( N^{k(1/2+\eps)-\frac{k}{k+1}(1/2-\eps)}\right)+O(1), \quad
k=1,2,3,\dots.
\end{equation}
We finally use the Holder inequality to get
\begin{multline}
\label{eq_x43} \sup_{v_1,\dots,v_k\in \mathcal V} \E\left[ \prod_{a=1}^k |\Delta
G_N(v_a)| \right]\le  \sup_{v\in \mathcal V} \E\left[ |\Delta G_N(v)|^k
\right]=O\left( N^{k(1/2+\eps)-\frac{k}{k+1}(1/2-\eps)}\right) +O(1) \\=O\left(
N^{k(1/2+\eps)-1/6}\right) +O(1) , \quad k=1,2,3,\dots,
\end{multline}
where we silently assumed that $\eps>0$ is small enough for the last equality to
hold.

At this moment we can iterate the argument. Expanding cumulants in terms of moments,
we deduce from \eqref{eq_x43} the following three bounds:
\begin{multline} \label{eq_bound_x1_2}
\sup_{v_0,\dots,v_n\in \mathcal V} \frac{1}{N} \cdot \frac{1}{4\pi\i} \oint_{
\gamma_{[0,\m]}}
 \frac{1}{v_0-z}  M_c\Bigl(\xi_N(z)+\xi_N''(z)  \frac{\partial}{\partial z} \Delta G_N(z) , \Delta G_N(v_1),\dots, \Delta
 G_N(v_n)\Bigr)dz\\ =O\left(N^{(n+1)(1/2+\eps)-1-1/6}\right),
\end{multline}
\begin{multline} \label{eq_bound_x2_2}
\sup_{v_0,\dots,v_n\in \mathcal V}
 \frac{1}{N} \cdot \frac{1}{4\pi\i} \sum_{A\subset\{1,\dots,n\}} \oint_{
\gamma_{[0,\m]}}
 \frac{1}{v_0-z} {\mathfrak c}_A(z;\t,\v) M_c
 (\xi_N'(z), \Delta G_N(v_a), a\in A) dz
 \\=O\left(N^{n(1/2+\eps)-1-1/6}\right),
\end{multline}
\begin{multline} \label{eq_bound_x3_2}
 \sup_{v_0,\dots,v_n\in \mathcal V}
 \frac{1}{4\pi\i} \oint_{ \gamma_{[0,\m]}}
 \frac{1}{v_0-z} \left( \frac{1}{N^{h-1}} \cdot \frac{(Q/R)_h}{2 h! } M_c\bigl( (
\Delta G_N(z))^h, \Delta G_N(v_1),\dots, \Delta G_N(v_n)\bigr) \right)dz
\\=O\left(N^{(h+n)(1/2+\eps)+1-h-1/6}\right).
\end{multline}
The only difference between \eqref{eq_bound_x1}, \eqref{eq_bound_x2},
\eqref{eq_bound_x3} and \eqref{eq_bound_x1_2}, \eqref{eq_bound_x2_2},
\eqref{eq_bound_x3_2} is that the degree of $N$ in the bound decreased by $1/6$
(which is because the bound of Lemma \ref{Lemma_a_priory_1} is replaced by
\eqref{eq_x43}).

 We then plug \eqref{eq_bound_x1_2}, \eqref{eq_bound_x2_2},
\eqref{eq_bound_x3_2} into \eqref{eq_x35} to get
\begin{multline}
\label{eq_iterative_bound_2}
 \sup_{v_0,\dots,v_n\in \mathcal V} \frac{Q_\m(v_0)}{2} M_c\bigl( \Delta G_N(v_0), \Delta G_N(v_1),\dots, \Delta
G_N(v_n)\bigr) \\ + \left[\sum_{h=2}^{3n}  O\left(N^{(h+n)(1/2+\eps)+1-h-1/6}\right)
\right]+ O\left(N^{(n+1)(1/2+\eps)-1-1/6}\right)+O\left(N^{n(1/2+\eps)-1-1/6}\right)
 = O(1).
\end{multline}
In the same way as \eqref{eq_iterative_bound} implied \eqref{eq_x43}, the bound
\eqref{eq_iterative_bound_2} implies
\begin{multline}
\label{eq_x43_cop} \sup_{v_1,\dots,v_k\in \mathcal V} \E\left[ \prod_{a=1}^k |\Delta
G_N(v_a)| \right]=O\left( N^{k(1/2+\eps)-\frac{k}{k+1}(1/2-\eps)-1/6}\right) +O(1)
\\=O\left( N^{k(1/2+\eps)-2\cdot 1/6}\right) +O(1) , \quad k=1,2,3,\dots.
\end{multline}
Repeating the same argument $m-2$ more times, we improve \eqref{eq_x43_cop} to
\begin{equation}
\label{eq_final_bound} \sup_{v_1,\dots,v_k\in \mathcal V} \E\left[ \prod_{a=1}^k
|\Delta G_N(v_a)| \right]=O\left( N^{k(1/2+\eps)-m/6}\right) +O(1) , \quad
k=1,2,3,\dots.
\end{equation}
Since $m$ is arbitrary, this implies
\begin{equation}
\sup_{v_1,\dots,v_k\in A} \E  \left[ \prod_{a=1}^k |\Delta G_N(v_a)| \right]=O(1), \quad
k=1,2,\dots,
\end{equation}
and finishes the proof of Proposition \ref{Prop_uniform_bound}. \end{proof}

\begin{proof}[Proof of Theorem \ref{Theorem_main_1}]

Take  \eqref{eq_x28} and observe that the bounds of Section \ref{Section_estimates} imply that all
the terms except for the second line and the last line are negligible as $N\to\infty$. The second
line of \eqref{eq_x28} is precisely the left--hand side of \eqref{eq_first_order_1}, while the last
line of \eqref{eq_x28} is the right--hand side of \eqref{eq_first_order_1}.
\end{proof}

\section{General setup}

\label{Section_setup}

\subsection{Definition of the system}

\label{Section_definition}

Our next goal is to generalize the arguments of the previous section to a much more general setting
of a multi--cut fixed filling fractions model with fixed parameter $\theta>0$ and general weight
$w(x)$. Informally, we want to consider measures of the form
$$
 \prod_{1\le i<j \le N}
 \frac{\Gamma(\ell_j-\ell_i+1)\Gamma(\ell_j-\ell_i+\theta)}{\Gamma(\ell_j-\ell_i)\Gamma(\ell_j-\ell_i+1-\theta)}
 \prod_{i=1}^N w(\ell_i;N)
$$
on ordered $N$--tuples $\ell_1<\ell_2<\dots<\ell_N$ referred to as positions of $N$
particles and satisfying two additional constraints. First, the particles are
separated into $k$ groups, and particles in each group must belong to its own
interval of the real line. Second, if $i$th and $(i+1)$st particles are in the same
group, then $\ell_{i+1}-\ell_i\in \{\theta,\theta+1,\theta+2,\dots\}$.

For instance, if we have a single group $(k=1)$, then after defining $\lambda_i$
through $\ell_i=\lambda_i+\theta i$, the constraint boils to down to saying that all
$\lambda_i$ are integers and they satisfy
$\lambda_1\ge\lambda_2\ge\dots\ge\lambda_N$.

\bigskip

Let us give precise definitions for general $k$. The model depends on an integer
parameter $N=1,2,\dots$ and amounts to fixing for each $N$ a probability
distribution on certain $N$--point subsets of $\mathbb R$.

 We fix an integer
$k=1,2,\dots$, whose meaning is the number of segments in the support of the measure. For each
$N=1,2,\dots$ we take $k$ integers $n_1(N),\dots,n_k(N)$, such that $\sum_{i=1}^k n_i(N)=N$ and $k$
disjoint intervals $(a_1(N),b_1(N))$, \dots, $(a_k(N),b_k(N))$ of the real line ordered from left
to right.\footnote{For a generalization to the case of infinite support see Sections
\ref{Section_convex}, \ref{Section_zw-measures}.}

 We assume that $b_i(N)+\theta\le a_{i+1}(N)$ for $i=1,\dots,k-1$. The numbers $a_i(N)$,
$b_i(N)$ must also satisfy the conditions
\begin{equation}
\label{eq_intervals_conditions}
 b_i(N)-\theta n_i(N)-a_i(N)\in\mathbb Z.
\end{equation}
Further, the number $n_i(N)$ counts the number of the particles in the $i$th interval; to make this
statement precise we define the sets of indices $I_j\subset\{1,\dots,N\}$, $j=1,\dots,k$, via
$$
 I_j=\left\{i\in\mathbb Z \, \Bigl| \ \sum_{m=1}^{j-1} n_m(N) < i \le \sum_{m=1}^j n_m(N)\right\}.
$$
We also set $I_j^+$ and $I_j^-$ to be the maximal and minimal elements of $I_j$,
respectively.

\begin{definition} \label{Def_state_space}The state space $\W$ consists of $N$--tuples $\ell_1<\ell_2<\dots<\ell_N$
such that for each $j=1,\dots,k$:
\begin{enumerate}
 \item If $i=I^-_j$, then $\ell_i-a_i(N)\in \mathbb Z_{>0}$.
 \item If $i=I^+_j$, then $b_i(N)-\ell_i\in\mathbb Z_{>0}$.
 \item If $i\in I_j$, but $i\ne I^+_j$, then $\ell_{i+1}-\ell_{i}\in\{\theta,\theta+1,\theta+2,\dots\}$.
\end{enumerate}
\end{definition}
Note that the conditions of Definition \ref{Def_state_space} imply that for every $i$ from $I_j$,
we have $\ell_i\in [a_i(N)+1,b_i(N)-1]$. An example of a configuration from $\W$ is shown in Figure
\ref{Fig_conf_space}.

\begin{figure}[t]
\center \scalebox{1.0}{\includegraphics{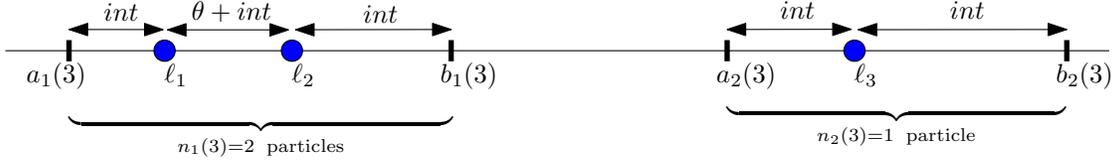}} \caption{The state space
for $N=3$, $k=2$. Numbers $int$ indicate various (possibly different) nonnegative
integers.\label{Fig_conf_space}}
\end{figure}

We also take a \emph{positive} weight function $w(x;N)$ for $x\in \cup_{i=1}^k [a_i(N)+1,b_i(N)-1]$
and define a probability measure $\P$ on $\W$ given by
\begin{equation}\label{eq_distribution_form}
 \P(\ell_1,\dots,\ell_N)= \frac{1}{Z_N} \prod_{1\le i<j \le N}
 \frac{\Gamma(\ell_j-\ell_i+1)\Gamma(\ell_j-\ell_i+\theta)}{\Gamma(\ell_j-\ell_i)\Gamma(\ell_j-\ell_i+1-\theta)}
 \prod_{i=1}^N w(\ell_i;N),
\end{equation}
where $Z_N$ is a normalizing constant which we will refer to as the \emph{partition function}.

\bigskip

\subsection{List of regularity assumptions.}

\label{Section_list_of_assumptions}

Our ultimate goal is to study the asymptotics of the measures $\P$ as $N\to\infty$. For that we
need to assume that the weights $w(x;N)$, as well as all other data specified in Section
\ref{Section_definition}, depend on $N$ in a regular way. Let us present all the technical
assumptions that we impose on the data.

\begin{assumption} \label{Assumptions_basic}
We require that for each $i=1,\dots,k$, as $N\to\infty$
 $$
 a_i(N)=N\hat a_i+O(\ln(N)),\quad  b_i(N)=
   N \hat b_i +O(\ln(N)),
$$
 $$\hat a_1<\hat b_1<\hat a_2<\dots< \hat a_k<\hat b_k.$$
We require that $w(x;N)$ in the intervals $[a_i(N)+1,b_i(N)-1]$, $i=1,\dots,k$, has
the
  form
  $$
   w(x;N)=\exp\left(-N V_N\left(\frac{x}{N}\right)\right)
  $$
  for a function $V_N$ that is continuous in the intervals
  $[a_i(N)+1,b_i(N)-1]$ and such that
  \begin{equation}
  \label{eq_potential_expansion_1}
   V_N(u)=V(u)+O\left(\frac{\ln(N)}{N}\right)
  \end{equation}
  uniformly over $x=u\cdot N$ in the  intervals $[a_i(N)+1,b_i(N)-1]$. The function $V(u)$ is
  differentiable and the following bound holds for a constant $C>0$
  \begin{equation} \label{eq_derivative_bound}
   |V'(u)|\le C\left[1+\sum_{i=1}^k \bigl(|\ln(u-\hat a_i)|+|\ln(u-\hat b_i)|\bigr)\right].
  \end{equation}
\end{assumption}
\begin{remark}  We believe that the assumption on the remainders can be
weakened with minor changes in all the further statements and proofs. However, we do not pursue
this direction due to lack of natural examples.
\end{remark}

For the filling fractions $n_i(N)$ we make a weaker assumption, as it might be important for future
applications, cf.\ \cite{BoG2}.

\begin{assumption} \label{Assumption_filling}
 There exists a constant $C>0$ such that for all $1\le i \le k$ and all large enough
 $N$ we have
 $$
  C<\frac{n_i(N)}{N}<\theta^{-1}(\hat b_i-\hat a_i)-C.
 $$
\end{assumption}
Note that our definition of the state space $\W$ implies that $\frac{n_i(N)}{N}<\theta^{-1}(\hat
b_i-\hat a_i)+o(1)$.

Introduce the notation
$$
 \frac{n_i}{N}=\hat n_i, \quad i=1,\dots,k.
$$
Note that the numbers $\hat n_i$ \emph{ still depend on} $N$. However, we will hide this dependence
from our notations. It is important here that all the limiting values as well as all the remainders
in what follows will be uniform over $\hat n_i$ satisfying Assumption \ref{Assumption_filling}.

\bigskip

The next two assumptions deal with analytic properties of the weight $w(x;N)$ and the equilibrium
measure $\mu$. We fix an open set $\mathcal M_N\subset \mathbb C$, such that $\cup_{i=1}^k
[a_i(N),b_i(N)]\ \subset \mathcal M_N$.

\begin{assumption} \label{Assumptions_ratio} There exist a pair of analytic in $x\in \mathcal M_N$ functions $\phi^+_N(x)$, $\phi^-_N(x)$ such that
$$
 \frac{w(x;N)}{w(x-1;N)} = \frac{\phi^+_N(x)}{\phi^-_N(x)}.
$$
Moreover,
$$
 \phi^\pm_N(x)=\phi^\pm\left(\frac{x}{N}\right)+\frac{1}{N}
 \varphi^\pm_N\left(\frac{x}{N}\right)+O\left(\frac{1}{N^2}\right)
$$
uniformly over $x/N$ in compact subsets of an open set $\mathcal M_N$, which
contains the union of the intervals $[\hat a_i,\hat b_i]$, $i=1,\dots,k$. All the
aforementioned functions are holomorphic in $\mathcal M_N$ and functions
$\varphi^\pm_N$ are uniformly bounded as $N\to\infty$.
\end{assumption}

\begin{remark}
 In the case when $V_N(x)$ in Assumption \ref{Assumptions_basic} is smooth and \emph{uniformly} converges
 to $V$ together with its derivative in a neighborhood of a point $x$,
 we have
 \begin{equation}
 \label{eq_V_to_phi}
  \exp\left(-\frac{\partial}{\partial x} V(x) \right)=
  \frac{\phi^+(x)}{\phi^-(x)}.
 \end{equation}
 Indeed, in this case,
\begin{multline*}
 \frac{\phi^+(x)}{\phi^-(x)}=\lim_{N\to\infty}
 \frac{w(Nx;N)}{w(Nx-1;N)}=\lim_{N\to\infty} \exp\bigl(-N (V_N(x)-V_N(x-1/N) \bigr)=
\\=\lim_{N\to\infty} \exp\left(-N \left(\int_{x-1/N}^x V_N'(y)dy\right)  \right)=
 \exp\left(-\frac{\partial}{\partial x} V(x) \right).
\end{multline*}
\end{remark}

\bigskip

Recall that the equilibrium measure $\mu$ with the density $\mu(x)$ encodes the Law
of Large Numbers for $\P$ stated in Theorem \ref{theorem_LLN_intro} and discussed in
more details in Section \ref{Section_LLN}. A convenient way of working with the
equilibrium measure is through its Stieltjes transform $G_\mu(z)$ defined through
\begin{equation} \label{eq_St_transform_1}
 G_\mu(z)=\int_{-\infty}^{\infty} \frac{\mu(x) dx}{z-x}.
\end{equation}

The following two functions $R_\mu(z)$, $Q_\mu(z)$ are important for our asymptotic study, cf.\
Section \ref{Section_LLN_1}:
\begin{align}
  R_\mu(z):=\phi^-(z)\exp(-\theta G_\mu(z))+\phi^+(z)\exp(\theta G_\mu(z)),\label{eq_def_R} \\
  Q_\mu(z):=\phi^-(z)\exp(-\theta G_\mu(z))-\phi^+(z)\exp(\theta G_\mu(z)). \label{eq_def_Q}
\end{align}

We explain in Section \ref{Section_LLN} that $R_\mu(z)$ is analytic, while
$Q_\mu(z)$ is a branch of a two--valued analytic function which is
 the square--root of a function holomorphic in $\mathcal M_N$. An important technical ingredient of
 our method is a restriction on its zeros, as is summarized in the following assumption.

\begin{assumption} \label{Assumption_simple} We require that for each large enough $N$ and corresponding
 $Q_\mu$ (which depends on $N$ through the filling fractions $\hat n_i$ in the
 definition of $\mu$),
there exists a function $H(z)$ holomorphic in $\mathcal M_N$ and numbers
$\{\alpha_i,\beta_i\}_{i=1}^k$ such that
 \begin{itemize}
 \item $\hat a_i \le \alpha_i < \beta_i \le \hat b_i$, $i=1,\dots,k$;
 \item
 $
  Q_\mu(z)=H(z)\, \prod\limits_{i=1}^k \sqrt{(z-\alpha_i)(z-\beta_i)},
 $\\
 where the branch of the square root is such that $\sqrt{(z-\alpha_i)(z-\beta_i)}\sim z$ when
 $z\to\infty$;
 \item
 $H(z)\ne 0$ for all $z\in\bigcup_{i=1}^k [\hat a_i,\hat b_i]$.
\end{itemize}
\end{assumption}

 We remark that Assumption \ref{Assumption_simple} does not describe a generic case.
 In particular, it implies that there is precisely one interval of support of $\mu(x)$ in each
 interval $[\hat a_i,\hat b_i]$.
 The authors are not aware of simple ways to check
 such property by examining the potential $V(x)$. Nevertheless, many natural models
 arising in the applications satisfy Assumption \ref{Assumption_simple}. We demonstrate this general principle
 by considering several examples in Section \ref{Section_examples}.

\bigskip

Finally, we need a simple vanishing assumption. It is convenient to work with it, yet we later show
in Section \ref{Section_non_van} how it can be relaxed.

\begin{assumption} \label{Assumption_vanishing}
For all $i=1,\dots,k$, we have $\phi^-_N((a_i(N)+1)=\phi^+_N(b_i(N))=0$.
\end{assumption}

\section{Nekrasov's equation}

The main tool for our study of the probability distributions $\P$ from the last section is a
 statement, which is essentially due to Nekrasov \cite{Nekrasov}, \cite{Nek_Pes}, \cite{Nek_PS}. Its
 affine and $q$--versions are given in the following two theorems.

\begin{theorem} \label{Theorem_discrete_loop}
 Let $\P$ be a distribution on $N$--tuples $(\ell_1,\dots,\ell_N)\in\W$ as in the previous
section. Suppose that
\begin{equation}
\label{eq_ratio_def}
 \frac{w(x;N)}{w(x-1;N)} = \frac{\phi^+_N(x)}{\phi^-_N(x)},
\end{equation}
and for all $i=1,\dots,k$, we have $\phi^-_N(a_i(N)+1)=\phi^+_N(b_i(N))=0$.
 Define
\begin{equation} \label{eq_observable}
R_N(\xi)=\phi^-_N(\xi)\cdot
\E_{\P}\left[\prod_{i=1}^N\left(1-\frac{\theta}{\xi-\ell_i}\right)
\right]+\phi^+_N(\xi)\cdot
\E_{\P}\left[\prod_{i=1}^N\left(1+\frac{\theta}{\xi-\ell_i-1}\right)\right].
\end{equation}
If $\phi^\pm_N(\xi)$ are holomorphic in a domain $\mathcal M_N \subset\mathbb C$, then so is
$R_N(\xi)$. Moreover, if $\phi^\pm_N(\xi)$ are polynomials of degree at most $d$, then so is
$R_N(\xi)$.
\end{theorem}
\begin{proof} The possible singularities of $R_N(\xi)$ are simple poles arising from
the denominator of the expression under expectation $\E_\P$ in
\eqref{eq_observable}. Let us compute a residue at such a pole $m$.

 The expectation
$\E_\P$ in \eqref{eq_observable} is a sum over all $(\ell_1,\dots,\ell_N)\in\W$.
Such a configuration contributes to the residue if $\ell_i=m$ or $\ell_i=m+1$ for
some $i=1,\dots,N$.

We separately analyze the contributions appearing from each  $i=1,\dots,N$, which we
now fix. According to definitions, the possible values for $\ell_i$ are $\{A, A+1,
A+2,\dots,B\}$ for certain $A$ and $B$. Given a particle configuration
$\ell=(\ell_1,\dots,\ell_N)$, let $\ell^+$ denote the configuration with $i$th
coordinate increased by $1$, and let $\ell^-$ denote the configuration with $i$th
coordinate decreased by $1$. Note that in principle, $\ell^+$ (similarly $\ell^-$)
may fail to be in $\W$. However, in this case the formula for $\P(\ell^+)$ still
applies, but gives zero.

 Let us explain how the weight \eqref{eq_distribution_form}
 of a configuration $(\ell_1,\dots,\ell_N)$ changes when one
 coordinate is changed from $\ell_i=x$ to $\ell_i=x-1$, i.e.\ we compute the ratio of the weights
 at $\ell_i=x-1$ and at $\ell_i=x$ (all other coordinates are unchanged).
  The double product over $i<j$ in \eqref{eq_distribution_form} produces factors (we denote
  $\ell_j=r$ here)
\begin{equation}
\label{eq_x6}
\frac{\Gamma(r-x+1)\Gamma(r-x+\theta)}{\Gamma(r-x)\Gamma(r-x+1-\theta)} \cdot
\frac{\Gamma(r-x+1)\Gamma(r-x+2-\theta)}{\Gamma(r-x+2)\Gamma(r-x+1+\theta)} =
\frac{(r-x)(r-x+1-\theta)}{(r-x+1)(r-x+\theta)},
\end{equation}
if $i<j$, and the factor
\begin{equation}
\label{eq_x7}
\frac{\Gamma(x-r+1)\Gamma(x-r+\theta)}{\Gamma(x-r)\Gamma(x-r+1-\theta)} \cdot
\frac{\Gamma(x-r-1)\Gamma(x-r-\theta)}{\Gamma(x-r)\Gamma(x-r+\theta-1)} =
\frac{(x-r)(x-r+\theta-1)}{(x-r-1)(x-r-\theta)},
\end{equation}
if $i> j$. Note that \eqref{eq_x6} and \eqref{eq_x7} are two forms of the same rational expression.

Now take $m\in\{A,A+1,\dots,B\}$. The contribution to the residue of $R_N$ at $z=m$,
arising from the $i$th coordinate of configurations $(\ell_1,\dots,\ell_N)$ is
\begin{multline}
\label{eq_x8}
 -\theta \sum_{\ell\in\W\mid \ell_i=m} \phi^-_N(m) \P(\ell_1,\dots,\ell_N) \prod_{j\ne i}
 \left(1-\frac{\theta}{m-\ell_j}\right)\\ + \theta \sum_{\ell\in\W\mid \ell_i=m-1} \phi^+_N(m)
 \P(\ell_1,\dots,\ell_N) \prod_{j\ne i} \left(1+\frac{\theta}{m-\ell_j-1}\right)
\end{multline}

Note the difference between two summation sets in \eqref{eq_x8}. If $\ell_i$ is \emph{not} the
smallest particle in an interval $[a_h(N)+1,b_h(N)-1]$, then the first sum contains the terms with
$\ell_{i-1}=m-\theta$, while the second one does not. However, each such term in the first sum is
actually zero. Also if $\ell_i$ is the smallest particle and $m=A=a_h(N)+1$, then the second sum is
empty as $\ell_i$ is never $m$. But we also know that
 $\phi^-_h(m)=\phi^-_h(a_h(N)+1)=0$ and the first sum vanishes as well.
Similar considerations apply to two cases whether $\ell_i$ is the largest particle or not. We
conclude, that it suffices to study the case when there is one-to-one correspondence between terms
of two sums in \eqref{eq_x8}.

Using \eqref{eq_x6}, \eqref{eq_x7}, and \eqref{eq_ratio_def} we see that
\begin{multline*}
\phi^-_N(m) \P(\ell_1,\dots\ell_{i-1},m,\ell_{i+1},\dots,\ell_N) \prod_{j\ne i}
 \left(1-\frac{\theta}{m-\ell_j}\right)
\\ =\phi^+_N(m)
 \P(\ell_1,\dots\ell_{i-1},m-1,\ell_{i+1},\dots,\ell_N) \prod_{j\ne i}
 \left(1+\frac{\theta}{m-\ell_j-1}\right).
\end{multline*}
We conclude that for each $\mu$, the terms with $\ell=\mu$ and $\ell=\mu^+$ (or
$\ell=\mu^-$ and $\ell=\mu$) in the first and second sum in \eqref{eq_x8} cancel out
and the total residue is zero.

For the polynomiality statement it suffices to notice that if $\phi^\pm_N(\xi)$ are polynomials of
degree at most $d$, then $R_N(\xi)$ is an entire function which grows as $O(\xi^d)$ as
$\xi\to\infty$. Hence, by Liouville's theorem $R_N(\xi)$ is a polynomial.
\end{proof}

The proof of Theorem \ref{Theorem_discrete_loop} reveals that it admits a natural
$q$--deformation. Recall the definition of $q$--Gamma function $\Gamma_q$:
$$
 \Gamma_q(x)=(1-q)^{1-x} \frac{(q;q)_\infty}{(q^x;q)_\infty},
$$
where
$$
 (a;q)_\infty=\prod_{n=0}^{\infty} (1-a q^n).
$$
In the same framework of Section \ref{Section_setup}, define a $q$--deformation of
$\P$ through
 \begin{equation}\label{eq_q_distribution_form}
 \Pq(\ell_1,\dots,\ell_N)= \frac{1}{Z_N^q} \prod_{1\le i<j \le N}
 q^{-\theta(\ell_j-\ell_i)}
 \frac{\Gamma_q(\ell_j-\ell_i+1)\Gamma_q(\ell_j-\ell_i+\theta)}{\Gamma_q(\ell_j-\ell_i)\Gamma_q(\ell_j-\ell_i+1-\theta)}
 \cdot \prod_{i=1}^N w(\ell_i;N).
\end{equation}

\begin{theorem} \label{Theorem_q_discrete_loop}
 Let $\Pq$ be a distribution on $N$--tuples $(\ell_1,\dots,\ell_N)\in\W$ as above. Suppose that
$$
 \frac{w(x;N)}{w(x-1;N)} = \frac{\phi^+_N(x)}{\phi^-_N(x)},
$$
and for all $i=1,\dots,k$, we have $\phi^-_N(a_i(N)+1)=\phi^+_N(b_i(N))=0$.
 Define
\begin{equation}
R_N^q(\xi)=\phi^-_N(\xi)\cdot \E_{\Pq}\left[\prod_{i=1}^N\left(q^{\frac\theta
2}\frac{1-q^{\xi-\ell_i-\theta}}{1-q^{\xi-\ell_i}}\right) \right]+\phi^+_N(\xi)\cdot
\E_{\Pq}\left[\prod_{i=1}^N \left(q^{-\frac\theta
2}\frac{1-q^{\xi-\ell_i-1+\theta}}{1-q^{\xi-\ell_i-1}}\right)\right].
\end{equation}
If $\phi^\pm_N(\xi)$ are holomorphic in a domain $\mathcal M_N \subset\mathbb C$, then so is
$R_N^q(\xi)$.  Moreover, if $\phi^\pm_N(\xi)$ are polynomials of degree at most $d$, then so is
$R_N(\xi)$.
\end{theorem}
\begin{proof} The proof is the same as for Theorem \ref{Theorem_discrete_loop}. The
only new ingredient is a $q$--deformation of \eqref{eq_x6}, \eqref{eq_x7}, which now reads
\begin{multline}
\label{eq_x6_2}
q^{-\theta(r-x)}\frac{\Gamma_q(r-x+1)\Gamma_q(r-x+\theta)}{\Gamma_q(r-x)\Gamma_q(r-x+1-\theta)}
\cdot q^{\theta(r-x+1)}
\frac{\Gamma_q(r-x+1)\Gamma_q(r-x+2-\theta)}{\Gamma_q(r-x+2)\Gamma_q(r-x+1+\theta)}
\\= q^{\theta} \frac{(1-q^{r-x})(1-q^{r-x+1-\theta})}{(1-q^{r-x+1})(1-q^{r-x+\theta})}
\end{multline}
and
\begin{multline}
\label{eq_x7_2}
q^{-\theta(x-r)}\frac{\Gamma_q(x-r+1)\Gamma_q(x-r+\theta)}{\Gamma_q(x-r)\Gamma_q(x-r+1-\theta)}
\cdot
q^{\theta(x-r-1)}\frac{\Gamma_q(x-r-1)\Gamma_q(x-r-\theta)}{\Gamma_q(x-r)\Gamma_q(x-r+\theta-1)}
\\=q^{-\theta} \frac{(1-q^{x-r})(1-q^{x-r+\theta-1})}{(1-q^{x-r-1})(1-q^{x-r-\theta})}.
\end{multline}
One readily recognizes the same expression in \eqref{eq_x6_2} and \eqref{eq_x7_2}.
We further conclude that
\begin{multline*}
 \phi^-_N(m) \Pq(\ell_1,\dots\ell_{i-1},m,\ell_{i+1},\dots,\ell_N) \prod_{j\ne
i}
 \left(q^{\frac\theta 2}\frac{1-q^{m-\ell_j-\theta}}{1-q^{m-\ell_j}}\right)
\\ = \phi^+_N(m)
 \Pq(\ell_1,\dots\ell_{i-1},m-1,\ell_{i+1},\dots,\ell_N) \prod_{j\ne i}
 \left(q^{-\frac\theta 2}\frac{1-q^{m-\ell_j-1+\theta}}{1-q^{m-\ell_j-1}}\right),
\end{multline*}
and therefore there is a cancelation of the poles.
\end{proof}

\section{Law of Large Numbers}

\label{Section_LLN}

\subsection{Limit shape}

Assumptions \ref{Assumptions_basic} and \ref{Assumption_filling} guarantee that a certain law of
large numbers for the measures $\P$ holds as $N\to\infty$. To state it we introduce a random
probability measure $\mu_N$ on $\mathbb R$ via
\begin{equation}
 \mu_N=\frac{1}{N}\sum_{i=1}^N \delta\left(\frac{\ell_i}{N}\right),\quad
 (\ell_1,\dots,\ell_N) \text{ is } \P\text{-distributed.} \label{eq_empirical_mes}
\end{equation}
The measure $\mu_N$ is often referred to as the \emph{empirical measure} of point configuration
$\ell_1,\dots,\ell_N$, cf.\ \eqref{eq_empirical_mes_1}. Note that our definitions imply the
condition $\ell_{i+1}-\ell_i\ge \theta$, which shows that for any interval $[p,q]$, its
$\mu_N$--measure is bounded from above by $\theta^{-1}(q-p+N^{-1})$.

The first order asymptotic behavior of measures $\mu_N$ can be understood through a
variational problem. For a probability measure $\rho$ supported on $\cup_i [\hat
a_i,\hat b_i]$, define
 \begin{equation}
 \label{eq_functional_general}
   I_V[\rho]=\theta \iint_{x\ne y} \ln|x-y|
  \rho(dx)\rho(dy) - \int_{-\infty}^{\infty} V(x) \rho(dx).
 \end{equation}

\begin{lemma} \label{Lemma_maximizer} Let
$\Theta$ be the set of absolutely continuous probability measures $\rho(x)dx$
supported on $\cup_i [\hat a_i,\hat b_i]$, whose density is between $0$ and
$\theta^{-1}$ and such that
 $$
 \int_{\hat a_i}^{\hat b_i} \rho(x)dx=\hat
 n_i,\quad 1\le i \le k,
 $$
 where $0<\hat n_i<\theta^{-1}(\hat b_i-\hat a_i)$, $i=1,\dots,k$, are such that
 $\sum_{i=1}^k \hat n_i=1$.
 Then the functional $I_V$ has a unique maximum $\mu(x)dx$ on $\Theta$.
\end{lemma}
\begin{remark}
 The maximizer of $I_V$ is called the \emph{equilibrium measure}.
\end{remark}
\begin{proof}[Proof of Lemma \ref{Lemma_maximizer}]
 Let us equip $\Theta$ with weak topology, i.e.\ the topology of pointwise
 convergence for the distribution functions. Then $\Theta$ is compact.

 Observe that the functional $I_V$ is continuous in the weak topology on
 $\Theta$; here it is crucial that the measures in $\Theta$ have density
 between $0$ and $\theta^{-1}$. Therefore, $I_V$ attains its maximum on
 $\Theta$. It remains to show that such a maximum is unique.

 For that we note that $I_V$ is strictly concave, i.e.\ for any $\rho,\rho'\in
\Theta$ and any $0<t<1$
 $$
 I_V[t\rho+ (1-t) \rho']> t I_V[\rho]+ (1-t) I_V[\rho'].
 $$
 Indeed, the linear part of $I_V$ is concave by the definition. The quadratic part
 is negatively--definite due to the formula \eqref{eq_quadratic_alternative} for it,
 therefore, it is strictly concave.

 Since a strictly concave functional cannot have more than one maximum on a convex set, we are done.
\end{proof}

\begin{theorem} \label{Theorem_LLN}
 Under Assumptions \ref{Assumptions_basic} and \ref{Assumption_filling}
 the random measures $\mu_N$ concentrate (in
 probability) near  $\mu(x)dx$ of Lemma \ref{Lemma_maximizer}. More
 precisely, for each Lipshitz function $f(x)$ defined in a real
 neighborhood of $\cup_{i=1}^k [\hat a_i, \hat b_i]$,  and each $\eps>0$ the random variables
 $$
  N^{1/2-\eps}\left|\int_{-\infty}^{\infty} f(x) \mu_N(dx)- \int_{-\infty}^{\infty} f(x)
  \mu(x)dx\right|
 $$
 converge to $0$ in probability and in the sense of moments.
\end{theorem}
\begin{remark}
 Since $\mu(x)$ depends on $\hat n_i$, and the latter depend on $N$, so is the former.
 When $\hat n_i$ converge as $N\to\infty$, this dependence can be removed by using
 the continuity of $\mu(x)$ in $\hat n_i$, established below in Proposition \ref{Prop_dependence_on_filling}.
\end{remark}
 Analogues of Theorem \ref{Theorem_LLN} are known in the
 literature, cf.\ \cite{MMS}, \cite{AGZ}, \cite{Feral}. Our proof of Theorem \ref{Theorem_LLN} relies on a different characterization for $\mu$ maximizing
the functional $I_V[\cdot]$.

The restriction $0\le \mu(x)\le \theta^{-1}$ leads to the subdivision of $\mathbb R$
into three types of regions (we borrow the terminology from \cite{BKMM}):
\begin{itemize}
 \item Maximal (with respect to inclusion) closed connected intervals where $\mu(x)=0$ are called \emph{voids}.
 \item Maximal open connected intervals where $0<\mu(x)<\theta^{-1}$ are called \emph{bands}.
 \item Maximal closed connected intervals where $\mu(x)=\theta^{-1}$ are called \emph{saturated
 regions}.
\end{itemize}
In a related context of random tilings and periodically-weighted dimers the voids
and saturated regions are usually called \emph{frozen}, while bands are \emph{liquid
regions}.

\bigskip

We further define the \emph{effective potential} $F_V(x)$ through (cf.\ \eqref{eq_effective_pot})
$$
 F_V(x)=2\theta \int_{-\infty}^\infty \ln|x-y|\mu(y)dy   - V(x).
$$

\begin{lemma}
 There exist $k$ constants $f_1,\dots,f_k$, such that
\begin{align}
 F_V(x)-f_i\le 0,& \text{ for all } x \text{ in voids in } [\hat a_i,\hat b_i];\label{eq_void_inequality}\\
 F_V(x)-f_i\ge 0,& \text{ for all } x \text{ in saturated regions in }
 [\hat a_i,\hat b_i];\label{eq_saturated_inequality}\\
 F_V(x)-f_i=0,& \text{ for all } x \text{ in bands in } [\hat a_i,\hat b_i].
 \label{eq_band_equality}
\end{align}
\end{lemma}
\begin{proof}
 This characterization is readily obtained from the variation of the functional $I_V[\cdot]$, cf.\ \cite{ST}, \cite{DS}.
\end{proof}

Further fix a parameter $p>2$ and let $\tilde \mu_N$ denote the convolution of the empirical
measure $\mu_N$ with uniform measure on the interval $[0,N^{-p}]$. The following statement is a
generalization of Proposition \ref{Prop_pseudodistance_bound}. Let us emphasize that the measure
$\mu$ here and below depends on the filling fractions $\hat n_i= n_i(N)/N$, and thus changes with
$N$.

\begin{proposition} \label{Prop_pseudodistance_bound_general}
 Let $\mu$ be the maximizer of Lemma \ref{Lemma_maximizer}.
 Then there exists $C\in\mathbb R$ such that for all $\gamma>0$ and all $N$ we have
 $$
  \Pp_{N}
  \Bigl( \Dr(\tilde\mu_N,\mu)\ge \gamma\Bigr)\le\exp\bigl( CN\ln^2(N)
  -\gamma^2 N^2\bigr),
 $$
 where $\Dr$ was defined in \eqref{eq_quadratic_main},
 \eqref{eq_quadratic_alternative}.
\end{proposition}
\begin{proof}
 We start by analyzing the asymptotic of the formula for  $\Pp_N$. We observe
 that the definition of $I_V[\cdot]$ by \eqref{eq_functional_general}
 makes sense for discrete measures. It is important here that we integrate only over $x\ne y$ in \eqref{eq_functional_general}
  as otherwise
 the integral would be infinite. Using for the double product
 in \eqref{eq_distribution_form} the following corollary of the Stirling's formula
 \begin{equation}
 \label{eq_Gamma_expansion}
  \frac{\Gamma(h+\theta)}{\Gamma(h)}= h^{\theta} \bigl(1+O(h^{-1}) \bigr), \quad
  h\to+\infty,
 \end{equation}
 we can write for every
$(\ell_1,\ell_2,\dots,\ell_N)\in \W$
\begin{equation} \label{eq_x37}
\Pp_{N}(\ell_1,\dots,\ell_N) =\dfrac{\exp\Bigl(2\theta N(N-1)\ln(N)+ N^2 I_V\bigl[ {\rm mes}[\ell_1,\dots,\ell_N] \bigr]
+O(N\ln(N))\Bigr)} {Z_N},
\end{equation}
where
$$
{\rm mes}[\ell_1,\dots,\ell_N]=\frac{1}{N}\sum_{i=1}^N
\delta\left(\frac{\ell_i}{N}\right).
$$
Let us obtain a lower bound for the partition function $Z_N$ in \eqref{eq_x37}. For that let $x_i$,
 $i=1,\dots,N$ be quantiles of $\mu$ defined through
 $$
  \int_0^{x_i} \mu(x) dx = \frac{i-1/2}{N},\quad 1\le i\le N.
 $$
 Since $\mu(x)\le \theta^{-1}$, $\theta(x_{i+1}-x_i)\ge 1/N$. Therefore, using the asymptotics of
 $a_i(N)$ and $b_i(N)$ of
 Assumption \ref{Assumptions_basic}
 we conclude that there exists an absolute constant $U$ (independent of $N$)
 and a configuration
 $(\hat \ell_1,\dots,\hat \ell_N)\in\W$ such that
 \begin{equation} \label{eq_x55}
 |Nx_i-\hat \ell_i|\le
 U
 \end{equation}
  for all $1\le i\le N$ except for $O(\ln(N))$ ones.  Arguing as in Proposition
 \ref{Prop_pseudodistance_bound},
 we  write
$$
  \frac{Z_N}{N^{2\theta N(N-1)}}\ge
 \exp\Bigl( N^2 I_V\bigl[ {\rm mes}[\hat \ell_1,\dots,\hat \ell_N] \bigr]
  +O(N\ln N)\Bigr)=\exp\Bigl( N^2 I_V\bigl[ \mu  \bigr] +O(N \ln^2(N))\Bigr).
 $$
 Note that the remainder in the last formula is $O(N \ln^2(N))$ instead of $O(N\ln(N))$ in the
 proof of Proposition \ref{Prop_pseudodistance_bound}. This is due to an additional error produced by
 the $\ell_i$ for which \eqref{eq_x55} does not hold: there are $O(\ln(N))$ of such $\ell_i$ and
 each of them produces (at most) $O(N \ln(N))$ of error.

 From here on the proof literally repeats that of Proposition
 \ref{Prop_pseudodistance_bound}.
\end{proof}

\begin{corollary} \label{Corrollary_tail_bound_general} Recall the notations $\| g\|_{1/2}$ and
$\| g \|_{\rm Lip}$ of Corollary \ref{Corrollary_tail_bound}.
  Fix any $p>2$ and let $\mu(x)dx$ be the maximizer of Lemma \ref{Lemma_maximizer}. Then there exists $C\in\mathbb R$
  such that for all $\gamma>0$, all $N$ and all $f$ we have
\begin{multline} \label{eq_x31}
 \Pp_{N}\left( \left|\int_{\mathbb R} g(x) \mu_N(dx)-\int_{\mathbb R} g(x) \mu(x)dx\right|
 \ge \gamma \| g \|_{1/2}+ \frac{\|g\|_{\rm Lip}}{N^p} \right)\\ \le \exp\left( C N\ln^2(N)-\frac{\gamma^2
 N^2}{2}\right).
\end{multline}
\end{corollary}
\begin{proof} The proof is the same as for Corollary \ref{Corrollary_tail_bound}.
\end{proof}

\begin{proof}[Proof of Theorem \ref{Theorem_LLN}] The desired statement follows from Corollary
\ref{Corrollary_tail_bound_general} with $\gamma=q\cdot N^{\eps-1/2}$, $q>0$.
\end{proof}

Recall that the set of measures $\Theta$ and the equilibrium measure $\mu$ of Lemma
\ref{Lemma_maximizer} implicitly depend on many parameters. For the next statement
we reconstruct a part of this dependence in the notations. For $\hat n=(\hat
n_1,\dots,\hat n_k)$ let $\mu^{\hat n}$ be the maximizer of $I_V$ over $\Theta^{\hat
n}$, the set of probability measures $\mu$ absolutely continuous with respect to
Lebesgue measure, with density bounded above by $1/\theta$  and with mass $\mu([\hat
a_i,\hat b_i])=\hat n_i$, $\sum_{i=1}^k\hat n_i=1$.
\begin{proposition} \label{Prop_dependence_on_filling}
Assume that $0<\hat n_i<\theta^{-1}(\hat b_i-\hat a_i)$ for all $i=1,\dots,k$ and
 that the measure $\mu^{\hat n_i}$ has at least one band in each of the
intervals $[\hat a_i, \hat b_i]$, $i=1,\dots,k$. Then there exists a finite constant $C_{\hat n}$
such that for $\hat n'$ close enough to $\hat n$
\begin{equation} \label{eq_x58} \Dr(\mu^{\hat  n},\mu^{\hat n'})\le
C_{\hat n} \|\hat n-\hat n'\|_\infty.
\end{equation}
\end{proposition}
\begin{remark}
 Arguing as in the proof of Corollary \ref{Corrollary_tail_bound}, one deduces from \eqref{eq_x58} that for any
 smooth function $g(x)$, the averages $\int g(x) \mu^{\hat n}(x) dx$ and $\int g(x) \mu^{\hat n'}(x) dx$
 are close. Their difference goes to zero as $\hat n'\to \hat n$.
\end{remark}
\begin{remark} In all the applications that we present, the assumption that the measure $\mu^{\hat n_i}$ has at least one band in each of the
intervals $[\hat a_i, \hat b_i]$, $i=1,\dots,k$, is satisfied automatically, see the
end of Section \ref{Section_multi_Krawtchouk} for a general argument in this
direction.
\end{remark}
\begin{proof}[Proof of Proposition \ref{Prop_dependence_on_filling}]
 By definition, for any probability
measure $\nu$ with density $\nu(x)$
\begin{eqnarray*}
I_V(\nu)&=&I_V(\mu^{\hat n}) -\theta \D(\nu,\mu^{\hat n}) - \int\left(V(x)-2\theta \int\log|x-y|\mu^{\hat n}(y)dy\right)(\nu(x)-\mu^{\hat n}(x))dx\\
&=&I_V(\mu^{\hat n}) -\theta \D(\nu,\mu^{\hat n}) + \int \tilde F_V(x) (\nu(x)-\mu^{\hat n}(x))dx+\sum_{i=1}^k f_i (\nu-\mu^{\hat n})([\hat a_i,\hat b_i]) \\
\end{eqnarray*}
where $\tilde F_V(x)=F_V(x)-f_i$ on $[\hat a_i,\hat b_i]$ is nonpositive on voids,
nonnegative on saturated regions and vanishes on bands for $\mu^{\hat n}$, cf.\
\eqref{eq_void_inequality}--\eqref{eq_band_equality}. We next choose $\hat n'\neq
\hat  n$ and let $\nu$ be a probability measure with filling fractions $\hat n'$ so
that
$$I_V(\nu)\le I_V(\mu^{\hat n'}).$$
The above decomposition and $(\nu-\mu^{\hat n'})([\hat a_i,\hat b_i])=0$ implies
$$\theta \D(\mu^{\hat n'},\mu^{\hat n})-\int \tilde F_V(x) (\mu^{\hat n'}(x)-\mu^{\hat n}(x)
)dx \le \theta \D(\nu,\mu^{\hat n})-\int \tilde F_V(x) (\nu(x)-\mu^{\hat n}(x) )dx\,.
$$
Observe that $\int \tilde F_V(x) (\nu(x)-\mu^{\hat n}(x) )dx$ is nonpositive for any probability
measure $\nu$ whose density is at most $\theta^{-1}$, and vanishes if $\nu-\mu^{\hat n}$ is
supported on the bands of $\mu^{\hat n}$. We assume that this is the case to deduce that
\begin{equation}\label{eq_x56}
\D(\mu^{\hat n'},\mu^{\hat n})\le \D(\nu,\mu^{\hat n})\,.
\end{equation}
 Finally we choose
$\nu=\mu^{\hat n} +\sum_i (\hat  n_i'-\hat n_i) \frac{1_{B_i} }{|B_i|}dx$ where
$B_i$ denotes a subset in the interior of  bands in $[\hat a_i,\hat b_i]$  so that
on $B_i$ we have $\delta\le d\mu^{\hat n}/dx\le 1/\theta-\delta$ for some small
$\delta>0$. Note that our assumption on the existence of bands in all the intervals
$[\hat a_i, \hat b_i]$ implies that $B_i$ is not empty.
 For $ n_i'- n_i$ small enough (smaller than $\delta |B_i|$), $\nu\in
\Theta^{\hat n'}$. Hence, we have \begin{equation} \label{eq_x57}
 \D(\nu,\mu^{\hat n})=-
 \sum_{i,j=1}^k \frac{(\hat  n_i'-\hat n_i) (\hat  n_j'-\hat n_j)}{|B_i||B_j|}
 \int_{t\in B_i} \int_{s\in B_j} \ln|t-s|\, dt\, ds \le C_{\hat n}\|\hat n-\hat n'\|^2_\infty
\end{equation}
 for a constant $C_{\hat n}>0$. Combining \eqref{eq_x56} and \eqref{eq_x57} we
 obtain
 $$\D(\mu^{\hat n'},\mu^{\hat n})\le C_{\hat n}\|\hat n-\hat n'\|_\infty^2\,.\qedhere $$
\end{proof}

\subsection{Functions $R$ and $Q$}

\label{Section_R_and_Q}

We work with the equilibrium measure $\mu$ of density $\mu(x)$ in terms of its
Stieltjes transform $G(z)$ defined through
\begin{equation} \label{eq_St_transform}
 G_\mu(z)=\int_{-\infty}^{\infty} \frac{\mu(x) dx}{z-x}.
\end{equation}
Observe that \eqref{eq_St_transform} makes sense for all $z$ outside the support of
$\mu(x)$ and $G_\mu(z)$ is holomorphic there.

We are going to extensively use the following two notations for any function $F(z)$
of a complex variable $z$:
$$
 F(z+\i 0):=\lim_{\delta\to 0} F(z+\i\delta),\quad \quad F(z-\i 0):=\lim_{\delta\to 0}
 F(z-\i\delta).
$$
At every point $x$ where $\mu(x)$ is continuous,
it can be reconstructed by its Stieltjes transform via
\begin{equation}
\label{eq_Stil_inversion}
 \mu(x)=\frac{1}{2\pi\i}\left(G_\mu(x+\i 0)-G_\mu(x-\i 0)\right),\quad x\in\mathbb R.
\end{equation}
On the other hand, the equilibrium measure characterization \eqref{eq_band_equality}
implies (cf.\ \cite{DS}, \cite{ST}) that for $x$ in a band of the equilibrium
measure
\begin{equation}
\label{eq_variational_equation}
 \theta \bigl(G_\mu(x+\i 0)+ G_\mu(x-\i 0)\bigr)=\frac{\partial}{\partial x}
 V(x).
\end{equation}

We also define the Stieltjes transform $G_N(z)$ of the prelimit empirical measure
\eqref{eq_empirical_mes} through
\begin{equation} 
 G_N(z)=\int_{-\infty}^{\infty} \frac{1}{z-x} \mu_N(dx)=\frac{1}{N}
 \sum_{i=1}^N
 \frac{1}{z-\ell_i/N},\quad (\ell_1,\dots,\ell_N)\in\W \text{ is }
 \P\text{-distributed}.
\end{equation}

The functions $R_\mu(z)$, $Q_\mu(z)$ defined in \eqref{eq_def_R}, \eqref{eq_def_Q}.
are important for our asymptotic study.

\begin{proposition} \label{Proposition_zero_order_equation}
 Under Assumptions \ref{Assumptions_basic}, \ref{Assumption_filling},
 \ref{Assumptions_ratio}, \ref{Assumption_vanishing},
  $R_\mu(z)$ is holomorphic in $\mathcal M_N$.
\end{proposition}
\begin{proof}[Proof of Proposition \ref{Proposition_zero_order_equation}]
 Since $G_\mu(z)$ is holomorphic everywhere outside the (real) support of the
 equilibrium measure $\mu(x)$, so is $R_\mu(z)$.

 Further take any point $x$ in the support of $\mu(x)$. Choose a
 simple contour $\gamma$ in $\mathcal M_N\setminus [\hat a_i ,\hat b_i]$ enclosing
 $x$. Observe that as $N\to\infty$ under the change of variables $\xi=Nz$, we have
\begin{equation}\label{eq_x13}
 \begin{split}
 &\left(1-\frac{\theta}{\xi-\ell_i}\right)= \left(1-\frac{1}{N}\cdot
 \frac{\theta}{z-\ell_i/N}\right)=\exp\left(-\frac{1}{N} \cdot \frac{\theta}{z-\ell_i/N} +
 O\left(\frac{1}{N}\right)\right),\\
 &\left(1+\frac{\theta}{\xi-\ell_i-1}\right)= \left(1+\frac{1}{N}\cdot
 \frac{\theta}{z-\ell_i/N-1/N}\right)=\exp\left(\frac{1}{N} \cdot \frac{\theta}{z-\ell_i/N} +
 O\left(\frac{1}{N}\right)\right).
 \end{split}
 \end{equation}
 Now fix $\hat n_i$, $i=1,\dots,k$ and choose the filling fractions $n_i(N)$ in such a way that ${|n_i(N)-N
 \hat n_i|\le 1}$ for all $N$ and $i$. We claim that
 $R_N(Nz)$ of Theorem \ref{Theorem_discrete_loop} converges
 to $R_\mu(z)$ uniformly on $\gamma$. Indeed, the functions $\phi^\pm_N$ converge to
 $\phi^\pm$ by Assumption \ref{Assumptions_ratio}. By \eqref{eq_x13} and Theorem \ref{Theorem_LLN}
 the two expectations in the definition of $R_N(Nz)$ approximate
 $\exp(-\theta G_{\nu}(z))$ and $\exp(-\theta G_{\nu}(z))$, respectively, where $\nu$
 is the maximizer of Lemma \ref{Lemma_maximizer} with filling fractions $\frac{n_1(N)}{N}$, \dots
 $\frac{n_k(N)}{N}$. Since $\frac{n_i(N)}{N}\to \hat n_i$ as $N\to\infty$,
 Proposition \ref{Prop_dependence_on_filling} implies that $G_\nu(z)$
 converges to $G_\mu(z)$.

 Since uniform convergence on $\gamma$ of holomorphic (everywhere inside the contour)
 functions $R_N(Nz)$
 implies the same convergence inside the contour and the holomorphicity of the
 limit, we are done.
\end{proof}
\begin{remark} Alternatively, one can prove Proposition
\ref{Proposition_zero_order_equation} by showing that $R_\mu(z)$ is continuous near real $x$ via
combining \eqref{eq_variational_equation} with \eqref{eq_V_to_phi}.
\end{remark}

Despite their similar form, the analytic properties of the function $Q_\mu(z)$ are very different.
Observe that
\begin{equation} \label{eq_quadratic_on_Q}
  \bigl( Q_\mu(z) \bigr)^2= \bigl( R_\mu(z) \bigr)^2 - 4\phi^+(z)\phi^-(z).
\end{equation}
 Thus, $Q_\mu(z)$ is a branch of a two--valued analytic function which is
 the square--root of a function holomorphic in $\mathcal M_N$. Note that $Q_\mu(z)$ is
 analytic outside $\cup_{i=1}^k [\hat a_i, \hat b_i]$. On the other hand, combining
 \eqref{eq_variational_equation} with \eqref{eq_V_to_phi} one observes that $Q_\mu(z)$
 must have discontinuities in the bands of $\mu(x)$ and the endpoints of the bands should be its branching points.
 Therefore, $Q_\mu(z)$ has at least two branching points inside each interval $[\hat a_i, \hat b_i]$, which also have to be zeros of $Q_\mu(z)$.
 For our method it is important to assume that there are precisely two zeros in each
 interval, as is summarized in Assumption \ref{Assumption_simple}.

\section{Second order expansion.}

\label{Section_second_order}

The goal of this section is to prove a generalization of Theorem \ref{Theorem_main_1} in the setup
of Section \ref{Section_setup}.

Theorem \ref{Theorem_LLN} implies that $G_N(z)-G_\mu(z)$ vanishes as $N\to\infty$ uniformly over
$z$ in compact subsets of $\mathbb C\setminus \bigcup_{i=1}^{k} [\hat a_i,\hat b_i]$. Moreover,
since all the involved functions are analytic in $z$, we can infer also the uniform convergence of
the derivatives.

Similarly to Section \ref{Section_Second_order_1}, we introduce a deformed version of $G_N(z)$.
Take $2m$ parameters $\t=(t_1,\dots,t_m)$, $\v=(v_1,\dots,v_m)$ such that $v_a+t_a-y\ne 0$ for all
$a=1,\dots,m$ and all $y\in \cup_{i=1}^{k} [\hat a_i,\hat b_i]$, and let the deformed distribution
$\Pp_{N}^{\t,\v}$ be defined through
\begin{equation}\label{eq_distr_deformed}
 \Pp_{N}^{\t,\v}(\ell_1,\dots,\ell_N)=\frac{Z_N}{Z_N(\t,\v)}\cdot \Pp_N(\ell_1,\dots,\ell_N) \cdot  \prod_{i=1}^N
\prod_{a=1}^{m}\left(1+\frac{t_a}{v_a-\ell_i/N}\right),
\end{equation}
where $\Pp_N$ was defined in \eqref{eq_distribution_form}.
 In general, $\Pp_{N}^{\t,\v}$ may be a \emph{complex--valued} measure, but
the normalizing constant $Z_N(\t,\v)$ in \eqref{eq_distr_deformed} is chosen so that
the total mass of $\Pp_{N}^{\t,\v}$ is $1$, i.e.\ $\sum_{\ell\in \W}
\Pp_{N}^{\t,\v}(\ell)=1$. The numbers $t_a$, $a=1,\dots, k$ are always assumed to be
in a small neighborhood of $0$, which guarantees, in particular, that that the
measure is normalizable, i.e.\ $Z_N(\t,\v)\ne 0$.

Observe that $\Pp_{N}^{\t,\v}$ satisfies the assumptions of Theorem \ref{Theorem_discrete_loop}
with (in the notations of Assumption \ref{Assumptions_ratio})
$$
\phi^+_N(x)=\left(\phi^+\left(\frac{x}{N}\right)+\frac{1}{N}\varphi^+_N\left(\frac{x}{N}\right)+O\left(\frac{1}{N^2}\right)
 \right)\cdot \prod_{a=1}^m
 \biggl(v_a-\frac{x}{N}+\frac{1}{N}\biggr)\biggl(t_a+v_a-\frac{x}{N}\biggr),\\
$$
$$
  \phi^-_N(x)=\left(\phi^-\left(\frac{x}{N}\right)+\frac{1}{N}\varphi^-_N\left(\frac{x}{N}\right)+O\left(\frac{1}{N^2}\right)
  \right)\cdot\prod_{a=1}^m
 \biggl(v_a-\frac{x}{N}\biggr)\biggl(t_a+v_a-\frac{x}{N}+\frac{1}{N}\biggr).
$$
We also define
$$
 \psi_N^-(z)= \varphi_N^-(z)+\sum_{a=1}^m \frac{\phi^-(z)}{t_a+v_a-z},
 \quad \psi_N^+(z)=\varphi_N^+(z)+ \sum_{a=1}^m \frac{\phi^+(z)}{v_a-z}.
$$
Clearly,
$$
 \phi^\pm_N(Nz)=
 \left(\phi^\pm(z)+\frac{\psi_N^\pm(z)}{N}+O\left(\frac{1}{N^2}\right)\right) \cdot
 \prod_{a=1}^m\bigl(v_a-z\bigr)\bigl(t_a+v_a-z\bigr).
$$

We further define $\mu_N^{\t,\v}$ as the empirical distribution of $\Pp_{N}^{\t,\v}$ and set, cf.\
\eqref{eq_St_transform_deformed}
\begin{multline} \label{eq_St_transform_deformed_general}
 G_N(z)=
 \int_{-\infty}^{\infty} \frac{1}{z-x}\, \mu_N^{\t,\v}(dx)
 =\frac{1}{N}
 \sum_{i=1}^N
 \frac{1}{z-\ell_i/N},\quad (\ell_1,\dots,\ell_N) \text{ is }
 \Pp_{N}^{\t,\v}\text{-distributed}.
\end{multline}
Define
\begin{equation}
 \Delta G_N(z)=N\bigl(G_N(z)-G_\mu(z)\bigr).
\end{equation}
We will now formulate a generalization of Theorem \ref{Theorem_main_1}. For that we need to
introduce certain notations from the theory of hyperelliptic integrals.

Fix $k$ simple positively--oriented complex contours $\gamma_1$, $\gamma_2$,\dots,$\gamma_k$, such
that each $\gamma_i$ $(1\le i \le k)$ encloses the segment $[\hat a_i,\hat b_i]$ (and thus, also
$[\alpha_i,\beta_i]$), but not the other segments.

Let $P(z)=p_0 + p_1 z+\dots+ p_{k-2} z^{k-2}$ be a polynomial of degree $k-2$, and consider the map
$$
\Omega: P(z)\mapsto \left(\frac{1}{2\pi\i} \oint_{\gamma_1}
 \frac{P(z)dz}{\prod_{i=1}^k \sqrt{(z-\alpha_i)(z-\beta_i)}},\dots, \frac{1}{2\pi\i} \oint_{\gamma_k}
 \frac{P(z)dz}{\prod_{i=1}^k \sqrt{(z-\alpha_i)(z-\beta_i)}} \right)
$$
Note that the sum of the integrals in the definition of $\Omega$ equals $(-2\pi\i)$
times the residue of $\frac{P(z)}{\prod_{i=1}^k\sqrt{(z-\alpha_i)(z-\beta_i)}}$ at
 infinity, which is zero. Therefore, $\Omega$ is a linear map between
$(k-1)$--dimensional vector spaces. This map is very non-trivial, but it is known to
be an isomorphism for any $k\ge 2$ (cf.\ \cite[Section 2.1]{Dubrovin}).

Using $\Omega$ we can now define a more complicated map $\Upsilon$. Given a (continuous) function
$f(z)$ defined on the contours $\gamma_i$ and such that the sum of its integrals over these
contours is zero, we define a function $\Upsilon_z[f]$ through
$$
 \Upsilon_z[f]=f(z)+ \frac{P(z)}{\prod_{i=1}^k \sqrt{(z-\alpha_i)(z-\beta_i)}},
$$
where $P(z)$ is a unique polynomial of degree at most $k-2$, such that for each
$i=1,\dots,k$
$$
 \frac{1}{2\pi\i} \oint_{\gamma_i} \Upsilon_z[f] dz=0.
$$
The polynomial $P(z)$ can be evaluated in terms of the map $\Omega$ via
$$
 P=\Omega^{-1}\left(-\frac{1}{2\pi\i} \oint_{\gamma_1}f(z)dz,\; -\frac{1}{2\pi\i} \oint_{\gamma_2}f(z)dz,\; \dots,\;-\frac{1}{2\pi\i}
 \oint_{\gamma_k}f(z)dz\right).
$$
Note that the map $f\mapsto \Upsilon_z[f]$ is linear and does not depend on $\t$,
$\v$.

\begin{theorem} \label{Theorem_main_general} Fix $m=0,1,\dots$ and choose  complex numbers
$\v=(v_1,\dots,v_m)\subset (\mathbb C\setminus \cup_{i=1}^k [\hat a_i,\hat b_i])^m$. Under
Assumptions \ref{Assumptions_basic}--\ref{Assumption_vanishing} we have as $N\to\infty$
\begin{multline} \label{eq_first_order}
 \frac{\partial^m}{\partial t_1\cdots \partial t_m} \E_{\Pp_{N}^{\t,\v}} \bigl(  \Delta
 G_N(u)\bigr) \Bigr|_{t_a=0,\, 1\le a\le m}
 = o(1)+ \frac{\partial^m}{\partial t_1\cdots \partial t_m}\Biggl(\\
 \Upsilon_u \Biggl[
 \frac{\theta^{-1}}{2\pi\i\,
{\prod_{i=1}^k\sqrt{(u-\alpha_i)(u-\beta_i)}} } \oint_{\bigcup_{i=1}^k \gamma_i}
\frac{dz}{(u-z)H(z)}
  \cdot\Biggl(\psi^-_N(z) e^{-\theta G_\mu(z)} +\psi^+_N(z) e^{\theta
G_\mu(z)}
\\+
\phi^-(z) e^{-\theta G_\mu(z)} \frac{\theta^2}{2} \frac{\partial}{\partial z}G_\mu(z) +
 \phi^+(z) e^{\theta G_\mu(z)}
\left(\frac{\theta^2}{2}-\theta\right) \frac{\partial}{\partial z}G_\mu(z)\Biggr)
 \Biggr] \Biggr)_{t_a=0,\, 1\le a\le m},
\end{multline}
where $\gamma_i$ are  simple positively--oriented contours enclosing the segment
$[\hat a_i, \hat b_i]$ (the points $u$ and $v_1,\dots,v_m$ are outside the
contours). The remainder $o(1)$ is uniform over $u, v_1,\dots,v_m$ in compact
subsets of the unbounded component of $\mathbb C\setminus \cup_{i=1}^k\gamma_i$.
\end{theorem}
\begin{remark}
 Theorem \ref{Theorem_main_general} generalizes Theorem  \ref{Theorem_main_1} in
 several directions. First, before we had $k=1$ and now we allow any $k=1,2,\dots$.
 Second, $\theta$ was equal to $1$ and now $\theta>0$ is arbitrary. Finally, in Theorem
 \ref{Theorem_main_1} the functions $H(z)$, $\phi^\pm(z)$, $G_\mu(z)$ had a
 specific form.
\end{remark}
\begin{proof} Similarly to Theorem \ref{Theorem_main_1}, the proof has two parts:
algebraic manipulations and self-improving estimates for the remainders. The latter
literally repeats Section \ref{Section_estimates} and we are not going to present
it.

For the former we start from the statement of Theorem \ref{Theorem_discrete_loop}.
Making the change of variables $\xi=Nz$, we can write for $z$ outside
$\bigcup_{i=1}^k [\hat a_i,\hat b_i]$
\begin{equation}
\begin{split}\label{eq_first_term}
 \prod_{i=1}^{N}\left(1-\frac{\theta}{\xi-\ell_i} \right)&=
 \exp\left(\sum_{i=1}^{N} \ln\left(1-\frac{1}{N}\cdot
 \frac{\theta}{z-\ell_i/N}\right)\right)
 \\
 &=
  \exp\left(-\theta G_N(z)+\frac{\theta^2}{2N}\cdot \frac{\partial}{\partial z}G_N(z)+O\left(\frac1{N^2}\right) \right)
  \\
 &=
  \exp\left(-\theta G_\mu(z) -\frac{\theta}{N} \Delta G_N(z)+\frac{\theta^2}{2N}\cdot \frac{\partial}{\partial z}G_N(z)+O\left(\frac1{N^2}\right) \right)
  .
\end{split}
\end{equation}
 Similarly, also
\begin{multline} \label{eq_second_term}
 \prod_{i=1}^{N}\left(1+\frac{\theta}{\xi-\ell_i-1} \right)=
 \exp\left(\sum_{i=1}^N\ln\left(1+\frac{1}{N}\cdot
   \frac{\theta}{z-\ell_i/N-1/N}\right)\right)
 \\
 =
  \exp\left(\theta G_\mu(z)+\frac{\theta}{N} \Delta G_N(z)+\frac{\theta^2/2-\theta}{N}\cdot \frac{\partial}{\partial z}G_N(z)+O\left(\frac1{N^2}\right)
  \right).
\end{multline}

Recalling the definition of $R_\mu(z)$, we conclude that the function $R_N(Nz)$ from Theorem
\ref{Theorem_discrete_loop} for $\Pp_{N}^{\t,\v}$ can be written in the following form
\begin{multline} \label{eq_x14}
 R_N(Nz)=\prod_{a=1}^m (v_a-z)(t_a+v_a-z)\cdot \Bigl[ R_\mu(z)\\ +
 \phi^{-}(z)e^{\theta G_\mu (z)} \E_{\Pp_{N}^{\t,\v}} \left( \exp\left( \frac{\theta}{N} \Delta G_N(z) \right) -1 \right)
 +\phi^{+}(z)e^{-\theta G_\mu (z)} \E_{\Pp_{N}^{\t,\v}} \left( \exp\left(-\frac{\theta}{N} \Delta G_N(z) \right) -1 \right)
\\ + \frac{\psi^-_N(z)}{N} e^{-\theta G_\mu(z)} +\frac{\psi^+_N(z)}{N} e^{\theta
G_\mu(z)}
\\+
\dfrac{\phi^-(z) e^{-\theta G_\mu(z)}}{N}  \frac{\theta^2}{2} \frac{\partial}{\partial z}G_\mu(z) +
 \dfrac{\phi^+(z) e^{\theta G_\mu(z)}}{N}
\left(\frac{\theta^2}{2}-\theta\right) \frac{\partial}{\partial z}G_\mu(z)\Biggr]
+o\left(\frac{1}{N}\right).
\end{multline}
We further want to simplify the expression in the second line of \eqref{eq_x14}, by replacing
$e^{h}-1$ by $h$ under expectations. As in Theorem \ref{Theorem_main_1}, for that we use the
inequality
$$
 |e^h-h-1|\le |h|^2 e^{|h|},\quad h\in\mathbb C,
$$
and the fact that
$$
 \E_{\Pp_{N}^{\t,\v}}\left( \left|\frac{\theta}{N} \Delta G_N(z) \right|^2 \exp\left( \left|\frac{\theta}{N} \Delta
 G_N(z)\right|\right) \right)=o\left(\frac1N\right), \quad N\to\infty,
$$
which is established via self--improving estimates.

We therefore can rewrite \eqref{eq_x14} as
\begin{multline} \label{eq_x15}
 R_N(Nz)=\prod_{a=1}^m (v_a-z)(t_a+v_a-z)\cdot \Bigl[ R_\mu(z)
 -\frac{\theta Q_\mu(z)}{N} \E_{\Pp_{N}^{\t,\v}} \left(  \Delta G_N(z) \right)
 + \frac{\psi^-_N(z)}{N} e^{-\theta G_\mu(z)} +\frac{\psi^+_N(z)}{N} e^{\theta
G_\mu(z)}
\\+
\dfrac{\phi^-(z) e^{-\theta G_\mu(z)}}{N}  \frac{\theta^2}{2} \frac{\partial}{\partial z}G_\mu(z) +
 \dfrac{\phi^+(z) e^{\theta G_\mu(z)}}{N}
\left(\frac{\theta^2}{2}-\theta\right) \frac{\partial}{\partial z}G_\mu(z)\Biggr]
+o\left(\frac{1}{N}\right),
\end{multline}
where the remainder is uniform over $z$ in compact subsets of $\mathbb
C\setminus\bigcup_{i=1}^k [\hat a_i,\hat b_i]$.

Let us now fix $u$ outside the contours $\gamma_1,\dots,\gamma_k$, divide \eqref{eq_x15} by $2\pi
\i  \prod_{a=1}^m (v_a-z)(t_a+v_a-z) \cdot H(z) \cdot (u-z)$ (the definition of $H(z)$ is given in
Assumption \ref{Assumption_simple}) and integrate over the union of contours $\gamma_i$. We get
\begin{multline} \label{eq_x16}
 \frac{\theta}{2\pi\i} \oint_{\bigcup_{i=1}^k \gamma_i}
 \frac{\prod_{i=1}^k\sqrt{(z-\alpha_i)(z-\beta_i)}}{u-z}
 \cdot \E_{\Pp_{N}^{\t,\v}} \left(  \Delta G_N(z)
 \right) dz \\ = \frac{1}{2\pi\i} \oint_{\bigcup_{i=1}^k \gamma_i} \frac{dz}{(u-z)H(z)}
  \cdot\Biggl(\psi^-_N(z) e^{-\theta G_\mu(z)} +\psi^+_N(z) e^{\theta
G_\mu(z)}
\\+
\phi^-(z) e^{-\theta G_\mu(z)}  \frac{\theta^2}{2} \frac{\partial}{\partial z}G_\mu(z) +
 \phi^+(z) e^{\theta G_\mu(z)}
\left(\frac{\theta^2}{2}-\theta\right) \frac{\partial}{\partial z}G_\mu(z) \biggr) \\
+\frac{1}{2\pi\i}\oint_{\bigcup_{i=1}^k \gamma_i}\left( \frac{- N
R_N(Nz)}{(u-z)H(z)}+ \frac{N R_\mu(z)}{(u-z)H(z)} +o\left(1\right) \right)dz.
\end{multline}
Since $H(z)$ is non-zero on all the segments $[\hat a_i,\hat b_i]$, by a proper choice of small
contours $\gamma_i$ we can guarantee that $H(z)$ is non-zero inside the integration contours. Then
the terms in the last line of \eqref{eq_x16} vanish as they have no singularities inside the
integration contours.

Turning to the first line of \eqref{eq_x16}, note that $\E_{\Pp_{N}^{\t,\v}} \left(
\Delta G_N(z) \right)$ is analytic outside the contours of integration and decays as
$1/z^2$ when $z\to\infty$. Therefore, we can compute the integral as (minus) the sum
of the residues at $z=u$ and at $z=\infty$. The former is
$$
  \theta
 \prod_{i=1}^k\sqrt{(u-\alpha_i)(u-\beta_i)}
 \cdot \E_{\Pp_{N}^{\t,\v}} \left(  \Delta G_N(u) \right),
$$
while the latter is a polynomial $P_N(u)$ of degree at most $k-2$ (to see that one uses
$(z-u)^{-1}=z^{-1}\sum_{\m\ge 0} (uz)^m$ ). We conclude that
\begin{multline} \label{eq_x17}
 \theta\cdot \E_{\Pp_{N}^{\t,\v}} \left(  \Delta G_N(u)\right) +
 \frac{P_N(u)}{\prod_{i=1}^k\sqrt{(u-\alpha_i)(u-\beta_i)}}
\\ = \frac{1}{2\pi\i \cdot {\prod_{i=1}^k\sqrt{(u-\alpha_i)(u-\beta_i)}} } \oint_{\bigcup_{i=1}^k \gamma_i} \frac{dz}{(u-z)H(z)}
  \cdot\Biggl(\psi^-_N(z) e^{-\theta G_\mu(z)} +\psi^+_N(z) e^{\theta
G_\mu(z)}
\\+
\phi^-(z) e^{-\theta G_\mu(z)} \frac{\theta^2}{2} \frac{\partial}{\partial z}G_\mu(z) +
 \phi^+(z) e^{\theta G_\mu(z)}
\left(\frac{\theta^2}{2}-\theta\right) \frac{\partial}{\partial z}G_\mu(z)\Biggr) +o(1).
\end{multline}
Now we are in a position to apply the map $\Upsilon_u$. Indeed, the integral of $G_N(z)$ around
$\gamma_i$ is deterministic and equals $n_i(N)/N$. On the other hand, the integral of $G_\mu(z)$
around $\gamma_i$ equals the total mass of $\mu(x)$ inside $\gamma_i$, which is $\hat n_i$. Since,
$\hat n_i=n_i(N)/N$, the integral of $\E_{\Pp_{N}^{\t,\v}} \left(  \Delta G_N(u)\right)$ around
each loop $\gamma_i$ vanishes.

Therefore,
\begin{multline} \label{eq_x18}
 \E_{\Pp_{N}^{\t,\v}}  \left(  \Delta G_N(u)\right)
 \\=
 \Upsilon_u\Biggl[ \frac{\theta^{-1}}{2\pi\i \cdot
{\prod_{i=1}^k\sqrt{(u-\alpha_i)(u-\beta_i)}} } \oint_{\bigcup_{i=1}^k \gamma_i}
\frac{dz}{(u-z)H(z)}
  \cdot\Biggl(\psi^-_N(z) e^{-\theta G_\mu(z)} +\psi^+_N(z) e^{\theta
G_\mu(z)}
\\+
\phi^-(z) e^{-\theta G_\mu(z)} \left( \frac{\theta^2}{2} \frac{\partial}{\partial
z}G_N(z)(z)\right) +
 \phi^+(z) e^{\theta G_\mu(z)} \left(
(\theta^2/2-\theta) \frac{\partial}{\partial z}G_\mu(z)(z)\right)\Biggr)\Biggr]
+o(1) .
\end{multline}
It remains to show that \eqref{eq_x18} can be differentiated with respect to
$t_1,\dots,t_m$ and $t_1=\dots=t_m=0$, which is done in the same way as in Theorem
\ref{Theorem_main_1}, see Section \ref{Section_estimates} for the details. It is
important to note here that due to its explicit definition, the map $\Upsilon_u$ is
continuous and even Lipshitz (in uniform norm on the contours $\gamma_i$,
$i=1,\dots,k$). Moreover, it is linear and does not depend on $\t$. Therefore, its
appearance in the formulas does not affect the argument.
\end{proof}

\section{Central Limit Theorem}

The following two theorems are  corollaries of Theorem \ref{Theorem_main_general},
cf.\ Section \ref{Section_Second_order_1}.

\begin{theorem} \label{Theorem_CLT} Under Assumptions
\ref{Assumptions_basic}--\ref{Assumption_vanishing}
the joint moments of the random variables $N\bigl(G_N(z)-\E_{\Pp_{N}} G_N(z)\bigr)$ approximate
(uniformly in $z$ in compact subsets of $\mathbb C\setminus \bigcup_{i=1}^k [\hat a_i, \hat b_i]$)
those of centered Gaussian random variables with covariance
$$
 N^2 \biggl(  \E_{\Pp_{N}}  \bigl( G_N(z) G_N(w) \bigr)- \E_{\Pp_{N}}  G_N(z)\,
 \E_{\Pp_{N}} G_N(w)
 \biggr)
=\frac{1}{2\theta} \, {\mathfrak C} (z,w) + o(1), \quad N\to \infty,
$$
where
\begin{multline} \label{eq_limit_cov_general} {\mathfrak C} (z,w)=
 \Upsilon_w\Biggl[- \frac{1}{(w-z)^2}
+\frac{{\sqrt{\prod_i(z-\alpha_i)(z-\beta_i)}}}{\sqrt{\prod_i(w-\alpha_i)(w-\beta_i)}}\Biggl(\frac{1}{(z-w)^2}\\-\frac{1}{2(z-w)}\sum_{i=1}^k
\left(\frac{1}{z-\alpha_i}+\frac{1}{z-\beta_i}\right)\Biggr)\Biggr].
\end{multline}
\end{theorem}
\begin{remark} The map $\Upsilon$ was defined in Section \ref{Section_second_order}.
In cases $k=1$ and $k=2$ the resulting covariance function ${\mathfrak C} (z,w)$ can
be brought to a more explicit form. When $k=1$,
$$
{\mathfrak C} (z,w)= -\frac{1}{(z-w)^2} \left(1-\frac{zw
-\frac{1}{2}(\alpha_1+\beta_1)(z+w)+\alpha_1\beta_1}
 {\sqrt{(z-\alpha_1)(z-\beta_1)}\sqrt{(w-\alpha_1)(w-\beta_1)}}\right).
$$
The expression for $k=2$ involves the values of complete elliptic integrals and we
do not pursue it here, cf.\ \cite[Section 3]{BDE}.
\end{remark}
\begin{remark}
 Note that the numbers $\alpha_i, \beta_i$ may depend on $\mu(x)$, which, in turn,
 may depend on $N$ through filling fractions $\hat n_i=n_i(N)/N$. However, this
 dependence is continuous (see Proposition \ref{Prop_dependence_on_filling}) and
 thus, if $\hat n_i$ converges as $N\to\infty$, then so is ${\mathfrak C} (z,w)$, and
 Theorem \ref{Theorem_CLT} turns into a conventional central limit theorem for
 $NG_N(z)$ as $N\to\infty$. In particular, if $k=1$, then this is always the case.
\end{remark}
\begin{remark} In the general beta random matrix models the covariance has precisely
the same form, cf.\ \cite{Johansson_shape}, \cite{Shch}, \cite{BoG2}.
\end{remark}

\begin{corollary} \label{Corollary_CLT_linear_general}  Take $m\ge 1$ real functions
$f_1(z),\dots,f_m(z)$ on $\cup_{i=1}^k[\hat a_i,\hat b_i]$ that can be extended to holomorphic
functions in a complex neighborhood $\B$ of $\cup_{i=1}^k[\hat a_i,\hat b_i]$. Under Assumptions
\ref{Assumptions_basic}--\ref{Assumption_vanishing}, as $N\to\infty$ the joint moments of the $m$
random variables
$$
{\mathcal L}_{f_j}= \sum_{i=1}^{N} \bigl( f_j(\ell_i)-\E_{\Pp_{N}}f_j(\ell_i)
\bigr),\quad (\ell_1,\dots,\ell_N)
 \text{ is } \Pp_{N}\text{--distributed,}
$$
approximate those of centered Gaussian random variables with asymptotic covariance
\begin{equation}
\label{eq_linear_covariance}
\E_{\Pp_{N}} {\mathcal L}_{f_i} {\mathcal L}_{f_j}=
\frac{1}{(2\pi\i)^2}
 \oint_{\cup_i \gamma_i} \oint_{\cup_i \gamma_i} f_i(u) f_j(v) {\mathfrak C}(u,v)\, du\, dv +o(1),
\end{equation}
where $\gamma_i$, $i=1,\dots,k$, are positively oriented contours in $\B$ that enclose $[\hat
a_i,\hat b_i]$, respectively, and ${\mathfrak C}(u,v)$ is given by \eqref{eq_limit_cov_general}.
\end{corollary}
\begin{remark}
 Similarly to Corollary \ref{Corollary_CLT_linear_general},
 Theorem \ref{Theorem_main_general} can be used to obtain the first two terms in the asymptotic
 expansion of $\sum_{i=1}^{N} \E_{\Pp_{N}}f(\ell_i)$ for functions $f$
 holomorphic in a neighborhood of $\cup_{i=1}^k[\hat a_i,\hat b_i]$.
\end{remark}

\section{Non--vanishing weights}

\label{Section_non_van}

Assumption \ref{Assumption_vanishing} of the general setup of Section \ref{Section_setup} was
vanishing of the weight at the end--points of the supporting intervals. For future applications it
is convenient to relax this condition and replace it by the following exponential bound on
probabilities of having particles at $a_i(N)+1$ or $b_i(N)-1$.

\begin{assumption} \label{Assumption_endpoints}
 We require the existence of constants $C_1$, $C_2$, $C_3>0$ such that for all $N=1,2,\dots,$
 the $\P$--probability
 of the event
 $$
  \ell_j=a_i(N)+1\text{ or } \ell_j=b_i(N)-1\text { for at least one pair } 1\le j \le
  N,\; 1\le i\le k
 $$
 is bounded from above by $C_1\exp(-C_2 N^{C_3})$.
\end{assumption}

\begin{theorem} \label{Theorem_applies_to_nonvan}
 If we replace Assumption \ref{Assumption_vanishing} by Assumption
 \ref{Assumption_endpoints}, then the results of Theorem \ref{Theorem_main_general}, Theorem \ref{Theorem_CLT},
 Corollary \ref{Corollary_CLT_linear_general} are still valid
 for measures $\P$.
\end{theorem}

The proof of Theorem \ref{Theorem_applies_to_nonvan} is based on the following observation.

\begin{proposition} \label{Prop_discrete_loop_exponent}
 In the notations of Theorem \ref{Theorem_discrete_loop},
 suppose that the condition  $\phi^-_N((a_i(N)+1)=\phi^+_N(b_i(N))=0$, $i=1,\dots,k$, does not hold.
 If $\phi^\pm_N(\xi)$ are holomorphic in a domain $\mathcal M_N \subset\mathbb C$, then so is
$R_N(\xi)$ except for at most $2k$ simple poles. These poles are at points $\{a_i(N)+1,
b_i(N)-1\}_{i=1}^k$. Under Assumption \ref{Assumption_endpoints} the corresponding residues decay
exponentially in $N$ as $N\to\infty$ in the same sense as in the bound of Assumption
\ref{Assumption_endpoints} (perhaps, with different constants $C_1$, $C_2$, $C_3$).
\end{proposition}
\begin{proof} We repeat the proof of Theorem \ref{Theorem_discrete_loop} and observe
the same cancelation of the poles. The only poles for which the cancelations do not occur, are
endpoints of the interval: in Theorem \ref{Theorem_discrete_loop} the functions $\phi^\pm$ were
vanishing at these endpoints, but this is no longer the case. The residue at an end-point can be
bounded for the first term in \eqref{eq_observable} by the probability to have a particle at such
an end-point (denote it by $m$ ) multiplied by
$$
 \max_{\ell_i} \left|\prod_{i:\, \ell_i\ne m}\left(1-\frac{\theta}{m-\ell_i}\right)\right|\le
 \left|\prod_{i=1}^{N-1}\left(1+\frac{\theta}{\theta i}\right)\right|=N.
$$
Thus, the exponential decay of the probability in Assumption \ref{Assumption_endpoints} implies the
exponential decay of the residue. For the second term in \eqref{eq_observable} the argument is the
same.
\end{proof}

\begin{proof}[Proof of Theorem \ref{Theorem_applies_to_nonvan}]
Note that an analogue of Proposition \ref{Prop_discrete_loop_exponent} is readily established also
for the deformed measures ${\Pp_{N}^{\t,\v}}$ as in Section \ref{Section_second_order}. Indeed, we
need such measures only for small (i.e.\ tending to $0$) values of $t_i$, but then the exponential
bounds on probability remain valid. Therefore, all the arguments of Section
\ref{Section_second_order} go through for the measures $\P$. Indeed, the only difference between
Theorem \ref{Theorem_discrete_loop} and Proposition \ref{Prop_discrete_loop_exponent} is in the
appearance of finitely many simple poles. However, since the residues are exponentially small,
these poles will only add exponentially small terms to all the remainders and thus will not
contribute to the expansions in powers of $1/N$ that we study.
\end{proof}

\section{Examples}

\label{Section_examples}

The aim of this section is to demonstrate how the Assumptions
\ref{Assumptions_basic}--\ref{Assumption_simple} are checked in applications, which yields the
validity of Theorem \ref{Theorem_main_general}, Theorem \ref{Theorem_CLT}, and Corollary
\ref{Corollary_CLT_linear_general} for certain stochastic systems.

\subsection{Multi-cut general $\theta$ extension of the Krawtchouk ensemble}
\label{Section_multi_Krawtchouk} The first example is an extension of that of
Section \ref{Section_toy} to general values of $\theta$ and $k$.

We fix $k=1,2,\dots$ and take $3k$ numbers $\hat a_i$, $\hat b_i$, $\hat n_i$, such
that
$$
 \hat a_1<\hat b_1<\hat a_2<\hat b_2<\dots<\hat a_k <\hat b_k,
$$
 $0<\hat n_i < \theta^{-1} (\hat b_i-\hat a_i)$ for all $i=1,\dots,k$, and
 $\sum_{i=1}^k \hat n_i=1$. Then we further choose numbers $a_i(N)$, $b_i(N)$,
 $n_i(N)$ such that the model fits into the setup of Section \ref{Section_setup}.
 The weight $w(x;N)$ is then defined through the identity
 \begin{equation}
 \label{eq_x48}
  \frac{w(x;N)}{w(x-1;N)}=-\prod_{i=1}^k \frac{x-b_i(N)}{x+1-a_i(N)}.
 \end{equation}
 Note that \eqref{eq_x48} agrees with the conditions $w(x;N)>0$, $x\in\cup_{i=1}^k
 [a_i(N)+1,b_i(N)-1]$ and with Assumption \ref{Assumption_vanishing}. Since the weight $w(x;N)$ is supported on $\cup_{i=1}^k
 (a_i(N),b_i(N))$, we need to supplement \eqref{eq_x48} by the choice of $k$
 constants
 \begin{equation} \label{eq_x49}
 \mathfrak c_i(N)=w(a_i(N)+1;N),\quad i=1,\dots,k.
 \end{equation}
 Observe that the multiplication of all $\mathfrak c_i$ by a same constant leaves
 the probability distribution unchanged. Therefore, for each $N$ the system depends
 on the choice of $4k-2$ constants. In particular, if $k=1$ and $\theta=1$, then (up to a shift
 of the lattice) we arrive at the example of Section \ref{Section_toy}.

 If we now assume that all the parameters are chosen so that as $N\to\infty$
 $$
  a_i(N)=N\hat a_i +O(1),\quad b_i(N)= N\hat b_i +O(1),\quad n_i(N)=N \hat
  n_i+O(1),
 $$
 $$
  \mathfrak c_i(N)=\exp(N \hat {\mathfrak c}_i),\quad i=1,\dots,k,
 $$
 then using Stirling's formula for the factorials appearing in the explicit
 expressions for $w(x;N)$, it is easy to see that the model satisfies Assumptions
 \ref{Assumptions_basic}, \ref{Assumption_filling}, \ref{Assumptions_ratio}.
 Further,
 $$
  \phi^+(z)=-\prod_{i=1}^k (z-\hat b_i),\quad \phi^-(z)=\prod_{i=1}^k (z-\hat a_i).
 $$
 The function $R_\mu(z)$ is an analytic function in $z$, which is $O(z^{k-1})$ as
 $z\to\infty$. Therefore, by Liouville's theorem $R_\mu(z)$ is a polynomial of degree
 at most $k-1$. Hence, the quadratic equation \eqref{eq_quadratic_on_Q} implies that
 $Q_\mu(z)$ is the square root of a degree $2k$ polynomial. In other words,
 decomposing into linear factors we get
 \begin{equation}
 \label{eq_Q_multi_Kr_form}
  Q_\mu(z)=2 \sqrt{\prod_{i=1}^k (z-\alpha_i)(z-\beta_i)}.
 \end{equation}
 On the other hand, our choice of $\hat n_i$ guarantees, that the equilibrium measure $\mu(x)dx$
 in each interval $[\hat a_i, \hat b_i]$ has some points where $\mu(x)>0$
 and some points where $\mu(x)<\theta^{-1}$. Combining this with the observation (that is immediate from \eqref{eq_def_Q}) that
 $Q_\mu(\hat a_i)$ and $Q_\mu(\hat b_i)$ have different signs for each
 $i=1,\dots,k$, we conclude that $\alpha_i,\beta_i\in[\hat a_i,\hat b_i]$. For each
 $i$ there are two options: either $\mu(x)$ has a band inside $[\hat a_i,\hat b_i]$,
 and then $\alpha_i$ and $\beta_i$ are endpoints inside this band, or $[\hat
 a_i,\hat b_i]$ is a union of a void and a saturated region. Let us explain that the latter is
 impossible. Indeed, then there must be a point $x\in\mathbb R$ and $\eps>0$ such that $\mu(x)=0$ on
 $(x-\eps,x)$ and $\mu(x)=\theta^{-1}$ on $(x,x+\eps)$ (or vice versa which is considered in the
 same way). Therefore, as $z$ approaches $x$ along the real axis from the left, the Stieltjes
 transform $G_\mu(z)$ explodes:
 $$
  \lim_{z\to x-} G_\mu(z)=\lim_{z\to x-} \int \frac{\mu(t)}{z-t} dt = -\infty.
 $$
 This implies $\lim_{z\to x-} \exp(G_\mu(z))=0$ and $\lim_{z\to x-} \exp(-G_\mu(z))=+\infty$. But
 then the definition \eqref{eq_def_R} of $R_\mu(z)$ implies that $R_\mu(z)$ has a singularity at
 $z=x$, which contradicts the fact that $R_\mu(z)$ is a polynomial.
 Note that it is crucial in this argument that $\phi^{\pm}$ does not have zeros
 inside $\bigcup( \hat a_i, \hat b_i)$, and indeed the end--points $\hat a_i$,
 $\hat b_i$ might separate voids and saturated regions.

 The conclusion is that this class of
 probability models satisfies Assumptions \ref{Assumptions_basic}--\ref{Assumption_vanishing} and
 Theorem \ref{Theorem_main_general}, Theorem \ref{Theorem_CLT},
 Corollary \ref{Corollary_CLT_linear_general} are valid for them.
 As far as we know, these results are new with the exception of the case
 $\theta=k=1$.

\subsection{Lozenge tilings}
\label{Section_hex} The second example demonstrates that at least two instances of probability
measures arising in the study of uniformly random lozenge tilings fit into our framework.

Consider an $A\times B\times C$ hexagon drawn on the regular triangular lattice. We
 tile this hexagon with three types of elementary lozenges (which are unions of
 adjacent triangular faces of the lattice), cf.\ Figure \ref{Fig_tiling_hex}.
 There are finitely many such tilings, and we
 are interested in the asymptotic behavior of uniformly random tiling as
 $A,B,C\to\infty$. This is a well-studied model, with many results available, cf.\
 \cite{CLP}, \cite{BKMM}, \cite{JN}, \cite{Gor}, \cite{BG}, \cite{Petrov_Airy},
 \cite{Petrov_GFF}.

\begin{figure}[t]
\center \scalebox{0.7}{\includegraphics{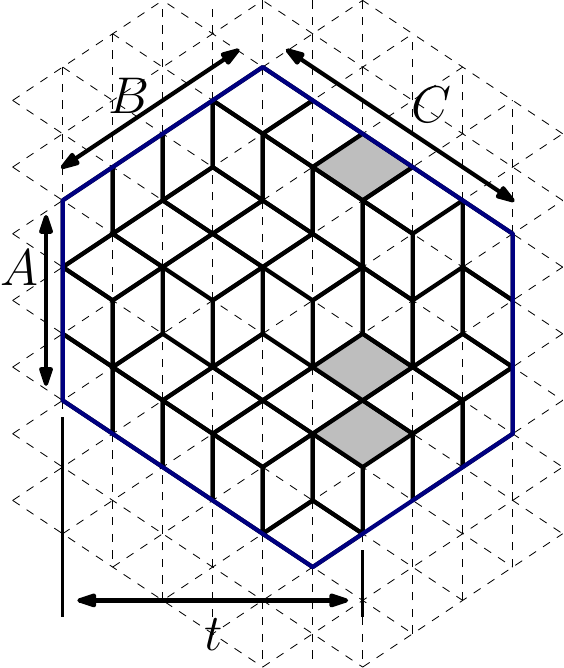}} \qquad
\scalebox{0.7}{\includegraphics{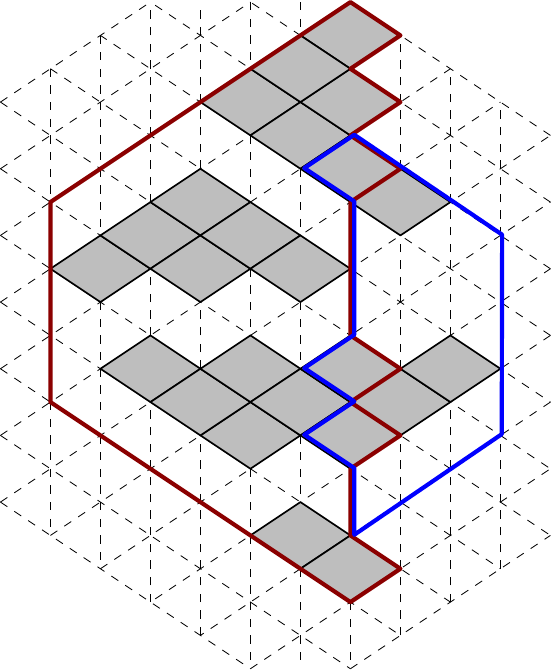}} \caption{Left panel: Lozenge tiling
of the $3\times 4\times 5$ hexagon and $3$ horizontal lozenges on the sixth from the
left vertical line. Right panel: two Gelfand--Tsetlin patterns (inside blue and red
contours) corresponding to each tiling. \label{Fig_tiling_hex}}
\end{figure}

Let us dissect the hexagon by a vertical line at distance $t$ from the left boundary. There will be
a fixed (depending on $A,B,C,t$) number $N$ of horizontal lozenges on this line;  let $\P$ denote
the distribution of these lozenges. This distribution can be computed by noticing that the tiling
can be viewed as two \emph{Gelfand--Tsetlin patterns} glued together, as shown in the right panel
of Figure \ref{Fig_tiling_hex}. The enumeration of Gelfand--Tsetlin patterns is well--known and can
be used to compute the distribution $\P$ (see \cite{CLP}, \cite{Gor}, \cite{BP_lectures},
\cite{BuGo} for more details). Assuming $t>\max(B,C)$ as in Figure \ref{Fig_tiling_hex}, which
yields $N=B+C-t$, and introducing the coordinate system such that the lowest possible position for
horizontal lozenges on the $t$th vertical line is $1$ and the highest one is $A+B+C-t$, we obtain
the formula
\begin{equation}
\label{eq_Hahn}
 \P(\ell_1,\dots,\ell_N)=\frac{1}{Z_N} \prod_{i<j} (\ell_i-\ell_j)^2 \prod_{i=1}^N
 \biggl[ (A+B+C+1-t-\ell_i)_{t-B} \, (\ell_i)_{t-C} \biggr],
\end{equation}
where $(a)_n$ is the Pochhammer symbol, $(a)_n=a(a+1)\cdots (a+n-1)$, and $Z_N$ is a normalizing
constant (which can be computed explicitly in this case). The distribution of the form
\eqref{eq_Hahn} is known as \emph{Hahn orthogonal polynomial ensemble}.

Take a large parameter $L$ and suppose that
\begin{equation}
\label{eq_hex_scaling}
 A=\hat A L +O(1),\quad B=\hat B L+ O(1),\quad C=\hat C L+O(1),\quad t=\hat t L+
 O(1),
\end{equation}
 for positive constants $\hat A$, $\hat B$, $\hat C$, $\hat t$. We assume that $\hat t>
 \max(\hat B,\hat C)$ --- other possibilities for $t$ are considered similarly. Let us check
 that under such choice of parameters the ensemble \eqref{eq_Hahn} satisfies
 Assumptions \ref{Assumptions_basic} - \ref{Assumption_vanishing}.
 \begin{itemize}
  \item Assumption \ref{Assumptions_basic} is satisfied due to Stirling's formula applied to
  Pochhammer symbols. The potential $V(u)$ has the form:
  \begin{multline*}
   V(u)=-(\hat A+\hat C-u)\ln(\hat A+\hat C-u)+(\hat A + \hat B + \hat C -\hat
   t-u)\ln(\hat A + \hat B + \hat C -\hat t-u)\\-(\hat t-\hat C +u)\ln(\hat t-\hat C
   +u)+u\ln(u).
  \end{multline*}
  \item Assumption \ref{Assumption_filling} is empty, as there is only one filling fraction,
  $n_1(N)=N$.
  \item Assumption \ref{Assumptions_ratio} is immediate from the definitions, and we have
 $$
   \phi^+(z)= (\hat t-\hat C +z) (\hat A+\hat B+\hat C-\hat
   t-z), \quad \phi^-(z)=  z (\hat A +\hat C -z).
$$
  \item For Assumption \ref{Assumption_simple} note that $R_\mu(z)$ is an analytic function which
  grows as $O(z^2)$ as $z\to\infty$ and therefore it is a polynomial of degree at most two.
  Hence, by \eqref{eq_quadratic_on_Q}, the function $Q_\mu(z)$ is a square root of a
  polynomial. The definition implies that $Q_\mu(z)$ is $O(z)$ as $z\to\infty$. Thus $Q_\mu(z)$ is
  a square root of degree two polynomial and has the form
  $$
   Q_\mu(z)={\rm const} \cdot \sqrt{(z-a)(z-b)}.
  $$
  The points $a$ and $b$ are necessarily endpoints of the band of the equilibrium measure $\mu(x)dx$.
  \item Assumption \ref{Assumption_vanishing} is immediate from the formula
  \eqref{eq_Hahn}.
 \end{itemize}
The conclusion is that Theorem \ref{Theorem_main_general}, Theorem \ref{Theorem_CLT},
 Corollary \ref{Corollary_CLT_linear_general} are valid.
 This implies a form of the Central Limit Theorem for fluctuations of lozenge tilings of a hexagon.
 We remark that the same CLT (and even stronger statement concerning joint asymptotic Gaussianity
 for several values of $\hat t$) can be also established by other methods, cf.\
 \cite{Petrov_GFF},  \cite{BD}, \cite{BuGo2}.

\begin{remark} One can probably use the fact that $R_\mu(z)$ is a degree two polynomial
to find an explicit formula for $R_\mu(z)$ and thus, also for $G_\mu(z)$, and for the equilibrium
measure $\mu(x)dx$ describing the \emph{limit shape} for lozenge tilings, cf.\ Section
\ref{Section_zw-measures} for a similar argument. Explicit formulas for $\mu(x)$ were previously
found by other methods in \cite{CLP}, \cite{BKMM}, \cite{KO}, \cite{Gor}, \cite{Petrov_Airy}.
\end{remark}

\begin{figure}[t]
\center \scalebox{0.7}{\includegraphics{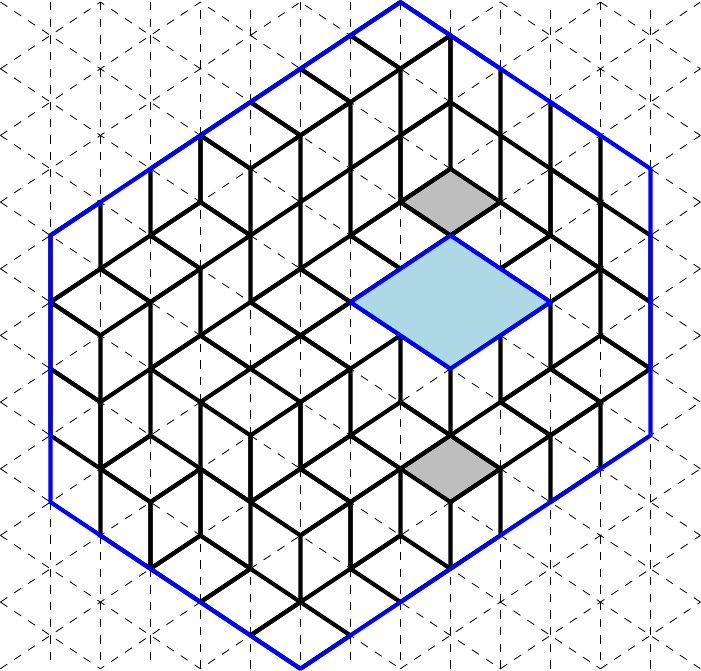}} \caption{Lozenge tiling of the $4\times
7\times 5$ hexagon with a rhombic $2\times 2$ hole (shown in blue). The center of the hole is at
distance $t$ from the left side of the hexagon. The remaining horizontal lozenges on $t$th vertical
line are shown in gray. \label{Fig_tiling_hex_2}}
\end{figure}

One can also analyze tilings of more complicated domains. Let us cut a rhombic $D\times D$ hole in
the hexagon, as shown in Figure \ref{Fig_tiling_hex_2}. Assume that the bottom point of the hole is
at distance $t$ from the left side of the hexagon and at height $H$ (counted from the bottom of the
hexagon along $t$th vertical line). Let $\P$ be the probability distribution of the horizontal
lozenges (outside the hole) on $t$th vertical line induced by the uniform measure on all tilings of
the hexagon with the hole. We can repeat the same argument as the one used for the complete hexagon
for computing $\P$. Assuming $t>\max(B,C)$ as in Figure \ref{Fig_tiling_hex}, which yields
$N=B+C-D-t$, and introducing the coordinate system such that the lowest possible position for
horizontal lozenge on the $t$th vertical line is $1$ and the highest one is $A+B+C-t$, we obtain
\begin{equation}
\label{eq_Hahn_cut}
 \P(\ell_1,\dots,\ell_N)=\frac{1}{Z_N} \prod_{1\le i<j\le N} (\ell_i-\ell_j)^2 \prod_{i=1}^N
 \biggl[ (A+B+C+1-t-\ell_i)_{t-B} \, (\ell_i)_{t-C}\, (H-\ell_i)_D\, (H-\ell_i)_D  \biggr].
\end{equation}
We also need two filling fractions $n_1$ and $n_2$: we consider only such tilings that there are
$n_1$ horizontal lozenges (on $t$th vertical line) below the hole and $n_2$ lozenges above. Take a
large parameter $L$ and suppose that in addition to \eqref{eq_hex_scaling} we have
$$
 H=\hat H L +O(1),\quad D=\hat D L+ O(1), \quad \hat H>0,\; \hat D>0,\; \hat H+\hat D< \hat A+\hat
 B+\hat C-\hat t,
$$
$$
 n_1=\hat n_1 L +O(1),\quad n_2=\hat n_2 L+ O(1), \quad 0<\hat n_1<\hat H,\; 0<n_2<\hat A+\hat B +\hat C-\hat t-\hat H -\hat D.
$$
We again assume that $\hat t>
 \max(\hat B,\hat C)$ and remark that other possibilities for $t$ are considered similarly. Let us check
 that under such choice of parameters the ensemble \eqref{eq_Hahn_cut} satisfies
 Assumptions \ref{Assumptions_basic} - \ref{Assumption_vanishing}.
 \begin{itemize}
  \item Assumption \ref{Assumptions_basic} is satisfied due to Stirling's formula applied to
  Pochhammer symbols.
  \item Assumption \ref{Assumption_filling} is satisfied due to restrictions on $\hat n_1$, $\hat
  n_2$.
  \item Assumption \ref{Assumptions_ratio} is immediate from the definitions, and we have
  $$
   \phi^+(z)= (\hat t-\hat C +z) (\hat A+\hat B+\hat C-\hat
   t-z) (\hat H-z)^2 ,\quad \phi^-(z)=z (\hat A +\hat C -z) (\hat H+\hat D-z)^2  .
  $$
  \item For Assumption \ref{Assumption_simple} note that $R_\mu(z)$ is an analytic function which
  grows as $O(z^4)$ as $z\to\infty$ and therefore it is a polynomial of degree at most four.
  Hence, by \eqref{eq_quadratic_on_Q}, the function $Q_\mu(z)$ is a square root of a
  polynomial. The definition implies that $Q_\mu(z)$ is $O(z^3)$ as $z\to\infty$. Thus, $Q_\mu(z)$ is
  a square root of a degree six polynomial. Further observe that according to definitions
  $Q_\mu(0)<0$, $Q_\mu(\hat H)>0$, $Q_\mu(\hat H+\hat D)<0$, $Q_\mu(\hat A+\hat B+\hat C-\hat
  t)>0$. Together with analyticity of $Q_\mu(z)$ outside the support of the equilibrium measure
  this yields
  $$
   Q_\mu(z)={\rm const}\, (z-c) \sqrt{(z-a_1)(z-b_1)(z-a_2)(z-b_2)},
  $$
  where
  $$
   0< a_1\le b_1 < \hat H <c < \hat H+\hat D < a_2<b_2<\hat A+\hat B+\hat C-\hat t.
  $$
  As before, the points $a_1,b_1,a_2,b_2$ can be identified with the endpoints of the bands of the
  equilibrium measure.
   \item Assumption \ref{Assumption_vanishing} is immediate from the formula
  \eqref{eq_Hahn_cut}.
 \end{itemize}
The conclusion is that Theorem \ref{Theorem_main_general}, Theorem \ref{Theorem_CLT},
 Corollary \ref{Corollary_CLT_linear_general} are valid.
 This implies a form of the Central Limit Theorem for fluctuations of lozenge tilings of a hexagon with a
 hole (with fixed filling fractions above and below the hole). As far as we know, these results are new; the same applies to  the
  examples of the next sections.

\subsection{Arbitrary convex potential on $\mathbb R$ with no saturation}

\label{Section_convex}

For our third example take a real \emph{convex} analytic function $V(x)$, i.e.\ such that $V''(x)>0$ for all $x\in
\mathbb R$. Fix a constant $\kappa>0$ such that
\begin{equation}
\label{eq_potential_behavior}
 \liminf_{x\to\infty} \frac{\kappa V(x)}{2\theta \ln |x|}>1
\end{equation}
  and consider a probability distribution
\begin{equation}
\label{eq_convex_distribution}
 \P(\ell_1,\dots,\ell_N)= \frac{1}{Z_N}
 \prod_{1\le i<j \le N}
 \frac{\Gamma(\ell_j-\ell_i+1)\Gamma(\ell_j-\ell_i+\theta)}{\Gamma(\ell_j-\ell_i)\Gamma(\ell_j-\ell_i+1-\theta)}
 \prod_{i=1}^N \exp\left( - \kappa N  \cdot V\left(\frac{\ell_i}{N}\right)\right)
\end{equation}
on $N$-tuples $\ell_1<\ell_2<\dots<\ell_N$ such that
$$
 \ell_i=\lambda_i+ \theta i,\quad \lambda_1\le\lambda_2\le \dots\le \lambda_N,\quad
 \lambda_i\in\mathbb Z.
$$
In other words, we are in the framework of Section \ref{Section_setup} except that now the
distribution is supported on an \emph{infinite and unbounded} subset of $\mathbb R^N$. In
particular, when $V(x)=x^2$, we obtain a discretization of the celebrated general $\beta$ Gaussian
ensemble from random matrix theory.

\begin{proposition} \label{Prop_maximizer}
 For any real analytic $V(x)$ satisfying $V''(x)>0$ and \eqref{eq_potential_behavior}, the functional $I_{\kappa V}$ of
 \eqref{eq_functional_general} has a unique maximizer (equilibrium measure) on the the set of all absolutely continuous probability measures on
 $\mathbb R$ (without any restrictions on the density). The equilibrium measure is compactly supported, has a continuous density $\mu(x)$ and
 has a single band. At the end--points of the band the density of the equilibrium
 measure behaves like $c\sqrt{x}$.
\end{proposition}
\begin{proof} This is well--known, see \cite{DS}, \cite{Johansson_CLT} and references therein. \end{proof}

Note that the equilibrium measure $\mu(x) dx$ of Proposition \ref{Prop_maximizer} is unchanged when
we multiply both $\theta$ and $\kappa$ on the same positive constant (as the functional is then
multiplied by the same constant) and, therefore, depends only on their ratio $\kappa/\theta$. In
particular, if we fix ratio and then choose small $\kappa$ (equivalently, small $\theta$), then the
density $\mu(x)$ will be smaller than $\theta^{-1}$ for all $\theta$. But then the solution to
constrained maximization problem with $0\le \mu(x)\le \theta^{-1}$ as in Theorem \ref{Theorem_LLN}
will be the same as the solution of the unconstrained minimization problem as in Proposition
\ref{Prop_maximizer}. The conclusion is that for small $\kappa$ the equilibrium measure has no
saturated regions; this is crucial for our considerations.

We now want to show that Theorem \ref{Theorem_main_general}, Theorem \ref{Theorem_CLT}, and
 Corollary \ref{Corollary_CLT_linear_general} apply in this situation. This is done by
 \emph{localizing} the probability measure $\P$ onto a finite interval with help of the following
 proposition.

 \begin{proposition} \label{Proposition_ld_support}
  Assume that $V(x)$ is convex and satisfies \eqref{eq_potential_behavior}. Then with exponentially high probability
  the measure \eqref{eq_convex_distribution} is supported on configurations in a linearly growing interval, i.e.\ there exist
  constants $C,D>0$ such that
 $$
  \P\left(-D\le \frac{\ell_1}{N} \le \frac{\ell_N}{N} \le D\right)>1- \frac{1}{C} \cdot \exp(-N C),\qquad N=1,2,\dots.
 $$
\end{proposition}
\begin{proof} This a particular case of Theorem \ref{Theorem_LDP_support}. We remark that statements of this flavor are well-known as large deviations principles for the
largest/smallest particles in both discrete and continuous log-gases, cf.\ \cite[Theorem
2.2]{Johansson_shape}, \cite[Theorem 4.2]{Feral}, \cite[Section 2.6]{AGZ}.
\end{proof}

Proposition \ref{Proposition_ld_support} motivates the definition of measure
$\widehat \P$ as $\P$ conditioned on the event that $| \ell_i/N |\le D+1$ for all
$i=1,\dots,N$. An advantage of the measure $\widehat \P$ is that $\ell_i/N$ are
$\widehat \P$--almost surely bounded, and therefore, there exists a finite complex
contour which enclose them all and we can apply the developments of previous
sections.

\begin{lemma} \label{Lemma_applies_to_convex}  Assume \eqref{eq_potential_behavior} and that analytic $V(x)$ satisfies $V''(x)>0$ for all $x\in\mathbb R$.
Then the results of Theorem \ref{Theorem_main_general}, Theorem \ref{Theorem_CLT},
 Corollary \ref{Corollary_CLT_linear_general} are valid for measures $\widehat \P$. Here $k=1$, the points
 $a_1$ and $b_1$ are two endpoints of the band of the equilibrium measure,
 $$
  \phi^-(z)=1,\quad \varphi^-_N(z)=0,\quad \phi^+(z)=\exp(-\kappa V'(z)),\quad \varphi^+_N(z)=\frac{\kappa V''(z)}{2}\exp(-\kappa
  V'(z))  .
 $$
\end{lemma}
\begin{proof} We will apply Theorem \ref{Theorem_applies_to_nonvan}, and we need to check all its assumptions.

  With $w(x;N)=\exp(-\kappa N V(x/N))$, we have
\begin{equation} \label{eq_x63}
 \frac{w(x;N)}{w(x-1;N)} = \exp\left[\kappa N
\left(V\left(\frac{\xi-1}{N}\right)-V\left(\frac{\xi}{N}\right) \right)\right].
\end{equation}
Therefore, the formulas for $\phi^+$ and $\varphi^+$ are obtained through large $N$ expansion of the right--hand side in \eqref{eq_x63}.

We only check that Assumption \ref{Assumption_simple} is satisfied, as the rest is automatic.
We have
$$
 Q_\mu(z)=\exp(-\theta G_\mu(z))-\exp(-\kappa V'(z)) \exp(\theta G_\mu(z)).
$$
As before, $Q_\mu(z)$ is a square root of an analytic function due to
\eqref{eq_quadratic_on_Q}. Further, $Q_\mu$ is analytic outside the single band of
the equilibrium measure $\mu(x)dx$. Near the end--points of the band, the density
$\mu(x)$ behaves like a square root and thus similar behavior for $Q_\mu(z)$.
Therefore,
$$
 Q_\mu(z)=H(z)\sqrt{(z-a)(z-b)},
$$
where $a<b$ are endpoints of the band. We will now show that $H(z)$ has no zeros on $\mathbb R$,
which would also imply that $H(z)$ is holomorphic, since $H^2(z)$ is.

We already ruled out $H(a)=0$ and $H(b)=0$, and there are 3 more cases to consider:
\begin{itemize}

\item $H(x)=0$, $a<x<b$. Then also $Q_\mu(x+\i 0)=Q_\mu(x-\i 0)=0$. Therefore,
$$
 \exp(\kappa V'(x))=\exp(2\theta G_\mu(x+\i 0))=\exp(2\theta G_\mu(x-\i 0)).
$$
Comparing with \eqref{eq_variational_equation} ($V'(x)$ should be replaced by $\kappa V'(x)$ to
match the notations) we conclude that
$$
 \exp\biggl(\theta G_\mu(x+\i 0)- \theta G_\mu(x-\i 0)\biggr)=1.
$$
But this contradicts $0<\mu(x)<\theta^{-1}$ and \eqref{eq_Stil_inversion}.
\item For $x>b$, if $H(x)=0$, then also $Q(x)=0$ and
\begin{equation}
\label{eq_x51}
 \exp(\kappa V'(x))=\exp(2\theta G_\mu(x)).
\end{equation}
But note that \eqref{eq_x51} holds at $x=b$ as a corollary of \eqref{eq_variational_equation}. And
as $x>b$ grows, the left--hand side of \eqref{eq_x51} also grows due to $V''(x)>0$, but the
right--hand side decays due to the definition of $G_\mu(x)$. Thus, \eqref{eq_x51} can not hold at
$x>b$.
\item For $x<a$ the argument is similar. \qedhere
\end{itemize}
\end{proof}

As a corollary we obtain a central limit theorem for the original measure $\P$.

\begin{corollary} \label{Corollary_CLT_linear_convex} Assume \eqref{eq_potential_behavior}, $V''(x)>0$ for all $x\in\mathbb
R$, and consider the probability measure $\P$ given by \eqref{eq_convex_distribution}. Suppose that
$\kappa$ is so small that the equilibrium measure has one band. Take any $m$ functions
$f_1(z),\dots,f_m(z)$ on $\mathbb R$ which are bounded and extend to holomorphic functions in a
complex neighborhood of $(-D,D)$. Then the $m$ random variables
$$
{\mathcal L}_{f_j}= \sum_{i=1}^{N} \bigl( f_j(\ell_i)-\E_{\P}f_j(\ell_i) \bigr),\quad
(\ell_1,\dots,\ell_N)
 \text{ is } \P\text{--distributed,}
$$
converge (in distribution and in the sense of moments) to centered Gaussian random variables with
covariance
\begin{equation}
\label{eq_linear_covariance_convex}
 \lim_{N\to\infty} \E_{\P} {\mathcal L}_{f_i} {\mathcal L}_{f_j}= \frac{\theta^{-1}}{(2\pi\i)^2}
 \oint_{\gamma} \oint_{\gamma} f_i(u) f_j(v) {\mathcal C}(u,v) du dv,
\end{equation}
where ${\mathcal C}(u,v)$ is given by \eqref{eq_limit_covariance} with $a_{\pm}$ being endpoints of
the band of the equilibrium measure $\mu(x)dx$, and $\gamma$ is a positively oriented contour which
encloses $[a_-,a_+]$.
\end{corollary}
\begin{proof} Due to
Proposition \ref{Proposition_ld_support} and boundedness of functions $f_j$, the joint moments of
$\sum_{i=1}^{N} f_j(\ell_i)$ with respect to $\P$ and with respect to $\widehat \P$ differ by
exponentially small (in $N$)  error. It remains to use Corollary \ref{Corollary_CLT_linear_general}
for $\widehat \P$.
\end{proof}
\begin{remark}
 One way to construct functions $f_j$ satisfying conditions of Corollary
 \ref{Corollary_CLT_linear_convex} is to take any analytic functions on $[-D,D]$ and extend them to
 all of $\mathbb R$ by setting equal to $0$ outside $[-D,D]$.
 The condition that $f_j$ must be bounded on $\mathbb R$ can be weakened. However, we need a growth
 condition on these functions, as we want to be able to replace $\E_{\P}(\sum_{i=1}^N  f_j(\ell_i))^r$ by $\E_{\widehat \P}(\sum_{i=1}^N  f_j(\ell_i))^r$ with negligible
 error. The analyticity assumption can also likely be weakened, but we do not address
 this in the present paper.
\end{remark}
\begin{remark}
 The covariance \eqref{eq_linear_covariance_convex} has the same form as for random matrices and log--gases in the one cut
 regime. It depends only on the restrictions of functions $f_j$ onto the interval $[a_-,a_+]$ and can be rewritten in several other equivalent forms, cf.\ \cite[Theorem
 4.2]{Johansson_CLT}, \cite[Chapter 3]{PS}.
\end{remark}

\subsection{$(z,w)$--measures} \label{Section_zw-measures}
Our last example originates in the asymptotic representation theory of unitary groups $U(N)$, cf.\
\cite{Olsh_U}, \cite{BO_U}, \cite{Olsh_hyper}.

Fix two sequences of non-real parameters $\z(N)$ and $\w(N)$ and define
\begin{multline}
\label{eq_zw_distribution}
 \P(\ell_1,\dots,\ell_N)= \frac{1}{Z_N}
 \prod_{1\le i<j \le N}
 \frac{\Gamma(\ell_j-\ell_i+1)\Gamma(\ell_j-\ell_i+\theta)}{\Gamma(\ell_j-\ell_i)\Gamma(\ell_j-\ell_i+1-\theta)}
 \\ \times \prod_{i=1}^N \frac{1}{\Gamma(\z(N)-\ell_i)\Gamma(\bar \z(N)-\ell_i)\Gamma(\w(N)+\ell_i)\Gamma(\bar \w(N)+\ell_i)}
\end{multline}
on $N$-tuples $\ell_1<\ell_2<\dots<\ell_N$ such that
$$
 \ell_i=\lambda_i+ \theta i,\quad \lambda_1\le\lambda_2\le \dots\le \lambda_N,\quad
 \lambda_i\in\mathbb Z.
$$
Here $\bar \z(N)$ and $\bar \w(N)$ are complex conjugates of $\z(N)$ and $\w(N)$, respectively.

If $\theta=1$, $\z(N)=z+N+1$, $\w(N)=w$ with $\Re(z+w)>-1/2$, and
$\{\lambda_i\}_{i=1}^N$ are identified with highest weights of irreducible
representations of $U(N)$, then \eqref{eq_zw_distribution} describes the
decomposition of the character of the ``generalized bi-regular'' representation of
the infinite-dimensional unitary group $U(\infty)$, see \cite{Olsh_U}.

With the notation
$$
 w(x;N)=\frac{1}{\Gamma(\z(N)-x)\Gamma(\bar \z(N)-x)\Gamma(\w(N)+x)\Gamma(\bar \w(N)+x)},
$$
we have
\begin{multline*}
 \frac{w(x;N)}{w(x-1;N)}=\frac{(x-\z(N))(x-\bar \z(N))}{(x+\w(N)-1)(x+\bar
 \w(N)-1)}\\=1-\frac{
 \w(N)+\bar
 \w(N)+\z(N)+
 \bar \z(N) -2}{x} + O\left(\frac{1}{x^2}\right).
\end{multline*}
The last formula implies that as $|x|\to\infty$, the weight decays as
\begin{equation}
\label{eq_decay_of_weight}
 w(x;N)= O\left( |x|^{  \w(N)+\bar
 \w(N)+\z(N)+
 \bar \z(N) -2} \right).
\end{equation}
Therefore, the real part of $\w(N)+\bar
 \w(N)+\z(N)+
 \bar \z(N)$ needs to be large in order to guarantee that the measure $\P$ is finite.
 Let us assume that as $N\to\infty$
 \begin{equation}
 \label{eq_zw_scaling}
  \w(N)=\w_\infty \cdot N+O(1),\quad \z(N)=\z_\infty \cdot N +O(1),\quad \Re(\w_\infty+\z_{\infty})>1.
 \end{equation}
 Then $w(x;N)$ decays fast enough so that a condition of the form \eqref{eq_potential_behavior} is
 satisfied. Note that the original representation--theoretic case $\z(N)=z+N+1$, $\w(N)=w$ does not satisfy this assumption,
  and indeed the equilibrium measure in this case is known to be somewhat degenerate, and
 one generally does not expect to see the Gaussian behavior. This case with $\theta=1$ was studied
 in \cite{BO_U}.

 \begin{proposition} \label{Proposition_ld_support_zw}
  Assume \eqref{eq_zw_scaling}. Then with exponentially high probability
  the measure \eqref{eq_zw_distribution} is supported on configurations in a linearly growing interval, i.e.\ there exist
  constants $C,D>0$ such that
 $$
  \P\left(-D\le \frac{\ell_1}{N} \le \frac{\ell_N}{N} \le D\right)>1- \frac{1}{C} \cdot \exp(-N C),\qquad N=1,2,\dots.
 $$
\end{proposition}
\begin{proof} This is a particular case of Theorem \ref{Theorem_LDP_support}.
\end{proof}

Proposition \ref{Proposition_ld_support_zw} implies that we can use for $\P$ the
techniques developed in Sections \ref{Section_setup}--\ref{Section_second_order}.
Let us present the functions $R_\mu$ and $Q_\mu$, as they can be found explicitly in
this case. Indeed, we have
$$
 R_\mu(\xi)=(\xi-\z_\infty)(\xi-\bar z_\infty)\exp(\theta G_\mu(\xi))+
 (\xi+\w_\infty)(\xi+\bar w_\infty)\exp(-\theta G_\mu(\xi)).
$$
Since $R_\mu(\xi)$ is analytic and grows as $2\xi ^2$ as $|\xi|\to\infty$,
$R_\mu(\xi)$ is a degree two polynomial, i.e.\
$$
 R_\mu(\xi)=2\xi^2+ A\xi + B.
$$
Let us find the coefficients $A$ and $B$. Expand $G_\mu$ in power series near $\xi=\infty$ as
$$
 G_\mu(\xi)=\frac{1}{\xi}+\frac{p_1}{\xi^2}+O(\xi^{-3}),
$$
where $p_1$ is unknown. Plugging this into the definition of $R_\mu(\xi)$ and expanding up to
$O(\xi^{-1})$ we get
\begin{multline} \label{eq_x52}
 R_\mu(\xi)=
 (\xi-\z_\infty)(\xi-\bar \z_\infty)
 \left(1+\frac{\theta }{\xi}+\frac{\theta p_1+\theta^2}{\xi^2}\right)+
\\
 (\xi+\w_\infty)(\xi+\bar \w_\infty)
 \left(1-\frac{\theta }{\xi}+\frac{-\theta p_1+\theta^2}{\xi^2}\right)+O(\xi^{-1})
\\= 2\xi^2 +(\w_\infty+\bar \w_\infty-\z_\infty-\bar \z_\infty ) \xi +
\z_\infty \bar \z_\infty + \w_\infty \bar \w_\infty- \theta(\z_\infty+\bar \z_\infty
+ \w_\infty+\bar \w_\infty) +2\theta^2.
\end{multline}
Therefore, we can also find $Q_\mu$ through \eqref{eq_quadratic_on_Q}:
\begin{equation} \label{eq_x53}
 Q_\mu(\xi)= \sqrt{ \left( R_\mu(\xi)\right)^2- 4 (\xi-\z_\infty)(\xi-\bar
 \z_\infty)(\xi+\w_\infty)(\xi+\bar\w_\infty)}= c \sqrt{(\xi -a_- )(\xi-a_+)}
\end{equation}
 with explicit $a_\pm$ and $c$ which are found by plugging \eqref{eq_x52} in
 \eqref{eq_quadratic_on_Q}
 (the resulting formulas are somewhat complicated and we omit them). This is
 precisely the form we need for Assumption \ref{Assumption_simple}.

Now repeating the argument of Lemma \ref{Lemma_applies_to_convex}, we see that the
results of Theorem \ref{Theorem_main_general}, Theorem \ref{Theorem_CLT},
 Corollary \ref{Corollary_CLT_linear_general} are valid for measures $\P$
 conditioned on the event of Proposition \ref{Proposition_ld_support_zw}. Thus, as
 in Corollary \ref{Corollary_CLT_linear_convex}, we arrive at the following result.

\begin{corollary} \label{Corollary_CLT_linear_convex_2} Assume \eqref{eq_zw_scaling}
and consider the probability measure $\P$ given by \eqref{eq_zw_distribution}. Take $m\ge 1$
functions $f_1(z),\dots,f_m(z)$ on $\mathbb R$ that are bounded and extend to holomorphic functions
in a complex neighborhood of $(-D,D)$, where $D$ is given by Proposition
\ref{Proposition_ld_support_zw}. Then $m$ random variables
$$
{\mathcal L}_{f_j}= \sum_{i=1}^{N} \bigl( f_j(\ell_i)-\E_{\P}f_j(\ell_i)
\bigr),\quad (\ell_1,\dots,\ell_N)
 \text{ is } \P\text{--distributed,}
$$
converge (in distribution and in the sense of moments) to centered Gaussian random
variables with covariance
\begin{equation}
\label{eq_linear_covariance_convex_2}
 \lim_{N\to\infty} \E_{\P} {\mathcal L}_{f_i} {\mathcal L}_{f_j}= \frac{\theta^{-1}}{(2\pi\i)^2}
 \oint_{\gamma} \oint_{\gamma} f_i(u) f_j(v) {\mathcal C}(u,v) du dv,
\end{equation}
where ${\mathcal C}(u,v)$ is given by \eqref{eq_limit_covariance} with $a_{\pm}$
found from \eqref{eq_x52}, \eqref{eq_x53}, and $\gamma$ is a positively oriented
contour which encloses $[a_-,a_+]$.
\end{corollary}

\section{Exponential bound on the support.}
\label{Section_support} Take a continuous function $V:\mathbb R\to\mathbb R$ and numbers $\eps>0$,
$H>1+\theta^{-1}$ such that
\begin{equation}
\label{eq_potential_behavior_2} \frac{V(x)}{2\theta \ln |x|}> 1+\eps, \text{ when }
|x|>H.
\end{equation}
Set $T=\lfloor \theta \rfloor+1$  and assume that $V(x)$ is increasing for $x>H$, decreasing for $x<-H$ and is
 Lipshitz with a constant $\mathfrak s$ for $|x|<H+T$.

 Consider a probability distribution
\begin{equation}
\label{eq_generic_distribution}
 \P (\ell_1,\dots,\ell_N)= \frac{1}{Z_{N}}
 \prod_{1\le i<j \le N}
 \frac{\Gamma(\ell_j-\ell_i+1)\Gamma(\ell_j-\ell_i+\theta)}{\Gamma(\ell_j-\ell_i)\Gamma(\ell_j-\ell_i+1-\theta)}
 \prod_{i=1}^N \exp\left( - N  \cdot V\left(\frac{\ell_i}{N}\right)\right)
\end{equation}
on $N$-tuples $\ell_1<\ell_2<\dots<\ell_N$ such that
\begin{equation}\label{eq_lattice_1}
 \ell_i=\lambda_i+ \theta i,\quad \lambda_1\le\lambda_2\le \dots\le \lambda_N,\quad
 \lambda_i\in\mathbb Z.
\end{equation}
 The aim of this section is to prove the following statement describing the tails
of $\P$.

\begin{theorem}
\label{Theorem_LDP_support}
 There exist two constants $C=C(\theta, \eps,H, \mathfrak s )$ and $D=D(\theta, \eps,H, \mathfrak s )$ that depend only on $\theta, \eps,H, \mathfrak s$, and
 such that
 $$
  \mathbb{P}_{N,N}
  \left(-D\le \frac{\ell_1}{N} \le \frac{\ell_N}{N} \le D\right)>1- \frac{1}{C} \cdot \exp(-N C),\qquad N=1,2,\dots.
 $$
\end{theorem}
 The proof of Theorem \ref{Theorem_LDP_support} borrows ideas from similar proofs in
 \cite{Johansson_shape}, \cite[Section 2.7]{AGZ}, \cite{Feral}, but additional care is required because of the shifts by $\theta$ in the definition of $\ell_i$.  We present the proof as a series of lemmas.

It is useful to consider several modifications of the measure $\P$, which we now introduce. The probability measure ${\mathbb P}_{N,+}$ is defined
on the same space of $\ell$s by the formula
\begin{equation}
\label{eq_generic_distribution_2}
 {\mathbb P}_{N,+} (\ell_1,\dots,\ell_N)= \frac{1}{Z_{N,+}}
 \prod_{1\le i<j \le N}
 \frac{\Gamma(\ell_j-\ell_i+1)\Gamma(\ell_j-\ell_i+\theta)}{\Gamma(\ell_j-\ell_i)\Gamma(\ell_j-\ell_i+1-\theta)}
 \prod_{i=1}^N \exp\left( - N  \cdot V\left(\frac{\ell_i}{N+1}\right)\right).
\end{equation}
The probability measure ${\mathbb P}_{N,++}$ is defined
on the same space of $\ell$s by the formula
\begin{multline}
\label{eq_generic_distribution_3}
 {\mathbb P}_{N,++} (\ell_1,\dots,\ell_N)\\= \frac{1}{Z_{N,++}}
 \prod_{1\le i<j \le N}
 \frac{\Gamma(\ell_j-\ell_i+1)\Gamma(\ell_j-\ell_i+\theta)}{\Gamma(\ell_j-\ell_i)\Gamma(\ell_j-\ell_i+1-\theta)}
 \prod_{i=1}^N \exp\left( - (N+1)  \cdot V\left(\frac{\ell_i}{N+1}\right)\right).
\end{multline}
We also define shifted measures $\mathbb{P}_{N}^{(k)}$, $\mathbb{P}_{N,+}^{(k)}$, $\mathbb{P}_{N,++}^{(k)}$, $k=1,\dots,N+1$,
which are given by the same formulas \eqref{eq_generic_distribution}, \eqref{eq_generic_distribution_2}, \eqref{eq_generic_distribution_3} as
$\mathbb{P}_{N}^{(k)}$, $\mathbb{P}_{N,+}^{(k)}$, $\mathbb{P}_{N,++}^{(k)}$, respectively,
 but with $\ell_i$ confined to a different lattice, namely
\begin{equation} \label{eq_lattice_shifted}
 \ell_i=\begin{cases} \lambda_i+ \theta i,& i<k,\\
                      \lambda_i+ \theta (i+1),& i\ge k, \end{cases}
 \quad \quad \lambda_1\le\lambda_2\le \dots\le \lambda_N,\quad
 \lambda_i\in\mathbb Z.
\end{equation}
We also let $Z_{N}^{(k)}$, $Z_{N,+}^{(k)}$, $Z_{N,++}^{(k)}$ to be the normalizing constant for the corresponding measures. Further denote
 $$
 \M(\ell_1,\dots,\ell_N)=Z_{N} \P(\ell_1,\dots,\ell_N),
$$
and similarly for $\mathbb{P}_{N,+}$, $\mathbb{P}_{N,++}$,
$\mathbb{P}_{N}^{(k)}$, $\mathbb{P}_{N,+}^{(k)}$, $\mathbb{P}_{N,++}^{(k)}$.

All the constants $c_1,c_2,\dots$ in the following statements depend only on $\theta, \eps,H, \mathfrak s$, the exact
values of the constants might change from statement to statement.

\begin{lemma} \label{Lemma_shifts}
   Then there exists $c_1>0$ such that
 for any $1\le k \le N+1$, and any $\ell$ satisfying \eqref{eq_lattice_1}  there exists
 $\ell'$ satisfying \eqref{eq_lattice_shifted}, for which
 \begin{equation} \label{eq_close_coord}
 |\ell_i-\ell'_i|\le T, \text{ for all }
 1\le i \le N
 \end{equation}
 and
 \begin{equation}
  \label{eq_mes_ratio}
  \frac{\M^{(k)}(\ell')}{\M(\ell)}
 \le c_1 \exp(N c_1).
 \end{equation}
 There also exists a (possibly different) $\ell''$ such that
 \begin{equation} \label{eq_close_coord_2}
 |\ell_i-\ell''_i|\le T, \text{ for all }
 1\le i \le N
 \end{equation}
  and
 \begin{equation}
  \label{eq_mes_ratio_2}
  \frac{\M^{(k)}(\ell'')}{\M(\ell)}
 \ge \frac{1}{c_1} \exp(-N c_1).
 \end{equation}
 Similarly, for any  $\ell'$ satisfying \eqref{eq_lattice_shifted}, there exists $\ell$ satisfying \eqref{eq_lattice_1}
 such that \eqref{eq_close_coord} and \eqref{eq_mes_ratio} hold. Further,
 for any  $\ell''$ satisfying \eqref{eq_lattice_shifted}, there exists $\ell$ satisfying \eqref{eq_lattice_1}
 such that \eqref{eq_close_coord_2} and \eqref{eq_mes_ratio_2} hold.
 Finally, the same statements hold for the measures $\mathbb{P}_{N,+}^{(k)}$ and $\mathbb{P}_{N,++}^{(k)}$.
\end{lemma}
\begin{proof} We will only prove the first two statements of the Lemma, as the rest can be proven similarly.

Let $\mathbf x=(x_1,\dots,x_N)\in\mathbb R^N$ and $\mathbf
y=(y_1,\dots,y_N)\in\mathbb R^N$ be such that $x_j-x_i\ge \theta(j-i)$, for $1\le
i<j\le N$ and there exists $m=1,\dots,N$ and $M\in\mathbb R$ such that
$$
 y_i=\begin{cases} x_i, &1\le i \le m,\\ x_i+M, &m<i \le N. \end{cases}
$$

 We claim that there exists $c_2=c_2(M)$ such that
 \begin{multline} \label{eq_cross_terms_bound}
  \frac{1}{c_2} \exp(-Nc_2) \\ \le
  \prod_{1\le i<j \le N} \left(
 \frac{\Gamma(x_j-x_i+1)\Gamma(x_j-x_i+\theta)}{\Gamma(x_j-x_i)\Gamma(x_j-x_i+1-\theta)}
 \cdot\frac{\Gamma(y_j-y_i)\Gamma(y_j-y_i+1-\theta)}{\Gamma(y_j-y_i+1)\Gamma(y_j-y_i+\theta)}\right)
  \le c_2 \exp(N c_2).
 \end{multline}
 Indeed, \eqref{eq_Gamma_expansion} implies that
 $$
  \prod_{i<j} \left(
 \frac{\Gamma(x_j-x_i+1)\Gamma(x_j-x_i+\theta)}{\Gamma(x_j-x_i)\Gamma(x_j-x_i+1-\theta)}
 \cdot\frac{\Gamma(y_j-y_i)\Gamma(y_j-y_i+1-\theta)}{\Gamma(y_j-y_i+1)\Gamma(y_j-y_i+\theta)}\right)
  = \prod_{i\le m < j}\left(1+ O\left(\frac{1}{x_j-x_i}\right)\right)
 $$
 Since $x_j-x_i\ge \theta(j-i)$, the last product is bounded below by $ \frac{1}{c_2}
 \exp(-Nc_2)$ and above by ${c_2} \exp(-Nc_2)$ for some $c_2>0$.

\smallskip
 Take $\ell$ satisfying \eqref{eq_lattice_1}. Suppose that $\ell_i<-H$ for $i=1,\dots,m_1$, $|\ell_i|\le H$ for $i=m_1+1,\dots, m_2$,
 and $\ell_i>H$ for $i=m_2+1,\dots,N$.

 Define $\tilde \ell=(\tilde \ell_1,\dots,\tilde \ell_N)$ through
 $$
  \tilde \ell_i=\begin{cases} \ell_i,& i< k,\\ \ell_i+\theta, & i\ge k.\end{cases}
 $$
 And further define $\ell'=(\ell'_1,\dots,\ell'_N)$ through
 $$
  \ell'_i=\begin{cases} \tilde \ell_i-T,& i\le m_1,\\ \tilde \ell_i, & i>m_1
  \end{cases}
 $$

 We claim that \eqref{eq_mes_ratio} holds.
 Indeed, the ratio of the factors in double product $\prod_{i<j}$
 is bounded by two applications of \eqref{eq_cross_terms_bound}, and it remains to bound
$$
  \prod_{i=1}^N \exp\left( N V\left(\frac{\ell_i}{N}\right) - N V\left(\frac{\ell'_i}{N}\right)\right).
$$
 If $i\le m_1$, then $\ell'_i\le \ell_i$ and the monotonicity of $V(x)$ implies that corresponding factors are less than $1$. If $i>m_2$,
 then $\ell'_i \ge \ell_i$ and again the monotonicity of $V(x)$ implies that corresponding factors are less than $1$. Finally, if $m_1<i\le m_2$,
 then the Lipshitz property of $V(x)$ gives the desired bound.

 \smallskip

 Next, we construct $\ell''$. For $i\le m_1$ we set $\ell''_i:=\tilde \ell_i$. For $i>m_2$ we set $\ell''_i=\tilde \ell_i-T$. Since $H>1+\theta^{-1}$, we can always choose the remaining coordinates $\ell''_i$, $m_1<i\le m_2$ in such a way that $\ell''$ satisfies \eqref{eq_lattice_shifted} and
 \eqref{eq_close_coord}. We claim that \eqref{eq_mes_ratio_2} holds. Indeed, the ratio of the factors in double product $\prod_{i<j}$ is bounded by \eqref{eq_cross_terms_bound}, and it remains to bound
$$
  \prod_{i=1}^N \exp\left( N V\left(\frac{\ell_i}{N}\right) - N V\left(\frac{\ell'_i}{N}\right)\right).
$$
 If $i\le m_1$, then $\ell'_i\ge \ell_i$ and the monotonicity of $V(x)$ implies that corresponding factors are greater than $1$. If $i>m_2$,
 then $\ell'_i \le \ell_i$ and again the monotonicity of $V(x)$ implies that corresponding factors are greater than $1$. Finally, if $m_1<i\le m_2$,
 then the Lipshitz property of $V(x)$ gives the desired bound.
\end{proof}

\begin{lemma} \label{Lemma_partition_1}
There exists $c_2>0$ such that
 \begin{equation} \label{eq_change_of_lat}
 \frac{Z_{N,+}^{(k)}}{Z_{N,+}} \le c_2\exp\bigl( c_2 N \bigr),
\quad 1\le k \le N+1.
 \end{equation}
\end{lemma}
\begin{proof}
 For $\ell'$ satisfying \eqref{eq_lattice_shifted} let $s(\ell)$ denote the corresponding $\ell$ of Lemma \ref{Lemma_shifts}. Then
 \begin{multline*}
  Z_{N,+}^{(k)}=\sum_{\ell'} {\mathbb M}_{N,+}^{(k)} (\ell') \le c_1 \exp(Nc_1) \sum_{\ell'} {\mathbb M}_{N,+}(s(\ell'))\\ \le
  c_1 \exp(Nc_1) (2T+1)^N \sum_{\ell} {\mathbb M}_{N,+}(\ell)=   c_1 \exp(Nc_1) (2T+1)^N Z_{N,+}.\qedhere
 \end{multline*}
\end{proof}

\begin{lemma} \label{Lemma_partition_2}
 There exists $c_3>0$ such that
 $$
  \frac{Z_{N-1,++}}{Z_{N}} \le  c_3 \exp(Nc_3) \cdot N^{-2\theta N} , \quad N=1,2,\dots.
 $$
\end{lemma}
\begin{proof}
 We have
\begin{multline} \label{eq_x60}
 N \frac{Z_N}{Z_{N-1,++}} =N \sum_{\ell} \frac{{\mathbb M}_{N} (\ell)}{Z_{N-1,++}}
 \\= \sum_{k=1}^N  \sum_{\ell^{(k)}} \frac{ {\mathbb M}_{N-1,++}^{(k)}(\ell^{(k)}) } {Z_{N-1,++}}
  \sum_{m=\ell^{(k)}_{k-1}+\theta}^{\ell^{(k)}_k-\theta}
  \prod_{i=1}^{N-1} \frac{\Gamma(|m-\ell^{(k)}_i|+1)\Gamma(|m-\ell^{(k)}_i|+\theta)}{\Gamma(|m-\ell^{(k)}_i|)\Gamma(|m-\ell^{(k)}_i|+1-\theta)}
  \exp\left(-NV\left(\frac{m}{N}\right)\right),
\end{multline}
where $\ell^{(k)}$ varies over \eqref{eq_lattice_shifted} with $N$ replaced by
$N-1$, and we use the notation  $\ell^{(1)}_0=-\infty$, $\ell^{(N)}_N=+\infty$. Note
that by the definitions, the sum
$\sum_{m=\ell^{(k)}_{k-1}+\theta}^{\ell^{(k)}_k-\theta}$ is always non-empty. Take
$\ell$ satisfying \eqref{eq_lattice_1} with $N$ replaced by $N-1$ and let
$s^{(k)}(\ell)$ denote the corresponding $\ell'$ of Lemma \ref{Lemma_shifts}. Then
\eqref{eq_x60}  implies
\begin{multline} \label{eq_x61}
 N \frac{Z_{N}}{Z_{N-1,++}} \ge (2T+1)^{-N}
  \sum_{k=1}^N \sum_{\ell} \frac{ {\mathbb M}_{N-1,++}^{(k)}(s^{(k)}(\ell)) } {Z_{N-1,++}}
    \\ \times \sum_{m={s^{(k)}(\ell)}_{k-1}+\theta}^{s^{(k)}(\ell)_k-\theta}
  \prod_{i=1}^{N-1} \frac{\Gamma(|m-s^{(k)}(\ell)_i|+1)\Gamma(|m-s^{(k)}(\ell)_i|+\theta)}{\Gamma(|m-s^{(k)}(\ell)_i|)\Gamma(|m-s^{(k)}(\ell)_i|+1-\theta)}
  \exp\left(-NV\left(\frac{m}{N}\right)\right),
\end{multline}
 where  $\ell$ varies over \eqref{eq_lattice_1}.

 Using Stirling's formula, its corollary \eqref{eq_Gamma_expansion}, and definition of the state space \eqref{eq_lattice_shifted},  we see that the product of Gamma functions in the above formula is bounded from below by
 $$
   \frac{\exp(-Nc_4)}{c_4} \cdot N^{2\theta N}
 $$
 with a constant $c_4>0$.
 Let $v$ be the maximum of $V(x)$ over $[-2T,2T]$. Then \eqref{eq_x61} and Lemma \ref{Lemma_shifts} imply that for $c_5>0$
$$
 N \frac{Z_{N}}{Z_{N-1,++}} \ge N^{2\theta N} \, \frac{\exp(-c_5 N)}{c_5}
  \sum_{k=1}^N \sum_{\ell} {\mathbb P}_{N-1,++}(\ell)
 \sum_{m={s^{(k)}(\ell)}_{k-1}+\theta}^{s^{(k)}(\ell)_k-\theta}
 I_{-2T\le m \le 2T} \exp(-Nv)
$$
 Note that for every $\ell$ satisfying \eqref{eq_lattice_1}
  there exists at least one $k$, such that for at least one
   $m\in s^{(k)}(\ell)_{k-1}+\theta, s^{(k)}(\ell)_{k-1}+\theta+1,\dots, s^{(k)}(\ell)_k-\theta$,
   we have $|m|\le 2T$ (here again we use the notation  $s^{(1)}(\ell)_0=-\infty$, $s^{(N)}(\ell)_N=+\infty$). Therefore,
$$
 N \frac{Z_{N}}{Z_{N-1,++}} \ge  N^{2\theta N} \, \frac{\exp(-c_5 N)}{c_5} \sum_{\ell}   {\mathbb P}_{N-1,++}(\ell) \exp(-Nv)=
N^{2\theta N} \,  \frac{\exp(-c_5 N)}{c_5} \exp(-Nv). \qedhere
$$
\end{proof}

\begin{lemma} \label{Lemma_partition_3}
 There exists $c_4>0$ such that
 \begin{equation}
 \label{eq_part_ratio_1}
  \frac{Z_{N,+}}{Z_{N,++}} \le c_4 \exp(c_4N), \quad N=1,2,\dots.
 \end{equation}
\end{lemma}
\begin{proof} Define the random probability measure $\nu_N$ through
$$
 \nu_N=\frac{1}{N}\sum_{i=1}^{N} \delta_{\ell_i/(N+1)},\quad (\ell_1,\dots,\ell_N)\text{ is } {\mathbb P}_{N,++} \text{-distributed}
$$
Then
\begin{equation}
\label{eq_part_ratio_2}
\frac{Z_{N,+}}{Z_{N,++}}=\sum_{\ell} \frac{{\mathbb M}_{N,+}(\ell)}{Z_{N,++}}=\E_{{\mathbb P}_{N,++}} \left[\exp\left(N \int V(x) \nu_N(dx)\right) \right].
\end{equation}
In order to bound \eqref{eq_part_ratio_2} we start with a lower bound
$$Z_{N,++}\ge {\mathbb M}_{N,++} (\theta, 2\theta,\dots, N\theta)\ge \exp\bigl(-c_5 N^2\bigr)\cdot N^{2\theta N^2}$$
for some $c_5>0$. On the other hand, using \eqref{eq_Gamma_expansion} we obtain
\begin{multline*}
\ln {\mathbb M}_{N,++}(\ell_1,\dots,\ell_N)\\ \le \exp\Biggl(2\theta N^2\ln N- N^2 \iint_{x\neq y} \left(\frac{V(x)+V(y)}{2}- \theta
\ln |x-y|\right)\nu_N (dx) \nu_N(dy)
 +c_6 N^2\Biggr),
\end{multline*}
for some $c_6>0$.
 Assumption \eqref{eq_potential_behavior_2} and inequality $\ln|x-y|\le \ln (|x|+1)+\ln(|y|+1)$
imply that  that there exist $c_7>0$   such that for all $x\ne y$,
\begin{multline*}
V(x)+V(y)- 2\theta \ln |x-y|\\ \ge
\frac{\eps}{1+\eps}( V(x)+V(y)) + \left( \frac{V(x)}{1+\eps}-2\theta\ln(|x|+1)\right)
+ \left( \frac{V(y)}{1+\eps}-2\theta\ln(|y|+1)\right)
\\ \ge
 \frac{\eps}{1+\eps}( V(x)+V(y))- c_7.\end{multline*}
Therefore, for each $L>0$ we have
\begin{multline*}\mathbb E_{\mathbb P_{N,++}}\left[  \int  V(x) \mu_N(dx) >L\right]
\\ \le \exp\left(c_5 N^2+ c_6 N^2 + \frac{c_7}{2} N^2 - \frac{\eps}{2(1+\eps)} N^2 L \right) \prod_{i=1}^N \left( \sum_{m\in \mathbb
Z+i\theta} \exp\left(-\frac{\eps}{2(1+\eps)} N  V(m/N)\right)\right).
\end{multline*}
Hence, there exists $c_8>0$, and $L_0$ such that for all $L>L_0$,
$$\mathbb E_{\mathbb P_{N,++}}\left[  \int  V(x) \mu_N(dx) >L\right]
\le \exp\left(- c_8 L N^2 \right).
$$
Together with \eqref{eq_part_ratio_2} this implies \eqref{eq_part_ratio_1}.
\end{proof}

\begin{proof}[Proof of Theorem \ref{Theorem_LDP_support}] We have
\begin{multline} \label{eq_x61_2}
 {\mathbb P}_{N} (\ell_N> DN)
  = \frac{Z_{N-1,+}}{Z_{N}}
 \sum_{ \tilde \ell} {\mathbb P}_{N-1,+} ( \tilde \ell)
 \sum_{m>\max(\tilde \ell_{N-1},\, DN) }
  \exp\left(-NV\left(\frac{m}{N}\right)\right)
 \\ \times \prod_{i=1}^{N-1} \left[ \frac{\Gamma(m-\tilde \ell_i+1)\Gamma(m-\tilde \ell_i+\theta)}{\Gamma(m-\tilde \ell_i)\Gamma(m-\tilde \ell_i+1-\theta)}
 \exp\left(-V \left(\frac{\tilde \ell_i} {N} \right)\right) \right]
 ,
\end{multline}
where $\tilde \ell$ ranges over \eqref{eq_lattice_1} with $N$ replaced by $N-1$.
The combination of Lemmas \ref{Lemma_partition_1} and \ref{Lemma_partition_3} implies that for $c_5>0$ we have a bound
$$
\frac{Z_{N-1,+}}{Z_{N}} \le c_5 \exp(Nc_5) \cdot N^{-2\theta N}.
$$
For the product of gamma-functions in \eqref{eq_x61_2} we use Stirling's formula and its corollary
\eqref{eq_Gamma_expansion}, which yields
\begin{equation*}
\begin{split}
\prod_{i=1}^{N-1} \left[ \frac{\Gamma(m-\tilde \ell_i+1)\Gamma(m-\tilde
\ell_i+\theta)}{\Gamma(m-\tilde \ell_i)\Gamma(m-\tilde \ell_i+1-\theta)}\right] &\le c_6 \exp(c_6
N) N^{2\theta N} \prod_{i=1}^{N-1}\left|\frac{m}{N}-\frac{\tilde \ell_i}{N}\right|^{2\theta}
\\ &\le c_6 \exp(c_6 N) N^{2\theta N}  2^{(2\theta+1)(N-1)}\prod_{i=1}^{N-1}\left[  \biggl|\frac{m}{N}\biggr|^{2\theta}+\biggl|\frac{\tilde \ell_i}{N}\biggr|^{2\theta} \right].
\end{split}
\end{equation*}
It follows that for $c_7>0$,
\begin{multline} \label{eq_x62}
 {\mathbb P}_{N} (\ell_N> DN) \le
 c_7 \exp(c_7N)
 \sum_{\tilde \ell} {\mathbb P}_{N-1,+} (\tilde \ell)
 \sum_{m>\max(\tilde \ell_{N-1},DN)}
  \exp\left(-NV\left(\frac{m}{N}\right)\right)
 \\ \times \prod_{i=1}^{N-1}\left[\left(  \biggl|\frac{m}{N}\biggr|^{2\theta}+\biggl|\frac{\tilde \ell_i}{N}\biggr|^{2\theta} \right)
 \exp\left(-V \left(\frac{\tilde \ell_i} {N} \right)\right) \right].
\end{multline}
Further, take $\delta>(2\theta)^{-1}$, and observe that when $D>H+1$,
\begin{equation*}
\begin{split}
 \sum_{m>\max(\tilde \ell_{N-1},DN)}
   \exp\left(-\delta V\left(\frac{m}{N}\right)\right)&\le
 \sum_{m=\lfloor DN \rfloor}^{\infty} \exp(-2 \delta \theta \ln(m/N))\
 =N^{2\delta\theta}\sum_{m=\lfloor DN \rfloor}^{\infty} \frac{1}{m^{2\delta\theta}}\\ &\le
 2 N D^{1-2\delta\theta}.
\end{split}
\end{equation*}
Also due to \eqref{eq_potential_behavior_2}, when $m>DN>(H+1)N$, we have
\begin{multline*}
 \left(  \biggl|\frac{m}{N}\biggr|^{2\theta}+\biggl|\frac{\tilde \ell_i}{N}\biggr|^{2\theta} \right)
 \exp\left(-V \biggl(\frac{\tilde \ell_i} {N} \biggr) -\frac{1+\eps/2}{1+\eps} V\biggl(\frac{m}{N}\biggr)\right)
 \\ \le{\rm const}\cdot \biggl|\frac{m}{N}\biggr|^{2\theta}   \exp\left( - \frac{1+\eps/2}{1+\eps} V\biggl(\frac{m}{N}\biggr)\right)
 \le {\rm const} \biggl|\frac{N}{m}\biggr|^{\theta\eps } \le {\rm const} \cdot D^{-\theta\eps}
\end{multline*}
Therefore, writing $ N(V(\frac{m}{N}))$ in \eqref{eq_x62} as
$$
 N\left(V\left(\frac{m}{N}\right)\right)= \frac{1+\eps/2}{1+\eps} (N-1) V\left(\frac{m}{N}\right) + \delta V\left(\frac{m}{N}\right) + \left(N-\frac{1+\eps/2}{1+\eps} (N-1) -\delta\right)V\left(\frac{m}{N}\right),
$$
and noticing that the last term is positive when $m>DN$ and $N$ is large, we conclude that for
  some $c_8>0$, all $D>H+1$ and all $N>N_0$
$$
{\mathbb P}_{N} (\ell_N> DN) \le c_8 \exp(c_8 N) \cdot N^2 \cdot D^{-\theta\eps N }.
$$
Choosing $D$ large enough, we obtain the desired exponential estimate for $\P(\ell_N>DN)$. The
estimate for $\P(\ell_1<-DN)$ is obtained in the same way --- the only difference is that we now
need to bound $\frac{Z_{N-1,+}^{(1)}}{Z_{N}}$ instead of $\frac{Z_{N-1,+}}{Z_{N}}$, but for that we
use Lemma \ref{Lemma_partition_2}.
\end{proof}

\begin{remark}
 It is very plausible that one can similarly establish an analogue of Theorem \ref{Theorem_LDP_support} for more
  general models in the framework of Section \ref{Section_setup} with $a_1(N)=-\infty$ and $b_k(N)=+\infty$.  The only
  necessary modification in the above proofs is in Lemma \ref{Lemma_shifts}, where we should take into account that the
  Lipshitz property of $V(x)$ might fail near the endpoints $a_i(N)$, $b_i(N)$, cf.\ \eqref{eq_derivative_bound}. We will
   not address here the exact conditions on $V(x)$ under which an analogue of Lemma \ref{Lemma_shifts} holds for the models in framework of Section \ref{Section_setup}.
\end{remark}


\begin{thebibliography}{BKMM}

\bibitem[AM]{AM} J.~Ambj{\o}rn and Yu.~Makeenko,
Properties of loop equations for the hermitian matrix model and for two-dimensional gravity, Mod. Phys. Lett. A 5 (1990), 1753--1763.

\bibitem[AGZ]{AGZ} G.~Anderson, A.~Guionnet, O.~Zeitouni, Introduction to Random Matrices,  Cambridge Studies in Advanced Mathematics, 2009.

\bibitem[BBDT]{BBDT} J.~Baik, A.~Borodin, P.~Deift, T.~Suidan, A Model for the Bus System in Cuernevaca (Mexico),
Journal of Physics A: Mathematical and General, 39, no.\ 28 (2006), 8965,
 arXiv:math/0510414

\bibitem[BKMM]{BKMM} J.~Baik, T.~Kriecherbauer, K.~T.-R.~McLaughlin, P.~D.~Miller,  Uniform Asymptotics for Polynomials Orthogonal With Respect to
a General Class of Discrete Weights and Universality Results for Associated Ensembles.
arXiv:math/0310278

\bibitem[BFG]{BFG} F.~Bekerman, A.~Figalli, A.~Guionnet, Transport maps for Beta-matrix models and
Universality.  To appear in Communications in Mathematical Physics. arXiv:1311.2315

\bibitem[BeGu]{BeGu} G.~Ben Arous and A.~Guionnet, Large deviations for
Wigner's law and Voiculescu's Non--Commutative Entropy. Probability Theory and Related Fields, 108,
no.\ 4 (1997), 517--542.

\bibitem[BDE]{BDE} G.~Bonnet, F.~David, B.~Eynard,
Breakdown of universality in multi-cut matrix models, Journal of Physics A:
Mathematical and General, 33, no.\ 38 (2000), 6739, arXiv:cond-mat/0003324.

\bibitem[B1]{Bor} A.~Borodin, Schur dynamics of the Schur processes, Advances in Mathematics,
228, no.\ 4 (2011), 2268--2291. arXiv:1001.3442

\bibitem[B2]{B_GFF} A.~Borodin, CLT for spectra of submatrices of Wigner random matrices,
Moscow Mathematical Journal, 14, no.\ 1 (2014), 29--38,
  arXiv:1010.0898.

\bibitem[BF]{BF} A.~Borodin, P.~Ferrari, Anisotropic growth of random
surfaces in 2 + 1 dimensions, Communications in Mathematical Physics, 325, no.\ 2 (2014),603--684,
arXiv:0804.3035.


\bibitem[BoGo]{BG} A.~Borodin, V.~Gorin, Shuffling Algorithm for Boxed Plane
Partitions,  Advances in Mathematics, 220 (2009), no. 6, 1739--1770,  arXiv:0804.3071

\bibitem[BoGo2]{BG_GFF} A.~Borodin, V.~Gorin, General beta Jacobi corners process and the Gaussian Free Field,
 Communications on Pure and Applied Mathematics, to appear. arXiv:1305.3627.


\bibitem[BO1]{BO_U} A.~Borodin. G.~Olshanski, Harmonic analysis on the infinite-dimensional unitary group and determinantal point processes,
Annals of Mathematics, 161, no.\ 3 (2005), 1319--1422,  arXiv:math/0109194


\bibitem[BO2]{BO1} A.~Borodin. G.~Olshanski, Asymptotics of Plancherel-type random partitions, Journal of Algebra 313 (2007), no.\ 1,
40-60, arXiv:math/0610240

\bibitem[BP]{BP_lectures} A.~Borodin, L.~Petrov, Integrable probability: From representation theory to Macdonald processes,
Probability Surveys 11 (2014), 1--58,  arXiv:1310.8007.




\bibitem[BoGu1]{BoG1} G. ~Borot and A.~Guionnet, Asymptotic expansion of beta matrix models in the one-cut
regime, Communications in Mathematical Physics, 317, no.\ 2 (2013), 447-483, arXiv:1107.1167


\bibitem[BoGu2]{BoG2} G. ~Borot and A.~Guionnet, Asymptotic expansion of beta matrix models in the multi-cut
regime,  arXiv:1303.1045


\bibitem[BEY]{BEY} P.~Bourgade, L.~Erdos, H.-T.~Yau, Edge Universality of Beta Ensembles, Communications in Mathematical Physics, to appear. arXiv:1306.5728

\bibitem[BPS]{BPS}  A.~Boutet de Monvel,  L.~Pastur,  and M.~Shcherbina,
On The Statistical Mechanics Approach in the Random Matrix Theory: Integrated Density of States,
Journal of Statistical Physics, 79, no.\ 3/4 (1995), 585--611.

\bibitem[BD]{BD} J.~Breuer, M.~Duits, Central Limit Theorems for Biorthogonal Ensembles and Asymptotics of Recurrence
Coefficients,  arXiv:1309.6224.

\bibitem[BIPZ]{BIPZ}E.~Br{\'e}zin, C.~ Itzykson, G.~ Parisi,  J.B~ Zuber, Planar diagrams, Communications in Mathematical Physics, 59,
      (1978), 35--5.

\bibitem[BuGo]{BuGo} A.~Bufetov, V.~Gorin,
Representations of classical Lie groups and quantized free convolution,  Geometric and Functional
Analysis (GAFA),  25, no.\ 3 (2015), 763--814, arXiv:1311.5780


\bibitem[BuGo2]{BuGo2} A.~Bufetov, V.~Gorin, Fluctuations of particle systems determined by Schur generating
functions.  arXiv:1604.01110.

\bibitem[Cha]{Chat} S.~Chatterjee, Rigorous solution of strongly coupled $SO(N)$ lattice gauge theory in the large $N$
limit, arXiv:1502.07719.

\bibitem[CE]{CE} L.O.~Chekhov and B.~Eynard, Matrix eigenvalue model: Feynman graph technique for all genera,
JHEP (2006), no. 0612:026. arXiv: math-ph/0604014.

\bibitem[CJY]{CJY} S.~Chhita, K.~Johansson, B.~Young, Asymptotic domino statistics in the Aztec
diamond, Annals of Applied Probability, 25, no.\ 3 (2015), 1232--1278, arXiv:1212.5414.

\bibitem[CLP]{CLP} H.~Cohn, Larsen, J.~Propp, The shape of a typical boxed plane partition, New York Journal of Mathematics 4 (1998), 137--165,
arXiv:math/9801059

\bibitem[CGM]{CGM} B.~Collins, A.~ Guionnet, E.~ Maurel-Segala,  Asymptotics of unitary and orthogonal matrix integrals, Advances in Mathematics, 222, (2009),
    172--215. arXiv:math/0608193

\bibitem[DF]{DFer} M.~Dolega, V.~Feray, Gaussian fluctuations of Young diagrams and structure constants of Jack
characters, arXiv:1402.4615.

\bibitem[DS]{DS} P.~D.~Dragnev,  E.~B.~Saff, Constrained energy
problems with applications to orthogonal polynomials of a discrete variable. Journal  d'Analyse Mathematique,  72 (1997), 223--259.

\bibitem[Du]{Dubrovin} B.~A.~Dubrovin,
Theta functions and non-linear equations, Russian Mathematical Surveys, 36, no.\ 2
(1981), 11--92.

\bibitem[Ey1]{Ey} B.~Eynard, All genus correlation functions for the hermitian 1-matrix
model, JHEP (2004), no. 0411:031. arXiv: hep-th/0407261.

 \bibitem[Ey2]{Ey2}  B.~Eynard, All order asymptotic expansion of large partitions, J. Stat. Mech. Theory Exp,
   ({2008}), P07023, 34.

\bibitem[Ey3]{Ey3}  B.~Eynard, A matrix model for plane partitions, J. Stat. Mech. Theory Exp.,
  (2009), P10011, 72.

\bibitem[EO]{EO} B.~Eynard,  N.~ Orantin, Topological recursion in enumerative geometry and random
              matrices,  J. Phys. A, 42,
      (2009), 293001, 117.

\bibitem[Fe]{Feral} D.Feral, On large deviations for the spectral
measure of discrete Coulomb gas. Seminaire de Probabilites XLI Lecture Notes in Mathematics, Vol.\
1934, 2008, 19--49

\bibitem[Fo]{For}  P.~J.~Forrester, Log-gases and Random Matrices. Princeton University Press, 2010.

\bibitem[G]{Gor} V.~Gorin, Non-intersecting paths and Hahn orthogonal ensemble, Functional Analysis and its Applications, 42 (2008), no.\ 3, 180--197, arXiv:0708.2349.

\bibitem[GS]{GS} V.~Gorin, M.~Shkolnikov, Multilevel Dyson Brownian motions via Jack polynomials, Probability Theory and Related Fields,to appear.  arXiv:1401.5595

\bibitem[GN]{GN} A.~ Guionnet, J.~Novak, Asymptotics of unitary multimatrix models: {T}he
              {S}chwinger--{D}yson lattice and topological recursion, Journal of Functional Analisis,  {268},  (2015), 2851--2905.

\bibitem[HO]{Hora}  A.~Hora, N.~Obata, Quantum Probability and Spectral Analysis of Graphs, Theoretical and Mathematical
Physics, Springer, 2007.

\bibitem[J1]{Johansson_CLT} K.~Johansson, On Fluctuations of eigenvalues of random Hermitian matrices,  Duke Mathematical Journal,  91, no.\ 1
 1998.

\bibitem[J2]{Johansson_Annals} K.~Johansson, Discrete orthogonal polynomial ensembles and the Plancherel measure,
Annals of Mathematics (2) 153 (2001), no.\ 2, 259--296,  math.CO/9906120

\bibitem[J3]{Johansson_shape} K.~Johansson, Shape fluctuations and
random matrices. Communications in Mathematical Physics 209, 437-476 (2000),  arXiv:math/9903134

\bibitem[J4]{Johansson_paths} K.~Johansson, Non-intersecting Paths, Random Tilings and Random
Matrices, Probability Theory and Related Fields (2002), 225--280,  arXiv:math/0011250

\bibitem[JN]{JN} K.~Johansson, E.~Nordenstam, Eigenvalues of GUE minors, Electronic Journal of Probability, 11  (2006), paper 50,  1342--1371, arXiv:math/0606760


\bibitem[K]{Kenyon} R.~Kenyon, Height fluctuations in the honeycomb dimer model, Communications in Mathematical
Physics, 281, no.\ 3 (2008), 675--709. arXiv:math-ph/0405052

\bibitem[KO]{KO} R.~Kenyon, A.~Okounkov, Limit shapes and the complex burgers equation, Acta Mathematica, 199 (2007), 263--302, arXiv:math-ph/0507007

\bibitem[KOR]{KOR} W.~Konig, N.~O'Connel, S.~Roch, Non-colliding random walks, tandem queues, and discrete orthogonal polynomial ensembles,
Electronic Journal of Probability, 7 (2002) paper 1, 1--24.

\bibitem[KS]{KS}
T.~Kriecherbauer and M.~Shcherbina, Fluctuations of eigenvalues of matrix
  models and their applications,
  arXiv: 1003.6121.



\bibitem[Ma]{Mac} I.~G.~Macdonald, \textit{Symmetric functions and Hall polynomials}, Second Edition. Oxford University Press,
1999.


\bibitem[MMS]{MMS} M.~Maida, E.~Maurel--Segala,
Free transport-entropy inequalities for non-convex potentials and application to concentration for
random matrices, Probability Theory and Related Fields, 159, no.\ 1--2 (2014), 329--356.
arXiv:1204.320

\bibitem[Me]{Meh}  M.~L.~Mehta, Random Matrices (3rd ed.),  Amsterdam: Elsevier/Academic Press, 2004.

\bibitem[Mo]{Moll} A.~Moll, in preparation

\bibitem[Mi]{Mi} A.~A.~Migdal,  Loop Equations and $1/N$ Expansion. Physical Reports 102 (1983), 199--290.

\bibitem[NS]{Nek_PS} N.~Nekrasov, V.~Pestun, S.~Shatashvili, Quantum geometry and quiver gauge
theories, High Energy Physics - Theory, 2013, 1-83,  arXiv:1312.6689.


\bibitem[NP]{Nek_Pes} N.~Nekrasov, V.~Pestun,  Seiberg-Witten geometry of four dimensional $N=2$ quiver gauge
theories. arXiv:1211.2240


\bibitem[N]{Nekrasov} N.~Nekrasov, Non-perturbative Dyson--Schwinger equations and  BPS/CFT correspondence, in preparation.

\bibitem[O]{Olsh_U} G.~Olshanski, The problem of
harmonic analysis on the infinite-dimensional unitary group, Journal of Functional Analysis,  205 (2003), no. 2, 464--524, arXiv:math/0109193

\bibitem[O2]{Olsh_hyper} G.~Olshanksi, Probability Measures on Dual Objects to Compact Symmetric Spaces and Hypergeometric
Identities, Functional Analysis and Its Applications, 37, no.\ 4 (20013), 281--301.


\bibitem[Ox]{ABF} The Oxford Handbook of Random Matrix
Theory, edited by G.~Akemann, J.~Baik, P.~Di~Francesco,  Oxford University Press, 2011.


\bibitem[PS]{PS} L.~Pastur, M.~Shcherbina, Eigenvalue Distribution of Large Random Matrices. AMS, 2011

\bibitem[P1]{Petrov_Airy} L.~Petrov, Asymptotics of Random Lozenge Tilings via Gelfand-Tsetlin Schemes,
Probability Theory and Related Fields, 160, no.\ 3-4 (2014), 429--487,
arXiv:1202.3901.

\bibitem[P2]{Petrov_GFF} L.~Petrov, Asymptotics of uniformly random lozenge tilings of polygons. Gaussian free field, Annals of Probability, 43, no. 1 (2015), 1--43,  arXiv:1206.5123


\bibitem[ST]{ST} E.~B.~Saff, V.~Totik, Logarithmic potentials with external fields. Springer, 1997.

\bibitem[Shch]{Shch} M.~Shcherbina, Fluctuations of linear eigenvalue statistics of $\beta$ matrix models
in the multi-cut regime, Journal of Statistical Physics,  151, no.\ 6 (2013), 1004--1034,
arXiv:1205.7062


\bibitem[Wi]{Wigner} E.~P.~Wigner, On the distribution of the roots of certain symmetric matrices,
Annals of Mathematics, 67 (1958), 325--327.

\end{thebibliography}
\end{document}